\documentclass[11pt,reqno]{amsart}
\usepackage[T1]{fontenc}
\usepackage[margin=1.15in]{geometry}
\usepackage{amsmath,amssymb,amsthm,mathtools}
\usepackage{enumitem}
\usepackage{booktabs}
\usepackage[space]{cite}
\usepackage[colorlinks=true,linkcolor=blue,citecolor=blue,hypertexnames=false]{hyperref}

\newtheorem{theorem}{Theorem}[section]
\newtheorem{proposition}[theorem]{Proposition}
\newtheorem{lemma}[theorem]{Lemma}
\newtheorem{corollary}[theorem]{Corollary}

\newtheorem*{theoremA}{Theorem A}
\newtheorem*{theoremB}{Theorem B}
\newtheorem*{theoremC}{Theorem C}
\newtheorem*{theoremD}{Theorem D}
\newtheorem*{theoremE}{Theorem E}
\theoremstyle{definition}

\newtheorem{example}[theorem]{Example}
\newtheorem{observation}[theorem]{Observation}
\theoremstyle{remark}
\newtheorem{remark}[theorem]{Remark}
\newtheorem{question}[theorem]{Question}

\newcommand{\Z}{\mathbb{Z}}
\newcommand{\B}{\mathbb{B}}
\newcommand{\Q}{\mathbb{Q}}
\newcommand{\F}{\mathbb{F}}
\newcommand{\UConf}{\operatorname{UConf}}
\newcommand{\Tors}{\operatorname{Tors}}
\newcommand{\coker}{\operatorname{coker}}
\newcommand{\Tor}{\operatorname{Tor}}
\newcommand{\Sw}{\widetilde{S}}
\newcommand{\rk}{\operatorname{rank}}
\newcommand{\cone}{\operatorname{cone}}
\newcommand{\ord}{\operatorname{ord}}

\title[Torsion of every finite order in graph braid groups]{Torsion of every finite order\\ in the homology of graph braid groups}
\author{Byung Hee An}
\address{Department of Mathematics Education, Kyungpook National University, Daegu 41566, Republic of Korea}
\email{anbyhee@knu.ac.kr}
\begin{document}

\begin{abstract}
We determine the torsion subgroup of $H_{m-1}(\B_mK_{m+1,m+r-1};\Z)$ for
$m\ge2$ and $r\ge0$: top homology with arbitrary coefficients is the kernel
of an unsigned subset-inclusion matrix, and its integral diagonal form
determines all primary summands. Every finite order occurs, with explicit
representatives. Generalized theta classes span an embedded copy of the
cokernel of the inclusion matrix, containing all torsion; for $r\ge m$ they
generate the torsion, each of order $\operatorname{lcm}(1,\ldots,m)$.
For every prime power $q$ and $m\ge q$, the graph $K_{m+1,m+q-1}$ is
minimal in the minor order for order-$q$ torsion in $H_{m-1}(\B_m)$. In
particular, odd torsion first appears in
$H_2(\B_3K_{4,5})\cong\Z^{155}\oplus(\Z/2)^4\oplus\Z/3$, and no proper
minor of $K_{4,5}$ has odd torsion in $H_2(\B_3)$.
For arbitrary part sizes, we give a multiplicity-free decomposition of
$H_m(\B_mK_{a,b};\Q)$ under vertex permutations and prove that
$H_{m-1}(\B_mK_{a,b};\Z)$ has no $p$-primary torsion when $a,b\ge2m-1$
and $p\ge m$ is an odd prime. The explicit order-$q$ class retains its
order under every enlargement of the second part of the graph, while for
$m=q=p$ an odd prime it is killed by a specified enlargement of the first part.
\end{abstract}

\maketitle


\section{Introduction}

For a connected finite graph $\Gamma$ with at least one edge, the unordered configuration space $\UConf_n(\Gamma)$ of
$n$ points in $\Gamma$ is aspherical, and its fundamental group $\B_n\Gamma$ is the
\emph{graph braid group}.  Thus $H_*(\B_n\Gamma)=H_*(\UConf_n\Gamma)$, and this depends
only on the homeomorphism type of $\Gamma$. For disconnected graphs we use
the homology of the entire configuration space, as specified in Section~\ref{sec:prelim}.

The first homology is completely understood.  By Ko and Park \cite{KP},
$H_1(\B_n\Gamma)\cong\Z^N\oplus(\Z/2)^{N'}$, where the $2$-torsion reflects
non-planarity for $n\ge2$.  In higher degrees, Chettih--L\"utgehetmann study
configuration spaces of trees with loops \cite{CL}.  For planar graphs, An--Knudsen show
that $H_2$ is generated under embedding, disjoint union and edge stabilization by loop,
star and theta classes \cite{AK}; that result concerns generators, not torsion.  It is
natural to ask which torsion can occur.  Apart from the $2$-torsion of Ko--Park, no torsion was known.  In their
computations of $H_2$ and $H_3$ for complete graphs, complete bipartite graphs and the
Petersen family, Maci\k{a}\.zek and Sawicki found only copies of $\Z/2$ and left open
``the question about the existence of torsion different than $\Z_2$ in higher
homologies'' \cite[Section~5]{MS}.  An, Drummond-Cole and Knudsen asked explicitly:
\emph{Is there nontrivial odd torsion in the homology of graph braid groups?  Are there
elements of even order greater than $2$?} \cite[Question~1.5]{ADK22}; see also
\cite{ADK26,HKW}, and \cite{IRL}, where the absence of any known odd-primary torsion in
unordered graph configuration spaces is contrasted with the case of surfaces.
Indeed, for the classical braid groups the picture is entirely different: by
Arnold, Fuks, Va\u{\i}nshte\u{\i}n and F.~Cohen \cite{Arnold,Fuks,Vainshtein,CLM},
$H_*(\B_n\mathbb{R}^2;\Z)$ has $p$-torsion for every prime $p\le n/2$, and its
$p$-primary torsion is annihilated by $p$; see \cite{Napolitano} for surfaces.
We answer both questions affirmatively: every prime occurs and, unlike for the
classical braid groups on $\mathbb{R}^2$, every prime-power order occurs.
In fact, we construct elements of every finite order.

For the \emph{ordered} configuration spaces $F_n(\Gamma)$, torsion-freeness of the
homology is a folklore conjecture \cite[Conjecture~1.7]{HKW}, and no torsion is known.
Our results are consistent with it.  The covering $F_n(\Gamma)\to\UConf_n(\Gamma)$ has
degree $n!$, so by the transfer $H_*(\B_n\Gamma;\Z_{(p)})$ is a direct summand of
$H_*(F_n(\Gamma);\Z_{(p)})$ whenever $p>n$; hence the conjecture predicts that
$p$-torsion in $H_*(\B_n\Gamma)$ requires $p\le n$.  Theorem~A realizes this bound
sharply: for $m$ particles on $K_{m+1,m+r-1}$ the primes dividing the torsion one
degree below the top are exactly the primes at most $\min\{m,r\}$
(Corollary~\ref{cor:primary}), and the first $p$-torsion appears with exactly $p$
particles.

\subsection{Complete bipartite graphs and torsion}\label{sec:intro-torsion}
The loop, star and theta classes of \cite{KP,AK} admit representatives on
complete bipartite graphs. The loop class is the fundamental cycle of $K_{2,2}$ in $H_1(\B_1K_{2,2})$;
the star class generates $H_1(\B_2K_{3,1})$;
and the theta class generates $H_2(\B_3K_{4,2})$, where $K_{4,2}$ is a
subdivision of the four-edge theta graph $\Theta_4$ \cite{AK}. Under embedding,
disjoint union and edge stabilization, these classes generate $H_1$ of every
graph braid group and $H_2$ of every planar one \cite{KP,AK}. Each has infinite
order on its supporting graph. This suggests looking for torsion in their
images under graph embeddings: the star class already acquires order $2$ in
$K_{3,3}$. The family $K_{m+1,m+r-1}$ provides a setting in which to test the same
mechanism for the theta class and its higher analogues.

Two features make the homology in degrees $m$ and $m-1$ accessible with $m$
particles. Degree $m$ is the top degree, so its homology is the cycle module.
Moreover, every $m$-matching leaves exactly one free vertex in the part $A$ of
size $m+1$, allowing the vertex relations to reduce to the subset-inclusion
matrix $U_{r-1}(m+r-1)$ on the part $B$ (Section~\ref{sec:bipartite}). Wilson's integral
diagonal form for this matrix \cite{Wilson,GKMM}, together with the universal
coefficient theorem, determines every primary summand of the torsion one degree
below the top. The cases $r=2,3$ also admit direct calculations.

Write $[n]=\{1,\dots,n\}$, and let $U_j(n)$ denote the inclusion matrix between
$j$-subsets and $(j+1)$-subsets of $[n]$, with coefficients in a commutative ring $R$:
\[
U_j(n)\colon R^{\binom{[n]}{j}}\to R^{\binom{[n]}{j+1}},\qquad (U_j(n)h)(T)=\sum_{S\subset T,\ |S|=j}h(S).
\]

\begin{theoremA}[Theorems~\ref{thm:top} and~\ref{thm:exacttors}]
Let $m\ge1$, $r\ge0$, $s=\min\{m,r\}$, and
\[
\mu_j=\binom{m+r-1}{s-j}-\binom{m+r-1}{s-j-1}.
\]
With the conventions $\binom n{-1}=0$ and
$\binom{[n]}{-1}=\varnothing$, for every commutative ring $R$,
\[
H_m(\B_mK_{m+1,m+r-1};R)\cong\ker U_{r-1}(m+r-1).
\]
Moreover,
\[
\Tors H_{m-1}(\B_mK_{m+1,m+r-1};\Z)
 \cong\bigoplus_{j=2}^{s}(\Z/j)^{\mu_j}.
\]
In particular this torsion group is zero for $r=0,1$.

\end{theoremA}

Theorem~A gives the following occurrence criterion; see
Corollary~\ref{cor:orderk} for the proof.

\begin{corollary}[Torsion of every finite order]\label{cor:intro-orderk}
Let $k\ge2$, and let $q(k)$ be the largest prime power dividing $k$.
For $m\ge1$ and $r\ge0$, the group
\[
H_{m-1}(\B_mK_{m+1,m+r-1};\Z)
\]
contains a direct summand $\Z/k$ if and only if
\[
m\ge q(k),\qquad r\ge q(k).
\]
In particular, every finite order occurs in the integral homology of a graph
braid group: for $q=q(k)$, the group
\[
H_{q-1}(\B_qK_{q+1,2q-1};\Z)
\]
contains a direct summand $\Z/k$.
\end{corollary}

For example, $H_2(\B_3K_{4,5};\Z)$ already contains elements of order $6$.
The full decomposition over $R$, including the free rank and the Catalan
case $r=m+1$, is given in Theorem~\ref{thm:exacttors}. Corollary~\ref{cor:primary}
gives the primary summands: for $q=p^a$ and $m\ge q$, the $p$-primary
part for $r=q$ has exactly one summand of order $q$, and every other
summand has order at most $p^{a-1}$.

For $p=2$ the statements recover a part of the Ko--Park theorem: $H_1(\B_2K_{3,3})$ has
torsion $\Z/2$, and no proper subgraph of $K_{3,3}$ has $2$-torsion in $H_1(\B_2)$
(Theorem~C(1) below).  In this sense Theorem~A is the analogue, for every prime, of the
appearance of $2$-torsion at the minimal non-planar graph $K_{3,3}$.

\subsection{Explicit cycles and generators}\label{sec:intro-cycles}
The torsion has an elementary source: the constant function $h\equiv1$ satisfies
$U_{r-1}(m+r-1)h=r\cdot1$, which vanishes modulo $r$ but not integrally.  When $r=p$ is
prime and $m\ge p$, this produces a mod-$p$ class in the top degree whose Bockstein has order $p$,
and the following theorem makes it explicit.

\begin{theoremB}[Theorems~\ref{thm:explicitgen}, \ref{thm:theta-injective} and~\ref{thm:torsgen}]
\begin{enumerate}[leftmargin=*]
\item For $q\ge2$ and $m\ge q$, there are explicit matching-indexed chains
$w_m,z_m$ on $K_{m+1,m+q-1}$ with
\[
\partial w_m=qz_m,\qquad \partial z_m=0,\qquad \ord[z_m]=q.
\]
For prime $q=p$, $[z_m]$ generates the entire $p$-primary part.
\item Let $m\ge2$, $r\ge1$, and write $A,B$ for the parts of sizes $m+1,m+r-1$.
The generalized theta classes $x_{B'}$ supported on
$K_{A,B'}$, $B'\in\binom B{m-1}$, generate a subgroup, and
$\delta_T\mapsto x_{B\setminus T}$, $T\in\binom Br$, induces an isomorphism
\[
\coker U_{r-1}(m+r-1)\xrightarrow{\ \cong\ }
\langle x_{B'}\rangle\subseteq H_{m-1}(\B_mK_{m+1,m+r-1};\Z).
\]
All relations are generated by
$\sum_{B'\subset C,\ |B'|=m-1}x_{B'}=0$ for $C\in\binom Bm$,
and this subgroup contains the entire torsion subgroup.
Every $x_{B'}$ has infinite order for $1\le r\le m-1$.
For $r\ge m$, the subgroup is exactly the torsion subgroup, and
every $x_{B'}$ has order $\operatorname{lcm}(1,\ldots,m)$.
\end{enumerate}

\end{theoremB}

For $r=0$ the same formula gives $w_m=0$ and a cycle $z_m^{(0)}$ on
$K_{m+1,m-1}$: the star class for $m=2$, the theta class for $m=3$, and in general
the class announced in \cite[Remark~3.16]{AK}. Its images in larger complete
bipartite graphs are the restricted theta classes $x_{B'}$; their sum is $[z_m]$
by \eqref{eq:zsum}. These individual classes and their sum can have different
orders. In $H_2(\B_3K_{4,5})$, each theta class has order $6$, and its $3$-primary
component is $[z_3]$, of order $3$ (Theorem~\ref{thm:thetatorsion});
Remark~\ref{rem:inspiration} explains how this observation led to the torsion cycles.

These cycles fit into one construction (Section~\ref{sec:universal}). For a
function $h$ on the $(r-1)$-subsets of the second part, the matching chain $c_h$
satisfies $\partial c_h=z_{U_{r-1}h}$, where $z_g$ is the corresponding weighted sum
of restricted theta cycles (Proposition~\ref{prop:universal}). For $r=m+1$, the
same construction gives loop classes $L_P$ indexed by perfect matchings of the
second part. The non-nesting matchings give a basis of
$H_m(\B_mK_{m+1,2m};\Z)\cong\Z^{C_m}$, where
$C_m=\frac{1}{m+1}\binom{2m}{m}$ is the Catalan number
(Theorem~\ref{thm:loop}). Each class is carried by an orientable pseudomanifold:
for $m=1$ this recovers the loop class of $K_{2,2}$, and for $m=2$ the supporting
space is a genus-two surface pinched at three points.

\subsection{Minimality and graph operations}\label{sec:intro-minimality}
For $m=3$ and $r=3$ the full integral homology follows without computer assistance:
\[
H_*(\B_3K_{4,5};\Z)=\Z,\quad \Z^{12}\oplus\Z/2,\quad \Z^{155}\oplus(\Z/2)^4\oplus\Z/3,\quad 0 .
\]
By Observation~\ref{obs:top}, torsion in $H_i(\B_n\Gamma)$ requires $n\ge i+1$, so
torsion in $H_2$ requires at least three particles.  We show
that $K_{4,5}$ is the smallest example in the following sense.
For graphs $\Gamma$ and $K$, write $\Gamma\le K$ if $\Gamma$ is homeomorphic to a
subgraph of a subdivision of $K/F$ for some set $F$ of edges, where the loops and
parallel edges created by contraction are retained, and write $\Gamma<K$ if moreover
$\Gamma$ is not homeomorphic to $K$.  This is the minor order up to homeomorphism;
see Section~\ref{sec:minororder}.

\begin{theoremC}[Theorem~\ref{thm:strongmin} and Corollary~\ref{cor:fullminor-k45}]
\begin{enumerate}[leftmargin=*]
\item Let $q=p^a$ with $p$ prime and $a\ge1$, let $m\ge q$, and put $K=K_{m+1,m+q-1}$.
      The $p$-primary part of $H_{m-1}(\B_mK;\Z)$ has exactly one cyclic summand of order
      $q$, and all its other summands have smaller order.  For every $\Gamma<K$, the
      group $H_{m-1}(\B_m\Gamma;\Z)$ has no element of order $q$.
\item $H_2(\B_3K_{4,5};\Z)$ has odd torsion, whereas $H_2(\B_3\Gamma;\Z)$ has no odd torsion
      for every $\Gamma<K_{4,5}$.
\end{enumerate}
\end{theoremC}

The prime-power restriction in Theorem~C(1) is necessary: the analogous assertion for
$q=6$ is false (Remark~\ref{rem:primepower}).
For proper subgraphs, part (1) follows by extension by zero: every mod-$q$ top cycle of
order $q$ is nonzero on every matching.  The passage to topological minors uses the
integral pendant-edge splitting of Lemma~\ref{lem:pendant}.  In degree $m-1$ with $m$
particles its added summands are free, so it preserves torsion at every prime.  In the
same bidegree, attaching a loop or a parallel edge changes the torsion subgroup by an
explicit direct summand, $\Tors H_{m-2}(\B_{m-1}\Gamma_v)$ or
$\Tors H_{m-2}(\B_{m-1}\Gamma_{uw})$ (Theorem~\ref{cor:loop},
Theorem~\ref{prop:parallel-top}), which controls the loops and parallel
edges created by contraction in Theorem~C.  For
arbitrary particle number, pendant edges and simple cycles attached at one nonisolated
vertex preserve odd torsion in $H_2$.  For parallel edges we prove the same assertion for
at most three particles (Corollary~\ref{cor:parallel3}); the case of four or more
particles is left open in Question~\ref{q:parallel}. These invariance statements are
specific to the stated degrees and weights: loops, parallel edges, and pendant edges can
create torsion in other ranges.  Section~\ref{sec:pendant} gives exact formulas and
explicit counterexamples.  Minimality does not imply upward persistence:
$H_2(\B_3K_{5,5};\Z)$ has no $3$-torsion although $K_{5,5}\supset K_{4,5}$
(Proposition~\ref{rep:prop:two-primary-three} and Theorem~E below).

\subsection{Symmetry and behaviour under graph inclusions}
The matching-function model also applies when both parts of the complete
bipartite graph vary. Permuting their vertices gives an action of
$\Sigma_a\times\Sigma_b$. We identify the rational top homology by
relating ordered matchings to the top homology of chessboard complexes
and applying the Laplacian eigenvalue computation of Friedman--Hanlon
\cite{FH}, which yields the isotypic decomposition of the top homology.
Here $S^\rho$ denotes the rational Specht module of shape $\rho$,
$\lambda'$ is the conjugate partition, and $\boxtimes$ is the external
tensor product.

\begin{theoremD}[Theorems~\ref{rep:thm:rational} and \ref{rep:thm:generalvanishing}]
For $m\ge1$ and $a,b\ge m$, there is a multiplicity-free decomposition
\[
 H_m(\B_mK_{a,b};\Q)\cong
 \bigoplus_{\substack{\lambda\vdash m\\
 \lambda_1\le a-m,\ \lambda'_1\le b-m}}
 S^{(a-m,\lambda)}\boxtimes S^{(b-m,\lambda')}.
\]
If $m\ge2$, $a,b\ge2m-1$, and $p\ge m$ is an odd prime, then
$H_{m-1}(\B_mK_{a,b};\Z)$ has no $p$-primary torsion.
\end{theoremD}

The hook-length formula therefore determines the integral top rank for
all these part sizes. The torsion assertion uses an additional modular
argument: a join of chessboard complexes supplies projective
$\F_p\Sigma_m$-modules, and sign-coefficient group homology controls the
difference between the characteristic-zero and characteristic-$p$
top dimensions. For three particles, a $2$-local slot-module calculation
further gives
\[
 \Tors H_2(\B_3K_{a,b};\Z)\cong(\Z/2)^{(a-1)(b-1)}
 \qquad(a,b\ge5)
\]
(Proposition~\ref{rep:prop:two-primary-three}). Its generators are external
products of disjoint loop and star classes, and the torsion subgroup is
canonically isomorphic to $H_1(K_{a,b};\F_2)$
(Corollary~\ref{cor:loop-star-torsion}).
The explicit classes and the vanishing theorem give contrasting
behaviours under enlargement of the two parts.

\begin{theoremE}[Persistence and disappearance of torsion]
\begin{enumerate}[leftmargin=*]
\item Let $q\ge2$, $m\ge q$, and $r\ge q$. The inclusion
      \[
      K_{m+1,m+q-1}\hookrightarrow K_{m+1,m+r-1}
      \]
      sends the explicit class $[z_m^{(q)}]$ of Theorem~B to an element
      of order exactly $q$. When $m=q=p$ is prime, it induces an
      isomorphism on $p$-primary torsion, with both groups equal to $\Z/p$.
\item For every odd prime $p$, the inclusion
      \[
      K_{p+1,2p-1}\hookrightarrow K_{2p-1,2p-1}
      \]
      annihilates $[z_p^{(p)}]$. Indeed,
      $H_{p-1}(\B_pK_{2p-1,2p-1};\Z)$ has no $p$-primary torsion.
\end{enumerate}
\end{theoremE}
The two assertions are proved in Proposition~\ref{prop:second-part-persistence}
and Corollary~\ref{rep:cor:allprimecounterexample}, respectively.
Thus the explicit classes on the prime-power minimal graphs of Theorem~C
retain their orders under every enlargement of the second part.
For $m=q=p$ odd prime, the indicated enlargement of the first part destroys them.

\medskip\noindent\textbf{Organization.}\enspace
Section~\ref{sec:prelim} recalls the \'Swi\k{a}tkowski complexes.
Section~\ref{sec:top} proves Theorem~A.  Section~\ref{sec:explicit} proves
Theorem~B, identifies the star and theta classes with the case $r=0$ of the same
construction, and presents the universal construction with its two consequences,
the presentation of the theta subgroup and the loop classes.
Section~\ref{sec:symmetry} proves Theorems~D and~E and determines the
full torsion for three particles on $K_{a,b}$ with $a,b\ge5$.
Section~\ref{sec:pendant} treats pendant edges, loops and multiple edges, and
Section~\ref{sec:minimal} proves Theorem~C, deduces minimality for arbitrary
finite orders, and establishes analogous topological-minor minimality for the
complete graphs $K_{2m+1}$. Section~\ref{sec:questions} presents further questions
arising from these results.

\medskip\noindent\textbf{Acknowledgement.}\enspace
This work was supported by Samsung Science and Technology Foundation under Project Number SSTF-BA2022-03.

\section{Preliminaries}\label{sec:prelim}

Graphs are finite $1$-dimensional CW complexes; loops and multiple edges are allowed.
For a vertex $v$, let $H(v)$ be the set of half-edges at $v$ and $d(v)=|H(v)|$.  For a
half-edge $h$, let $e(h)$ be the edge containing it; when the graph is simple we also
write $h_{v,e}$ for the half-edge of the edge $e$ at its endpoint $v$.  Unless coefficients are specified,
homology is over $\Z$.  For a disconnected graph, $H_i(\B_n\Gamma)$ denotes the homology of
the entire configuration space $\UConf_n(\Gamma)$, summed over all components.  Homology
in negative degree or negative weight is zero, empty sums are zero, and the lcm of an empty set is $1$.  We write
$[a,b]$ for the homological-degree shift by $a$ and the weight shift by $b$; an odd
homological shift reverses the differential.  The symbol $\simeq$ denotes a
quasi-isomorphism (or a specified chain homotopy equivalence), not an isomorphism of
chain complexes.  We set $\binom N{-1}=0$.

The \emph{\'Swi\k{a}tkowski complex} \cite{Sw,ADK} is
\[
S(\Gamma)=\Z[E]\otimes\bigotimes_{v}S(v),\qquad
S(v)=\Z\langle\varnothing,\ v,\ h\in H(v)\rangle,
\]
bigraded by degree and weight: edges have bidegree $(0,1)$, $\varnothing$ has $(0,0)$,
$v$ has $(0,1)$, and each $h$ has $(1,1)$.  The differential is $\Z[E]$-linear, satisfies
the Koszul sign rule, and $\partial h=e(h)-v$.  Geometrically:
\begin{itemize}
\item the state $v$ is a particle resting at the vertex $v$;
\item $h$ is a particle moving from $v$ into $e(h)$;
\item an edge variable $e$ is a particle in the interior of $e$.
\end{itemize}

Fix $h_v\in H(v)$ for each nonisolated vertex $v$.  The \emph{reduced complex}
$\Sw(\Gamma)$ is the subcomplex obtained by replacing each vertex factor with
\[
\Sw(v)=\Z\langle\varnothing\rangle\oplus\bigoplus_{h\ne h_v}\Z\langle h-h_v\rangle,
\qquad \partial(h-h_v)=e(h)-e(h_v).
\]
At an isolated vertex retain the unreduced factor $\Z\langle\varnothing,v\rangle$.  With
this convention both complexes compute configuration-space homology \cite{Sw,ADK}:
\[
H_i(S(\Gamma))_n\cong H_i(\Sw(\Gamma))_n\cong H_i(\B_n\Gamma).
\]

\begin{observation}\label{obs:top}
The weight-$n$ part of $\Sw(\Gamma)$ vanishes in degrees $>n$, so $H_i(\B_n\Gamma)=0$
for $i>n$, and $H_n(\B_n\Gamma)$ is a subgroup of a free abelian group and hence
torsion-free.  Consequently, torsion in $H_i(\B_n\Gamma)$ requires $n\ge i+1$.
\end{observation}

\begin{theorem}[Ko--Park {\cite[Theorem~3.16]{KP}}; see also {\cite[Lemma~C.4]{ADK}}]\label{thm:KP}
For every finite graph $\Gamma$ and every $n$, the torsion of $H_1(\B_n\Gamma)$ is a
direct sum of copies of $\Z/2$.
\end{theorem}
For disconnected graphs this follows from the connected case by the K\"unneth formula.

We shall also use the Euler characteristics of graph braid groups, which are given by
Gal's formula \cite{Gal} (see also \cite[Corollary~5.12]{ADK}, where it is derived from
the \'Swi\k{a}tkowski complex):
\[
\sum_n\chi(\B_n\Gamma)\,t^n=\frac{\prod_v\bigl(1-(d(v)-1)t\bigr)}{(1-t)^{|E|}} .
\]

\section{Matching functions and exact torsion}\label{sec:top}

This section proves Theorem~A.  The starting point is that in the top degree, with
as many particles as the degree, the reduced complex has no cells above and the
homology is the cycle module.  Section~\ref{sec:matching} describes this module for
an arbitrary simple graph as a module of functions on ordered $m$-matchings subject
to one linear relation per vertex.  For $K_{m+1,m+r-1}$ the relations collapse, in
Section~\ref{sec:bipartite}, to the kernel of the inclusion matrix $U_{r-1}(m+r-1)$ of
$(r-1)$-subsets versus $r$-subsets of the second part $B$, with coefficients in any ring.
Section~\ref{sec:kerU} collects the properties of $U_{r-1}$ that are needed, and
Section~\ref{sec:exacttors} applies Wilson's integral diagonal form to $U_{r-1}$: since
the top homology over $\Z/\ell^a$ is the kernel of $U_{r-1}\otimes\Z/\ell^a$, the
universal coefficient theorem reads off every primary summand of the torsion one
degree below.  Section~\ref{sec:smallcases} works out the cases $r\le3$ by hand and
records the complete integral homology for $m\le3$.

\subsection{Top cycles as functions on matchings}\label{sec:matching}
Let $\Gamma$ be a simple graph, let $m\ge1$, and let $R$ be a commutative ring.
Isolated vertices cannot be occupied in degree $m$ and weight $m$, so they can be
omitted in this subsection.  Let $E(v)$ be the set of edges at $v$.  For $e\in E(v)$ let
$\delta_e\in R^{E(v)}$ be the indicator function of $e$, that is, $\delta_e(e')=1$ if
$e'=e$ and $\delta_e(e')=0$ otherwise.  Since $\Gamma$ is simple, half-edges at $v$
correspond to edges at $v$, and the assignment $h-h_v\mapsto\delta_{e(h)}-\delta_{e(h_v)}$
identifies $\Sw(v)_1\otimes R$ with
\[
A_v=\Bigl\{\varphi\colon E(v)\to R\ \Bigm|\ \textstyle\sum_e\varphi(e)=0\Bigr\}.
\]
Indeed, the elements $\delta_{e(h)}-\delta_{e(h_v)}$ for $h\ne h_v$ form a basis of $A_v$.
The identification is compatible with the differential: identify the weight-one edge
polynomials $\Z[E]_1\otimes R$ with $R^E$; then $\partial(h-h_v)=e(h)-e(h_v)$ corresponds
to $\delta_{e(h)}-\delta_{e(h_v)}$.

The part of $\Sw(\Gamma)\otimes R$ in degree $m$ and weight $m$ is
$\bigoplus_{|S|=m}\bigotimes_{v\in S}A_v$.  We encode a chain as a function $f$ on
ordered tuples $((v_1,e_1),\dots,(v_m,e_m))$, where the $v_i$ are distinct and
$e_i\in E(v_i)$.  We require $f$ to be alternating and to have sum zero in each slot.
We say that $v_i$ is \emph{occupied} and that the $i$-th particle \emph{points along}
$e_i$.

Edges themselves are unoriented: for an edge $e$ with endpoints $v$ and $x$ we write
$e=vx$, and this notation carries no orientation.  The orientation is recorded by the
pair $(v,e)$, which we identify with the half-edge of $e$ at $v$, that is, with $e$
oriented away from $v$.  In \eqref{eq:dtop} below, the $e$ on the left is an unoriented
edge variable, a particle in the interior of $e$, while the pairs $(v,e)$ on the right
are half-edges.

With the Koszul sign rule, the differential is
\begin{equation}\label{eq:dtop}
(\partial f)\bigl(e;(v_2,e_2),\dots,(v_m,e_m)\bigr)=\sum_{\substack{v\in e\\ v\notin\{v_2,\dots,v_m\}}}f\bigl((v,e),(v_2,e_2),\dots,(v_m,e_m)\bigr).
\end{equation}

\begin{lemma}\label{lem:topcycle}
Write $\mathrm{rest}=(v_2,e_2),\dots,(v_m,e_m)$.  A top chain $f$ is a cycle if and only
if:
\begin{enumerate}[label=\textup{(\alph*)},leftmargin=*]
\item if $e=vx$ and $x$ is occupied, then $f((v,e),\mathrm{rest})=0$;
\item if $e=vx$ and neither $v$ nor $x$ is occupied by $\mathrm{rest}$, then
      $f((v,e),\mathrm{rest})=-f((x,e),\mathrm{rest})$.
\end{enumerate}
A cycle moreover satisfies:
\begin{enumerate}[label=\textup{(\alph*)},leftmargin=*,resume]
\item $f$ vanishes unless $e_1,\dots,e_m$ are pairwise disjoint and each $e_i$ points to
      an unoccupied vertex.
\end{enumerate}
\end{lemma}

\begin{proof}
Conditions (a) and (b) say that \eqref{eq:dtop} vanishes: in (a) the sum has one term,
in (b) two.

For (c), suppose $e_1=vx$ and $e_2=yx$.  By (a) we may assume $x$ is unoccupied.  Moving
particle $1$ to $x$ by (b) produces a tuple in which particle $2$ points to the
occupied vertex $x$, and this value is $0$ by (a).
\end{proof}

By Lemma~\ref{lem:topcycle}, a top cycle $f$ is determined by its values on ordered
$m$-matchings with a chosen orientation of each edge.  We use the following notation.
For an oriented $m$-matching, written as a sequence of oriented edges
$(v_1x_1,\dots,v_mx_m)$ with tails $v_i$, put
\[
F(v_1x_1,\dots,v_mx_m)=f\bigl((v_1,v_1x_1),\dots,(v_m,v_mx_m)\bigr),
\]
so that the particle sits at the tail of each edge.  By (b), reversing one edge changes
the sign: $F(\dots,x_iv_i,\dots)=-F(\dots,v_ix_i,\dots)$.  Thus $F$ is an alternating
function on ordered $m$-matchings which is odd under reversal of an edge.

The sum-zero condition in slot $1$ is automatic unless the rest forms an
$(m-1)$-matching $M'$ whose particles point to unoccupied vertices, and $v\notin V(M')$.
In every other case all terms vanish by (a) or (c).  In the remaining case it reads
\begin{equation}\label{eq:Z}\tag{Z}
\sum_{x\in N(v)\setminus V(M')}F(vx,M')=0,
\end{equation}
where $vx$ is oriented from $v$ to $x$.

\begin{proposition}\label{prop:Zcycles}
For a simple graph $\Gamma$, $H_m(\B_m\Gamma;R)$ is isomorphic to the module of
functions $F$ on ordered oriented $m$-matchings of $\Gamma$ which are alternating,
odd under reversal of an edge, and satisfy \eqref{eq:Z} for every ordered
$(m-1)$-matching $M'$ and every vertex $v\notin V(M')$.
\end{proposition}

\begin{proof}
By Observation~\ref{obs:top} there are no chains above degree $m$, so
$H_m=\ker\partial_m$.  Conversely, any such $F$, extended by zero to tuples that are not
matchings or point to occupied vertices, satisfies (a), (b) and the sum-zero conditions.
\end{proof}

\begin{example}\label{ex:C3}
For $\Gamma=C_3$ and $m=1$, \eqref{eq:Z} at $v$ reads $F(vx)+F(vy)=0$, and oddness gives
$F(xv)=-F(vx)$.  The solutions are the multiples of the cyclic orientation, so
$H_1(\B_1C_3;R)\cong R$, as it should be.
\end{example}

\subsection{The complete bipartite case}\label{sec:bipartite}
Let $K=K_{m+1,m+r-1}$ with parts $A$ ($|A|=m+1$) and $B$ ($|B|=m+r-1$), $m\ge1$, $r\ge0$.
Orient every edge from $A$ to $B$, and write $F(ab,\dots)$ for the value with the
particle at $a\in A$.  Put $N=m+r-1$ and fix an enumeration $B=[N]$; thus $(r-1)$-subsets of
$B$ are elements of $\binom{[N]}{r-1}$, and $U_{r-1}(N)$ is the inclusion matrix of
$(r-1)$-subsets versus $r$-subsets of $B$ defined in the introduction.  Throughout this
section we abbreviate $U_{r-1}=U_{r-1}(N)$, with the conventions $\binom{[N]}{-1}=\varnothing$
and $\binom N{-1}=0$, so that $U_{-1}$ is the zero map from $R^0=0$.

\begin{theorem}\label{thm:top}
For every commutative ring $R$,
\[
H_m(\B_mK_{m+1,m+r-1};R)\cong\ker\Bigl(U_{r-1}\colon R^{\binom{[N]}{r-1}}\to R^{\binom{[N]}{r}}\Bigr),
\qquad N=m+r-1 .
\]
\end{theorem}

\begin{proof}
For $r=0$ there are no $m$-matchings, so every top cycle vanishes by
Lemma~\ref{lem:topcycle}(c), and both sides are zero.  Let $r\ge1$.
Let $M'$ be an $(m-1)$-matching.  It leaves two free vertices $a,a'\in A$ and a set
$B'\subset B$ of $r$ free vertices.  Since edges are oriented from $A$ to $B$, the
conditions \eqref{eq:Z} become:
\begin{itemize}[leftmargin=*]
\item[(R1)] at $v=b\in B'$: $-F(ab,M')-F(a'b,M')=0$, i.e.\ $F(ab,M')+F(a'b,M')=0$;
\item[(R2)] at $v=a\in\{a,a'\}$: $\sum_{b\in B'}F(ab,M')=0$.
\end{itemize}
An $m$-matching is described by three pieces of data: its free $A$-vertex $\alpha$, its
free $(r-1)$-subset $\beta\in\binom{[N]}{r-1}$, and a bijection $\sigma\colon A\setminus\alpha\to B\setminus\beta$.
Order $A=\{a_1<\dots<a_{m+1}\}$ and put $p(a_i)=i$.  Let $G(\alpha,\beta,\sigma)$ be the value
of $F$ with edges listed in increasing order of their $A$-endpoints.

\emph{Step 1.}  For $ab,cd\in M$, three applications of (R1), cycling the free
$A$-vertex $\alpha\to a\to c\to\alpha$, followed by one transposition, give
\[
F(ab,cd,\dots)=-F(\alpha b,cd,\dots)=F(\alpha b,ad,\dots)=-F(cb,ad,\dots)=F(ad,cb,\dots).
\]
Hence $G$ is invariant under $\sigma\mapsto\sigma\circ(a\ c)$, and therefore
$G=g(\alpha,\beta)$ does not depend on $\sigma$.  (For $m=1$ there is nothing to prove:
$\sigma$ is the unique map between two singletons.)

\emph{Step 2.}  In (R1) the edge at $a$ moves to $\alpha$.  Restoring increasing order
moves it past the $|p(a)-p(\alpha)|-1$ vertices strictly between them, so
$(-1)^{p(\alpha)}g(\alpha,\beta)=(-1)^{p(a)}g(a,\beta)$.  Hence
$g(\alpha,\beta)=(-1)^{p(\alpha)}h(\beta)$ for a function $h\colon\binom{[N]}{r-1}\to R$.

\emph{Step 3.}  The $r$ terms of (R2) have the same $\alpha$ and the same order of
$A$-endpoints, and their free $(r-1)$-subsets are the $(r-1)$-subsets $B'\setminus b$ of $B'$.
So (R2) says $(U_{r-1}h)(B')=0$.  Every $r$-subset $T$ of $B$ occurs as some $B'$: match
the $m-1$ vertices of $A\setminus\{a,a'\}$ bijectively onto $B\setminus T$, which has
$m-1$ elements.

\emph{Step 4.}  Conversely, for $h\in\ker U_{r-1}$, the function $G=(-1)^{p(\alpha)}h(\beta)$
satisfies (R1) by Step~2 and (R2) by Step~3.  Every vertex not in $V(M')$ lies in
$\{a,a'\}\cup B'$, so these are all the conditions \eqref{eq:Z}.  The theorem follows
from Proposition~\ref{prop:Zcycles}.
\end{proof}

For $r=1$ the kernel of $U_0\colon R\to R^{m}$, $1\mapsto(1,\dots,1)$, is zero, so
$H_m(\B_mK_{m+1,m};R)=0$: with only one free $B$-vertex, (R2) forces $F=0$.  Thus the
top homology of $K_{m+1,m+r-1}$ vanishes for $r\le1$, and Theorem~\ref{thm:exacttors}
below shows that it vanishes over $\Z$ for all $r\le m$.

\subsection{The kernel of the inclusion matrix}\label{sec:kerU}
The results of this subsection hold for every $N\ge1$.  Let
$D_{j+1}\colon R^{\binom{[N]}{j+1}}\to R^{\binom{[N]}{j}}$ be the transpose of
$U_j=U_j(N)$, $(D_{j+1}g)(S)=\sum_{T\supset S}g(T)$.

\begin{lemma}\label{lem:identity}
$D_{j+1}U_j-U_{j-1}D_j=(N-2j)\,\mathrm{id}$ on $R^{\binom{[N]}{j}}$.
\end{lemma}

\begin{proof}
For $|S|=j$, $(D_{j+1}U_jh)(S)=(N-j)h(S)+\sum_{|S\cap S'|=j-1}h(S')$, since each $S'$
with $|S\cap S'|=j-1$ lies in exactly one $(j+1)$-set containing $S$, namely $S\cup S'$.
Similarly $(U_{j-1}D_jh)(S)=j\,h(S)+\sum_{|S\cap S'|=j-1}h(S')$.  Subtract.
\end{proof}

\begin{lemma}\label{lem:kerU}
\begin{enumerate}[leftmargin=*]
\item If $N>2j$, then $\ker(U_j\colon\Z^{\binom{[N]}{j}}\to\Z^{\binom{[N]}{j+1}})=0$.
\item The constant function $1$ satisfies $U_j1=(j+1)\cdot1$.  Hence, if $j+1=p$ is
      prime, $1\in\ker(U_j\otimes\F_p)$.
\item If $j+1=p$ is prime and $N\ge2p-1$, then $\ker(U_j\otimes\F_p)$ is one-dimensional,
      spanned by the constant function.
\end{enumerate}
\end{lemma}

\begin{proof}
(1) Let $U_jh=0$ with $h$ integral.  Taking the standard inner product with $h$ in
Lemma~\ref{lem:identity},
\[
0=\langle D_{j+1}U_jh,h\rangle=\langle U_{j-1}D_jh,h\rangle+(N-2j)\|h\|^2=\|D_jh\|^2+(N-2j)\|h\|^2,
\]
using that $D_j$ is the transpose of $U_{j-1}$.  Since $N>2j$, $h=0$.

(2) Every $(j+1)$-set has $j+1$ subsets of size $j$.

(3) By Wilson's rank formula \cite{Wilson}, for $t\le\min(k,N-k)$ the rank
of the inclusion matrix $W_{t,k}(N)$ of $t$-subsets versus $k$-subsets of an $N$-set over
$\F_p$ is
\[
\rk_{\F_p}W_{t,k}(N)=\sum_{i\colon\ p\nmid\binom{k-i}{t-i}}\Bigl[\binom Ni-\binom N{i-1}\Bigr].
\]
Here $t=j=p-1$, $k=p$ and $N\ge2p-1$, so $N-k\ge t$ and the formula applies;
$\binom{k-i}{t-i}=p-i$ is divisible by $p$ only for $i=0$.  Hence
$\rk_{\F_p}U_j=\sum_{i=1}^{p-1}\bigl[\binom Ni-\binom N{i-1}\bigr]=\binom N{p-1}-1$, and the
kernel is one-dimensional.  By (2) it is spanned by the constant.
\end{proof}

\begin{corollary}\label{cor:allprimes}
Let $m\ge2$ and $r\ge2$.
\begin{enumerate}[leftmargin=*]
\item If $m\ge r$, then $H_m(\B_mK_{m+1,m+r-1};\Z)=0$.
\item Let $p$ be a prime and $m\ge p$.  Then $H_m(\B_mK_{m+1,m+p-1};\F_p)\cong\F_p$,
      spanned by the class of the constant function, and the $p$-primary part of
      $H_{m-1}(\B_mK_{m+1,m+p-1};\Z)$ is a nonzero cyclic group.
\end{enumerate}
\end{corollary}

\begin{proof}
(1) is Theorem~\ref{thm:top} with Lemma~\ref{lem:kerU}(1), since $N=m+r-1>2(r-1)$.

(2) Put $r=p$; then $m\ge p=r$, so $H_m(\,\cdot\,;\Z)=0$ by (1), and $N=m+r-1\ge2p-1$.
By the universal coefficient theorem, $H_m(\,\cdot\,;\F_p)\cong\Tor(H_{m-1}(\,\cdot\,;\Z),\F_p)$,
whose dimension is the number of $p$-primary cyclic summands of $H_{m-1}$.  By
Theorem~\ref{thm:top} and Lemma~\ref{lem:kerU}(3), this dimension is exactly $1$.
\end{proof}

\begin{example}
$H_1(\B_2K_{3,3})$ and $H_2(\B_3K_{4,4})$ contain $\Z/2$; $H_2(\B_3K_{4,5})$ contains $\Z/3$;
$H_4(\B_5K_{6,9})$ contains $\Z/5$; $H_6(\B_7K_{8,13})$ contains $\Z/7$.  In general the
first $p$-torsion produced by the corollary lives in $H_{p-1}(\B_pK_{p+1,2p-1})$.
\end{example}

\subsection{Exact torsion one degree below the top}\label{sec:exacttors}
The top homology with finite coefficients determines the torsion one degree below, even
when the integral top homology has positive rank.

\begin{theorem}[Wilson \cite{Wilson}; see also {\cite[Corollary~3]{GKMM}}]\label{thm:wilson}
Let $0\le t\le k$ and $N\ge t+k$.  The inclusion matrix $W_{t,k}(N)$, with rows indexed
by $t$-subsets and columns by $k$-subsets, is equivalent over $\Z$ to a rectangular
diagonal matrix with nonzero entries
\[
\binom{k-i}{t-i}\quad\text{of multiplicity}\quad\binom Ni-\binom N{i-1},\qquad 0\le i\le t.
\]
Here equivalence means multiplication by unimodular integer matrices on the left and
right, and the entries are listed without regard to divisibility.
\end{theorem}

\begin{theorem}\label{thm:exacttors}
Let $m\ge1$ and $r\ge0$, put $N=m+r-1$, $s=\min\{m,r\}$, $b=\binom N{r-1}-\binom N{s-1}$ and
$\mu_j=\binom N{s-j}-\binom N{s-j-1}$, and let $K=K_{m+1,m+r-1}$.  For every commutative
ring $R$,
\begin{equation}\label{eq:topR}
H_m(\B_mK;R)\;\cong\;\ker(U_{r-1}\otimes R)\;\cong\;R^{b}\oplus\bigoplus_{j=2}^{s}R[j]^{\mu_j},
\qquad R[j]=\{x\in R: jx=0\},
\end{equation}
and
\begin{equation}\label{eq:torsformula}
\Tors H_{m-1}(\B_mK;\Z)\;\cong\;\Tors\coker U_{r-1}\;\cong\;\bigoplus_{j=2}^{s}\bigl(\Z/j\bigr)^{\mu_j} .
\end{equation}
In particular the integral top homology is free of rank $b$.  It vanishes if and only if $m\ge r$: for $m<r$ one has $s=m$ and
$b=\binom Nm-\binom N{m-1}>0$ since $m\le N/2$, and for $r=m+1$ this is the Catalan number
$\frac1{m+1}\binom{2m}m$.  In particular, for $m\ge r$ the torsion formula is
\[
\bigoplus_{i=0}^{r-2}\bigl(\Z/(r-i)\bigr)^{\binom Ni-\binom N{i-1}} .
\]
For a prime power $q=p^a$ and $r=q$, the $p$-primary part has exactly one cyclic
summand of order $q$ and all its other cyclic summands have order at most $p^{a-1}$ if
$m\ge q$; it has exponent less than $q$ if $m<q$.  For $a=1$: the $p$-primary part is
$\Z/p$ if $m\ge p$ and is zero if $m<p$.  (For $m=1$ the sum in \eqref{eq:torsformula}
is empty, and indeed $H_0$ is free.  For $r\in\{0,1\}$ one has $s=r\le1$, $b=0$
and empty sums: the top homology vanishes and $H_{m-1}(\B_mK_{m+1,m+r-1};\Z)$ is
torsion-free.)
\end{theorem}

\begin{proof}
For $r=0$ the matrix $U_{r-1}$ is empty, both sides of \eqref{eq:topR} are zero by
Theorem~\ref{thm:top}, and the universal coefficient argument below gives
$\Tor(H_{m-1},\Z/\ell^a)=0$ for all $\ell,a$, so $H_{m-1}$ is torsion-free.  Let
$r\ge1$.  If $m\ge r$, the transpose of $U_{r-1}$ is $W_{r-1,r}(N)$ and Wilson's hypothesis is
$N\ge2r-1$ (for $r=1$ this is $W_{0,1}(N)$, a row of ones, with $N=m\ge1$).  If $m<r$, complementing all subsets identifies $U_{r-1}$, up to row and
column permutations, with $W_{m-1,m}(N)$, and its hypothesis is $N\ge2m-1$.  Both hold in
the respective ranges.  Consequently the nonzero diagonal entries of $U_{r-1}$ are
\[
d_i=s-i\quad\text{with multiplicity}\quad\nu_i=\binom Ni-\binom N{i-1},\qquad 0\le i\le s-1.
\]
Their multiplicities sum to $\binom N{s-1}$, so $U_{r-1}$ has $b=\binom N{r-1}-\binom N{s-1}$
zero columns in diagonal form.  Since unimodular matrices stay invertible over every
$R$, the kernel of $U_{r-1}\otimes R$ is that of the diagonal matrix, namely
$R^b\oplus\bigoplus_iR[d_i]^{\nu_i}$; with Theorem~\ref{thm:top} and $j=s-i$ this is
\eqref{eq:topR}.  For $R=\Z$ all $R[d_i]$ vanish, proving the assertion about the
integral top homology.

Put $G=H_{m-1}(\B_mK;\Z)$.  For a prime $\ell$ and $a\ge1$, take $R=\Z/\ell^a$.  The
split universal coefficient sequence and \eqref{eq:topR} give isomorphisms of finite
abelian groups
\[
R^b\oplus\Tor(G,R)\;\cong\;H_m(\B_mK;R)\;\cong\;R^b\oplus\bigoplus_{i=0}^{s-1}R[d_i]^{\nu_i}.
\]
Cancellation of the common finite summand $R^b$ is valid
by the classification of finite abelian groups.  As
\[
R[d_i]\cong\Z/\ell^{\min\{a,v_\ell(d_i)\}},
\]
comparison for all $\ell$ and $a$ determines every elementary divisor of $\Tors G$.  The
result is the direct sum of the $\Z/d_i$ with the stated multiplicities; the entries
$d_{s-1}=1$ contribute nothing.  This also equals the torsion of $\coker U_{r-1}$.  Its
possible free summand has no effect on this argument.  Reindexing by $j=s-i$ gives
\eqref{eq:torsformula}.

When $r=q$ and $m\ge q$, the diagonal entries are $d_i=q-i$, $0\le i\le q-1$; only
$d_0=q$ has $p$-adic valuation $a$, with $\nu_0=1$, while $v_p(q-i)=v_p(i)\le a-1$ for
$1\le i\le q-1$.  When $m<q$, all diagonal entries are at most $m<q$, so their $p$-parts
are at most $p^{a-1}$.  This proves the final assertions.
\end{proof}

\begin{corollary}[Primary summands and exponent]\label{cor:primary}
In the notation of Theorem~\ref{thm:exacttors},
\[
\Tors_\ell H_{m-1}(\B_mK;\Z)\cong\bigoplus_{a\ge1}\bigl(\Z/\ell^a\bigr)^{M_{\ell,a}},\qquad
M_{\ell,a}=\sum_{\substack{2\le j\le s\\ v_\ell(j)=a}}\Bigl[\binom N{s-j}-\binom N{s-j-1}\Bigr].
\]
The exponent of the torsion subgroup is $\operatorname{lcm}(1,\dots,s)$.  Its prime
divisors are exactly the primes at most $s$.
\end{corollary}

\begin{proof}
Decompose each cyclic group in \eqref{eq:torsformula} into primary factors.  All
displayed multiplicities are positive because $0\le s-j\le s-2<N/2$.  Thus each order
$2,\dots,s$ occurs, giving the exponent and the set of prime divisors.
\end{proof}

The torsion of an arbitrary order $k$ is governed by the largest prime power
dividing $k$.  For $k\ge2$ with prime factorization $k=p_1^{a_1}\cdots p_t^{a_t}$
($p_i$ distinct, $a_i\ge1$), put
\[
q(k)=\max\{p_1^{a_1},\dots,p_t^{a_t}\}.
\]
Thus $q(k)=k$ for a prime power and, for instance, $q(6)=3$, $q(12)=4$, $q(72)=9$.

\begin{corollary}[Torsion of order $k$]\label{cor:orderk}
Let $m\ge1$, $r\ge0$, $k\ge2$, and $K=K_{m+1,m+r-1}$.  The following are equivalent.
\begin{enumerate}[leftmargin=*]
\item $H_{m-1}(\B_mK;\Z)$ has an element of order $k$.
\item $\Z/k$ is a direct summand of $H_{m-1}(\B_mK;\Z)$.
\item $q(k)\le\min\{m,r\}$, i.e.\ $m\ge q(k)$ and $r\ge q(k)$.
\end{enumerate}
In particular $H_{q-1}(\B_qK_{q+1,2q-1};\Z)$ has a direct summand $\Z/k$ for $q=q(k)$,
and $(q,q)$ is the least pair $(m,r)$ in the coordinatewise order for which
$H_{m-1}(\B_mK_{m+1,m+r-1};\Z)$ contains an element of order $k$.
\end{corollary}

\begin{proof}
Put $s=\min\{m,r\}$ and write
$G=\Tors H_{m-1}(\B_mK;\Z)\cong\bigoplus_{j=2}^s(\Z/j)^{\mu_j}$ with all $\mu_j>0$
(Corollary~\ref{cor:primary}).  An element of order $k$ exists in a finite abelian group
if and only if, for each $i$, some cyclic summand has order divisible by $p_i^{a_i}$.
Since the cyclic summands of $G$ have orders exactly $2,\dots,s$, such a summand exists
if and only if $p_i^{a_i}\le s$.  Hence (1) is equivalent to $\max_ip_i^{a_i}\le s$,
which is (3).  If (3) holds, choose one cyclic summand from the term
$j=p_i^{a_i}$ for each $i$. These terms are distinct, and the selected summands give
$\Z/p_1^{a_1}\oplus\dots\oplus\Z/p_t^{a_t}\cong\Z/k$ as a direct summand of $G$,
giving (2); and (2) trivially implies (1).
\end{proof}

\begin{remark}\label{rem:orderk}
\begin{enumerate}[leftmargin=*]
\item Since $k\mid\operatorname{lcm}(1,\dots,s)$ if and only if $q(k)\le s$, condition
      (3) can be restated as $k\mid\operatorname{lcm}(1,\dots,\min\{m,r\})$, in
      accordance with the exponent computed in Corollary~\ref{cor:primary}.  For
      $r\ge m$ every generalized theta class has order exactly
      $\operatorname{lcm}(1,\dots,m)$ (Theorem~\ref{thm:torsgen}), so a suitable multiple
      of a single explicit class realizes every order $k$ with $q(k)\le m$.
\item For example, $H_2(\B_3K_{4,5};\Z)\cong\Z^{155}\oplus(\Z/2)^4\oplus\Z/3$ already
      contains elements of order $6$. Within the family
      $H_{m-1}(\B_mK_{m+1,m+r-1};\Z)$, order $12$ first occurs at $(m,r)=(4,4)$,
      in $H_3(\B_4K_{5,7};\Z)$, and order $30$ at $(m,r)=(5,5)$,
      in $H_4(\B_5K_{6,9};\Z)$.
\item Corollary~\ref{cor:orderk-minimal} establishes minor-minimality for arbitrary
      order $k$, using $q(k)$ as the graph parameter. Thus the occurrence criterion
      above also identifies the graphs to which the minimality theorem applies.
\end{enumerate}
\end{remark}

\begin{example}
For $r=2$ and $m\ge2$, the torsion is $\Z/2$. For $r=3$ and $m\ge3$ it is $\Z/3\oplus(\Z/2)^{m+1}$.
For $r=4$ and $m\ge4$ it is $\Z/4\oplus(\Z/3)^{N-1}\oplus(\Z/2)^{\binom N2-N}$, where
$N=m+3$.  For $m=5$, $r=5$,
\[
\Tors H_4(\B_5K_{6,9};\Z)\cong\Z/5\oplus(\Z/4)^8\oplus(\Z/3)^{27}\oplus(\Z/2)^{48}.
\]
The prime-to-$5$ subgroup in this example has exponent $12$, not at most $4$: bounds on
the orders of displayed cyclic summands are not bounds on the orders of all elements.
For $r=11$ and $m\ge11$, the formula also gives $(\Z/9)^{\binom{m+10}2-(m+10)}$ as a
direct summand.
\end{example}

\begin{remark}
This argument determines the free rank of $H_m$, not of $H_{m-1}$.  The latter still
depends on lower Betti numbers through Gal's formula.  For $m=3$ the known $H_1$ suffices
in the cases below; see Corollary~\ref{cor:small-cases}.
\end{remark}

\subsection{Explicit calculations for small parameters}\label{sec:smallcases}
We give direct calculations in the range $0\le r\le3$.
For $r=0,1$, Theorem~\ref{thm:top} already gives vanishing of the top
homology over every coefficient ring. The remaining cases $r=2,3$ admit
explicit descriptions of the top homology over every ring.
For $m=1,2,3$, the general formulas also determine the full integral
homology for every $r\ge0$, as recorded in
Corollary~\ref{cor:small-cases}.

We now treat $r=2,3$ directly, because the identification of $\ker U_2$
with a module of functions with vanishing triangle sums is used in
Section~\ref{sec:minimal}.  For $r=3$ the kernel is the module of functions on pairs
with vanishing triangle sums,
\[
T_R(N)=\ker U_2=\Bigl\{h\colon\tbinom B2\to R\ \Bigm|\ h(xy)+h(xz)+h(yz)=0\text{ for all distinct }x,y,z\Bigr\}.
\]

\begin{lemma}\label{lem:T}
For $N\ge5$ and every commutative ring $R$,
\[
T_R(N)\cong R[6]\oplus R[2]^{\,N-2}.
\]
In particular $T_\Z(N)=0$, $T_{\F_3}(N)\cong\F_3$ (the constants), $T_{\F_2}(N)\cong\F_2^{N-1}$,
$T_{\F_p}(N)=0$ for $p\ge5$, $T_{\Z/4}(N)\cong(\Z/2)^{N-1}$ and $T_{\Z/9}(N)\cong\Z/3$.
\end{lemma}

\begin{proof}
Write $B=\{0,1,\dots,N-1\}$ and $x_i=h(0i)$ for $i\ge1$.  The triangle relation for
$\{0,i,j\}$ gives $h(ij)=-x_i-x_j$, and then the relation for $\{i,j,k\}$ with
$i,j,k\ge1$ reads $2(x_i+x_j+x_k)=0$.  These are all the relations.  Since $N-1\ge4$,
comparing the relations for $\{i,j,k\}$ and $\{i,j,l\}$ gives $2x_k=2x_l$ for all $k,l$.
Put $a=x_1$ and $b_i=x_i-a$ for $2\le i\le N-1$.  Then $2b_i=0$, and the relations become
$6a=0$.  Conversely $6a=0$ and $2b_i=0$ imply all triangle relations.  This gives the
isomorphism, with $a\in R[6]$ and $(b_2,\dots,b_{N-1})\in R[2]^{N-2}$.
\end{proof}

\begin{corollary}\label{cor:exact}
For every $m\ge3$ (this is Theorem~\ref{thm:exacttors} for $r=3$; we give the direct
argument),
\[
H_m(\B_mK_{m+1,m+2};\Z)=0,\qquad
\Tors H_{m-1}(\B_mK_{m+1,m+2};\Z)\cong\Z/3\oplus(\Z/2)^{m+1}.
\]
For $m=2$, $H_2(\B_2K_{3,4})\cong\Z^2$.
\end{corollary}

\begin{proof}
The vanishing is Corollary~\ref{cor:allprimes}(1).  Since $H_m(\,\cdot\,;\Z)=0$, the
universal coefficient theorem gives $H_m(\,\cdot\,;\Z/q)\cong\Tor(H_{m-1},\Z/q)$ for every
$q$, and Theorem~\ref{thm:top} with Lemma~\ref{lem:T} ($N=m+2\ge5$) computes the left
side:
\begin{itemize}
\item $q=p\ge5$ prime: $0$, so there is no $p$-torsion;
\item $q=3$: $\F_3$, so there is exactly one $3$-primary cyclic summand $\Z/3^a$;
\item $q=9$: $\Z/3$.  Since $\Tor(\Z/3^a,\Z/9)$ has order $\min(3^a,9)$, we get $a=1$;
\item $q=2$: $\F_2^{m+1}$, so there are exactly $m+1$ cyclic summands $\Z/2^{a_i}$;
\item $q=4$: $(\Z/2)^{m+1}$, of order $2^{m+1}$.  Since $\Tor(\Z/2^{a_i},\Z/4)$ has order
      $\min(2^{a_i},4)\ge2$, all $a_i=1$.
\end{itemize}
For $m=2$, $N=4$: the relations $2(x_i+x_j+x_k)=0$ involve only the single triple
$\{1,2,3\}$, and over $\Z$ the solution set of $x_1+x_2+x_3=0$ is $\Z^2$.
\end{proof}

\begin{corollary}[$r=2$]\label{cor:r2}
For every $m\ge2$ (Theorem~\ref{thm:exacttors} for $r=2$), $H_m(\B_mK_{m+1,m+1};R)\cong R[2]$
for every commutative ring $R$, and
\[
\Tors H_{m-1}(\B_mK_{m+1,m+1};\Z)\cong\Z/2 .
\]
\end{corollary}

\begin{proof}
The kernel of $U_1$ consists of the functions $h$ on $B$ with $h(x)+h(y)=0$ for all
$x\ne y$.  Since $|B|=m+1\ge3$, for distinct $x,y,z$ we get $h(x)=-h(y)=h(z)=-h(x)$, so
$h$ is constant with $2h=0$; conversely such $h$ lie in the kernel.  Thus
$\ker U_1\cong R[2]$, and Theorem~\ref{thm:top} gives the first claim.  The second
follows as in the proof of Corollary~\ref{cor:exact}: $H_m(\,\cdot\,;\Z)=0$ since
$R[2]=0$ for $R=\Z$; $H_m(\,\cdot\,;\F_2)\cong\F_2$ gives exactly one $2$-primary
summand and $H_m(\,\cdot\,;\Z/4)\cong\Z/2$ gives exponent one; and
$H_m(\,\cdot\,;\F_p)=0$ for odd $p$.
\end{proof}

\begin{remark}[Recovering part of Ko--Park]\label{rem:ko-park}
For $m=2$, Corollary~\ref{cor:r2} says $\Tors H_1(\B_2K_{3,3})\cong\Z/2$, and
Theorem~\ref{thm:subgraph} with $p=m=2$ says that no proper subgraph of $K_{3,3}$ has
$2$-torsion in $H_1(\B_2)$.  Both are special cases of Theorem~\ref{thm:KP} in the
form of \cite{KP}: $H_1(\B_n\Gamma)$, for $n\ge2$, has $2$-torsion exactly when $\Gamma$ is
non-planar, and $K_{3,3}$ is a minimal non-planar graph.  Thus the top-homology method
recovers this part of the Ko--Park theorem, and Theorems~\ref{thm:top}
and~\ref{thm:subgraph} may be viewed as its analogue for every prime.

\end{remark}

\begin{corollary}\label{cor:small-cases}
For $1\le m\le3$ and $r\ge0$, the integral homology of
$\B_mK_{m+1,m+r-1}$ is given in Table~\ref{tab:small-cases},
and vanishes in degrees greater than $m$.
Here
\[
g_2(r)=\binom{r+1}2-\binom{r+1}1,\qquad
g_3(r)=\binom{r+2}3 - \binom{r+2}2,
\]
and
\[
f_3(r)=
2\binom{r+2}3 +19 \binom{r+2}2-11\binom{r+2}1
\qquad(r\ge3).
\]
\end{corollary}

\begin{proof}
For $m=1$, the configuration space is the graph itself.
The graph $K_{2,0}$ consists of two isolated points. For every $r\ge1$,
$K_{2,r}$ is connected and has cycle rank $r-1$.
All the remaining configuration spaces are connected, so $H_0\cong\Z$.
Observation~\ref{obs:top} gives vanishing in degrees greater than $m$.

For $m=2$, the graphs at $r=0,1$ are the three-armed star and a
subdivision of the three-edge theta graph. The Ko--Park formula
\cite[Theorem~3.16]{KP} gives $H_1\cong\Z$ and $\Z^3$, respectively.
For the theta graph, the cycle rank $2$ and the $2$-cut contribution
$(3-1)(3-2)/2=1$ give rank $3$.
For every $r\ge2$, $K_{3,r+1}$ is simple, $3$-connected and non-planar,
so the same formula gives
\[
H_1(\B_2K_{3,r+1};\Z)\cong\Z^{2r}\oplus\Z/2.
\]
Theorem~\ref{thm:top} gives $H_2=0$ for $r=0,1$, while
Theorem~\ref{thm:exacttors} gives, for $r\ge2$,
\[
H_2(\B_2K_{3,r+1};\Z)\cong
\Z^{\binom{r+1}{2}-\binom{r+1}{1}}=\Z^{g_2(r)}.
\]

For $m=3$, Theorem~\ref{thm:exacttors} gives $H_3=0$ for $0\le r\le2$
and, for $r\ge3$,
\[
H_3(\B_3K_{4,r+2};\Z)\cong
\Z^{\binom{r+2}{3}-\binom{r+2}{2}}=\Z^{g_3(r)}.
\]
The same theorem gives torsion-free $H_2$ for $r=0,1$,
$\Tors H_2\cong\Z/2$ for $r=2$, and
\[
\Tors H_2\cong(\Z/2)^{r+1}\oplus\Z/3
\qquad(r\ge3).
\]
For $r=0$, the graph $K_{4,2}$ is a subdivision of the four-edge
theta graph. Its cycle rank is $3$, and the $2$-cut consisting of
the two degree-$4$ vertices has four components, contributing
$(4-1)(4-2)/2=3$ to the free rank. There are no cut vertices, and
the marked components are topological circles. Thus Ko--Park gives
$H_1\cong\Z^6$. For every $r\ge1$, $K_{4,r+2}$ is simple,
$3$-connected and non-planar, and the same formula gives
\[
H_1(\B_3K_{4,r+2};\Z)\cong\Z^{3(r+1)}\oplus\Z/2.
\]
Gal's formula yields
\begin{align*}
\chi(\B_3K_{4,r+2})
&=[t^3]\frac{(1-(r+1)t)^4(1-3t)^{r+2}}{(1-t)^{4(r+2)}}\\
&=\frac{r^3}{6}+\frac{21r^2}{2}+\frac{49r}{3}-4.
\end{align*}
For $r=0,1,2$ this gives $\chi=-4,23,72$, and hence
$b_2=\chi-1+b_1=1,28,80$, respectively.
For $r\ge3$, the Euler characteristic identity instead gives
\[
b_2=\chi-1+b_1+b_3
=\chi-1+3(r+1)+g_3(r)=f_3(r).
\]
Combining these ranks with the torsion calculation proves the table.
\end{proof}

\begin{table}[ht]
\centering
\renewcommand{\arraystretch}{1.12}
\begin{tabular}{c c c c c c c}
\toprule
$m$ & $r$ & $K$ & $H_0$ & $H_1$ & $H_2$ & $H_3$\\
\midrule
$1$ & $0$ & $K_{2,0}$ & $\Z^2$ & $0$ & &\\
$1$ & $\ge1$ & $K_{2,r}$ & $\color{red}\Z$ & $\color{red}\Z^{r-1}$ & &\\
\midrule
$2$ & $0$ & $K_{3,1}$ & $\color{red}\Z$ & $\Z$ & $0$ &\\
$2$ & $1$ & $K_{3,2}$ & $\color{red}\Z$ & $\Z^3$ & $0$ &\\
$2$ & $\ge2$ & $K_{3,r+1}$ & $\color{red}\Z$ & $\color{red}\Z^{2r}\oplus\Z/2$ & $\color{red}\Z^{g_2(r)}$ &\\
\midrule
$3$ & $0$ & $K_{4,2}$ & $\color{red}\Z$ & $\Z^6$ & $\Z$ & $0$\\
$3$ & $1$ & $K_{4,3}$ & $\color{red}\Z$ & $\color{red}\Z^6\oplus\Z/2$ & $\Z^{28}$ & $0$\\
$3$ & $2$ & $K_{4,4}$ & $\color{red}\Z$ & $\color{red}\Z^9\oplus\Z/2$ & $\Z^{80}\oplus\Z/2$ & $0$\\
$3$ & $\ge3$ & $K_{4,r+2}$ & $\color{red}\Z$ & $\color{red}\Z^{3(r+1)}\oplus\Z/2$
    & $\color{red}\Z^{f_3(r)}\oplus(\Z/2)^{r+1}\oplus\Z/3$ & $\color{red}\Z^{g_3(r)}$\\
\bottomrule
\end{tabular}
\medskip
\caption{Integral homology for $1\le m\le3$ and $r\ge0$.
Here $K=K_{m+1,m+r-1}$ and $H_i=H_i(\B_mK;\Z)$;
blank entries denote zero groups. Red entries mark ranges on which
$\dim_{\F}H_i(\B_mK;\F)$ is polynomial in $r$ for every
field $\F$, with the polynomial allowed to depend on its characteristic.
These ranges need not be optimal for an individual characteristic.}
\label{tab:small-cases}
\end{table}

\begin{remark}[Classical generators]\label{rem:classical-generators}
The infinite cyclic groups
\[
H_1(\B_1K_{2,2};\Z),\qquad
H_1(\B_2K_{3,1};\Z),\qquad
H_2(\B_3K_{4,2};\Z)
\]
are generated, respectively, by the loop class, the star class, and the
theta class $A_2$ of An--Knudsen \cite[Construction~3.13, Lemma~3.14]{AK}.
The corresponding parameter pairs are
\[
(m,r)=(1,2),\qquad (2,0),\qquad (3,0).
\]
Section~\ref{sec:explicit} places these classical generators within the
explicit constructions of this paper.
\end{remark}

\section{Explicit cycles and the theta presentation}\label{sec:explicit}

This section proves Theorem~B. We define matching chains $w_m$ and chains $z_m$
with one particle resting in the part $A$, identify the case $r=0$ with the star
and theta classes (Section~\ref{sec:theta}), prove the boundary formula
$\partial c_h=z_{U_{r-1}h}$ (Section~\ref{sec:universal}), derive the theta
presentation (Section~\ref{sec:thetapres}), and construct the loop classes
(Section~\ref{sec:loopclasses}).

We work in the unreduced complex $S(K)$, $K=K_{m+1,m+r-1}$.  For an edge $e=ab$ with
$a\in A$ and $b\in B$, put
\[
\gamma_e=h_{a,e}-h_{b,e},\qquad\partial\gamma_e=[b]-[a];
\]
this is a particle crossing $e$ from $a$ to $b$.  Products of $\gamma$'s along pairwise
disjoint edges, and of resting points $[v]$, are taken in the order written, with Koszul
signs; resting points have even degree.  For $D\subset A$ let
\[
\Gamma_D=\sum_{\varphi\colon D\hookrightarrow B}\ \prod_{d\in D}\gamma_{d\varphi(d)},
\]
where the product is taken in increasing order of $d$.  Define
\begin{align}
w_m&=\sum_{\alpha\in A}(-1)^{p(\alpha)}\,\Gamma_{A\setminus\alpha},\notag\\
z_m&=\sum_{\substack{a,a'\in A\\ a<a'}}(-1)^{p(a)+p(a')}\bigl([a]-[a']\bigr)\cdot\Gamma_{A\setminus\{a,a'\}}.\label{eq:zm}
\end{align}
These chains depend on the ambient graph $K_{m+1,m+r-1}$. We write
$z_m^{(r)}$ for the chain in \eqref{eq:zm} when this dependence needs to
be explicit, and abbreviate it to $z_m$ when $r$ is fixed.
The superscript records the graph parameter, not a power.

In $w_m$, the vertex $\alpha$ is the \emph{free $A$-vertex}: since $|A|=m+1$, every
$m$-matching misses exactly one vertex of $A$.  So $w_m$ has one $m$-cube for each
$m$-matching, in which the $m$ particles cross the edges of the matching.

In $z_m$, the sum runs over pairs $\{a,a'\}$ of distinct $A$-vertices.  The remaining
$m-1$ particles cross the edges of an $(m-1)$-matching from $A\setminus\{a,a'\}$ into $B$,
and one more particle rests at $a$ (the term $[a]$) or at $a'$ (the term $-[a']$).  In
the term $[a]$, the vertex $a'$ is unoccupied: it is the free $A$-vertex of the
$m$-matchings whose boundary produces this term, and symmetrically for $[a']$.  No two
particles ever share an edge or a vertex.

\subsection{Star and theta classes}\label{sec:theta}
The formulas \eqref{eq:zm} make sense for every $r\ge0$: for $r=0$ the graph is
$K_{m+1,m-1}$, $\Gamma_D$ is a sum over bijections $D\to B$ when $|D|=m-1$, and there
are no $m$-matchings, so $w_m=0$. This subsection shows that $z_m^{(0)}$ is a
nontrivial \emph{free} class generalizing the star and theta classes.
The torsion representatives arise as sums of its images in larger
complete bipartite graphs, as in \eqref{eq:zsum} below.

Recall the two classes.  Let $Y=K_{1,3}$ be the star with centre $v$ and edges
$e_1,e_2,e_3$, and let $h_i$ be the half-edge of $e_i$ at $v$.  The \emph{star class}
is the class of
\[
\sigma=(h_1-h_2)e_3+(h_2-h_3)e_1+(h_3-h_1)e_2\in\Sw_1(Y)_2 ,
\]
which generates $H_1(\B_2Y)\cong\Z$; loop and star classes, together with
edge stabilization, generate $H_1$ of graph braid groups \cite{KP,AK}.
Let $\Theta_4$ be the multigraph with two vertices $v_1,v_2$ and four
parallel edges $e_1,\dots,e_4$, whose subdivision is $K_{2,4}$,
and let $h_{i,j}$ be the half-edge of $e_j$ at $v_i$.  An and Knudsen
\cite[Construction~3.13, Lemma~3.14]{AK} define the \emph{theta class}
\[
A_2=\sum_{\sigma\in\Sigma_4}\operatorname{sgn}(\sigma)\,
(e_1-e_{\sigma(3)})\,(h_{1,\sigma(1)}-h_{1,1})\,(h_{2,\sigma(2)}-h_{2,1})\in\Sw_2(\Theta_4)_3
\]
and show that it generates $H_2(\B_3\Theta_4)\cong\Z$; they observe
\cite[Remark~3.16]{AK} that the same construction produces an element of
$H_n(\B_{n+1}K_{n,n+2})$ for every $n\ge1$.  In our notation $K_{n,n+2}=K_{m+1,m-1}$
with $m=n+1$, the part of size $n+2$ being $A$, and the construction is
$z_m^{(0)}$.  Grouping the terms of \eqref{eq:zm} by the resting vertex, and writing
$a'$ for the vertex of $A$ that is neither resting nor covered,
\begin{equation}\label{eq:theta}
z_m^{(0)}=\sum_{a\in A}\ \sum_{\psi\colon B\hookrightarrow A\setminus a}
(-1)^{p(a)+p(a')}\operatorname{sgn}(a'-a)\;[a]\prod_{b\in B}\gamma_{\psi(b)b},
\end{equation}
the product taken in increasing order of $\psi(b)$: one particle rests at a vertex
of the larger part, and $m-1$ particles cross $m-1$ disjoint edges
joining the smaller part to all but two vertices of the larger part.  Both classical
classes have this shape.  For $m=2$ the larger part is the set of leaves of the
star $K_{3,1}$ with centre $b$, and \eqref{eq:theta} is the star class written with a
particle resting at a leaf $a$ while the other particle crosses from a second
leaf to $b$: the six terms $\pm[a_i]\gamma_{a_jb}$, $i\ne j$.  For
$m=3$, the identification with $\pm[A_2]$ follows from the primitivity
argument below and $H_2(\B_3\Theta_4)\cong\Z$, without requiring a
chain-level identification between different subdivision models.

The coefficient projection needed for nonvanishing must keep track
of the new leaves created by explosion. For a vertex $a$ of $\Gamma$ put
\[
C_a=\Z[E(\Gamma)]\otimes\bigotimes_{v\ne a}S(v).
\]
This is the complex of $\Gamma_a$ with only the newly created univalent
vertices reduced; their reduction gives a quasi-isomorphism
$C_a\hookrightarrow S(\Gamma_a)$. As graded modules,
\[
S(\Gamma)=C_a\otimes\Z\langle\varnothing,[a]\rangle
 \ \oplus\ \bigoplus_{h\in H(a)}hC_a.
\]
The first summand $T_a$ is a subcomplex, and
$\partial(hy)=(e(h)-[a])y-h\partial y$. Therefore
\begin{equation}\label{eq:explode}
\pi_a\colon S(\Gamma)\longrightarrow S(\Gamma)/T_a
       \cong\bigoplus_{h\in H(a)}C_a[1,1]
\end{equation}
is a chain map with the shifted differential on the target. It reads
off each half-edge coefficient, with the Koszul sign of moving that
half-edge to the front, and induces
\[
H_i(\B_n\Gamma;R)\longrightarrow
 \bigoplus_{h\in H(a)}H_{i-1}(\B_{n-1}\Gamma_a;R).
\]
The same construction applies after partial reduction at other vertices,
so these projections can be iterated at distinct vertices.
A class with nonzero image is nonzero.

\begin{proposition}\label{prop:theta}
Let $m\ge2$.
\begin{enumerate}[leftmargin=*]
\item $z_m^{(0)}$ is a cycle in $S(K_{m+1,m-1})$, and its integral
      homology class is nonzero and primitive.
\item For $m=2$, $z_2^{(0)}$ is the star class of the centre $b\in B$ and
      represents a generator of $H_1(\B_2K_{3,1})\cong\Z$.  For $m=3$, $z_3^{(0)}$
      represents the theta class $\pm[A_2]$, a generator of $H_2(\B_3K_{4,2})\cong\Z$.
\end{enumerate}
\end{proposition}

\begin{proof}
(1) Fix an $(m-2)$-matching $M''$ and write $i<j<k$ for the three
vertices of $A$ it misses. For each $b\in B\setminus V(M'')$, the
contribution to the boundary with remaining crossings $M''$ is,
up to the common sign $(-1)^{p(i)+p(j)+p(k)-3}$,
\[
\bigl(([i]-[j])([b]-[k])-([i]-[k])([b]-[j])
      +([j]-[k])([b]-[i])\bigr)\gamma_{M''}=0.
\]
The differentiated crossing has position $p(k)-2$, $p(j)-1$,
or $p(i)$ in the respective products. All resting factors have
degree zero, so the displayed cancellation is literal. These are
all boundary terms.

To prove primitivity, use the coefficient projections \eqref{eq:explode} defined above. Enumerate $B=\{b_1,\ldots,b_{m-1}\}$ and choose distinct
$a_1,\ldots,a_{m-1}\in A$. Explode $b_1,\ldots,b_{m-1}$ successively,
taking at each step the component for the half-edge of $a_jb_j$ at $b_j$.
Only terms whose crossings are precisely these $m-1$ edges survive.
The image of $z_m^{(0)}$ is therefore
$\pm([a_m]-[a_{m+1}])$ in $H_0(\B_1\Gamma_B;\Z)$, where $a_m,a_{m+1}$
are the two remaining vertices of $A$. The graph $\Gamma_B$ is the
disjoint union of $m+1$ stars, one centered at each vertex of $A$.
This difference is primitive in its free zeroth homology.
Consequently $[z_m^{(0)}]$ is nonzero and primitive over $\Z$.

(2) For $m=2$, $H_1(\B_2K_{3,1})$ is infinite cyclic, generated by the star
class. For $m=3$, subdivision invariance and \cite[Lemma~3.14]{AK}
identify $H_2(\B_3K_{4,2})$ with the infinite cyclic group generated
by $[A_2]$. Primitivity proves the claimed identifications, up to orientation.
\end{proof}

Thus $A_{m-1}:=[z_m^{(0)}]\in H_{m-1}(\B_mK_{m+1,m-1})$ is a primitive
class generalizing the star and theta classes. We next prove persistence
of its mod-$2$ image when vertices are added to the second part.

\begin{lemma}\label{lem:thetacoef}
Let $a\in A$ and $b\in B$ in $K=K_{m+1,m-1}$, $m\ge2$.  The component of
$\pi_a(z_m^{(0)})$ at the half-edge $h_{a,ab}$ is
$(-1)^{p(a)-1}\,z_{m-1}^{(0)}$, where $z_{m-1}^{(0)}$ is formed on the complete
bipartite graph $K_{A\setminus a,\,B\setminus b}\cong K_{m,m-2}$ with the induced
order of $A\setminus a$, and is regarded as a chain of $K_a$.
\end{lemma}

\begin{proof}
The half-edge $h_{a,ab}$ occurs in \eqref{eq:zm} only through the factor
$\gamma_{ab}=h_{a,ab}-h_{b,ab}$ of the terms with $a\in D=A\setminus\{i,k\}$ and
$\varphi(a)=b$; the remaining factors are $([a_i]-[a_k])\prod_{d\in D\setminus a}\gamma_{d\varphi(d)}$
with $\varphi|_{D\setminus a}$ a bijection onto $B\setminus b$.  Moving $\gamma_{ab}$ to
the front passes the $\#\{d\in D: d<a\}=p(a)-1-[i<a]-[k<a]$ factors before it.  Let
$p'$ be the position in $A\setminus a$; then $p'(x)=p(x)-[x>a]$ and hence
$(-1)^{p(x)+[x<a]}=-(-1)^{p'(x)}$ for $x\ne a$.  The total sign is therefore
$(-1)^{p(a)-1}(-1)^{p'(i)+p'(k)}$, and the sum over $i<k$ in $A\setminus a$ and over
the bijections $D\setminus a\to B\setminus b$ is $z_{m-1}^{(0)}$ on
$K_{A\setminus a,B\setminus b}$.
\end{proof}

\begin{theorem}\label{thm:thetanonzero}
For every $m\ge2$ and $k\ge0$, the extension by zero of $z_m^{(0)}$ to
$K_{m+1,m-1+k}$ represents a nonzero class in $H_{m-1}(\B_mK_{m+1,m-1+k};\F_2)$.
Consequently $A_{m-1}=[z_m^{(0)}]\ne0$ in $H_{m-1}(\B_mK_{m+1,m-1};\Z)$, it is not
divisible by $2$, and it stays nonzero under every inclusion
$K_{m+1,m-1}\subseteq K_{m+1,m-1+k}$.
\end{theorem}

\begin{proof}
Induction on $m$, for all $k$ simultaneously.  Let $\Gamma=K_{m+1,m-1+k}$ with parts
$A$ and $B$, let $B_0\subset B$ be the $(m-1)$-subset on which $z_m^{(0)}$ is formed,
and let $a\in A$.  Then $\Gamma_a$ is $K_{A\setminus a,B}\cong K_{m,m-1+k}$ with one
pendant edge attached at each vertex of $B$.  Lemma~\ref{lem:thetacoef}, applied to
the chain $z_m^{(0)}$ on $K_{A,B_0}$, shows that the component of
$\pi_a(z_m^{(0)})$ at $h_{a,ab}$, $b\in B_0$, is $\pm z_{m-1}^{(0)}$ formed on
$K_{A\setminus a,B_0\setminus b}$, that is, the extension by zero of $z_{m-1}^{(0)}$
to $K_{A\setminus a,B}\cong K_{m,(m-1)-1+(k+1)}$, regarded as a chain of $\Gamma_a$.
By Lemma~\ref{lem:pendant}, applied with $\F_2$ coefficients and iterated over the
pendant edges, $H_{m-2}(\B_{m-1}K_{A\setminus a,B};\F_2)$ is a direct summand of
$H_{m-2}(\B_{m-1}\Gamma_a;\F_2)$.  Hence, if the statement holds for $(m-1,k+1)$, the
image of $[z_m^{(0)}]$ under $\pi_{a*}$ is nonzero, and so is $[z_m^{(0)}]$.

It remains to treat $m=2$, where $z_2^{(0)}$ is the star class $\sigma_b$ of the
vertex $b\in B$ of degree $3$ in $K_{3,1+k}$.  Note that a chain which is a boundary
in $K_{3,1+k}$ is a boundary in every larger $K_{3,1+k'}$, so it suffices to treat
$k\ge2$.  Then $K_{3,1+k}$ is $3$-connected and non-planar; by Ko--Park
\cite[Theorem~3.16]{KP}, $H_1(\B_2K_{3,1+k})\cong\Z^{\beta_1}\oplus\Z/2$, and the
element of order $2$ is represented by the star class of any vertex of degree at least
$3$ \cite[Theorem~3.16]{KP}, \cite[Lemma~C.4]{ADK}.  In particular $\sigma_b$ is nonzero modulo $2$.  (For
$k=0$ one can also explode $b$: $\Gamma_b$ is three disjoint arcs, and the component
of $\pi_b(\sigma_b)$ at $h_{b,a_jb}$ is $\pm([a_i]-[a_k])\ne0$ in $H_0(\B_1\Gamma_b)$;
this also shows that $\sigma_b$ is primitive.)
\end{proof}

We now return to $r\ge1$ and ask what becomes of $A_{m-1}$ inside the larger graphs
$K_{m+1,m+r-1}$.  For an $(m-1)$-subset $B'\subset B$, the cycle $z_m^{(0)}$ formed on
$K_{A,B'}$ is extended by zero to $K_{m+1,m+r-1}$ and denoted
$z_m^{(0)}|_{B'}$. This individual restricted theta cycle is distinct
from the full sum $z_m=z_m^{(r)}$. By the proof of
Proposition~\ref{prop:theta}(1),
\begin{equation}\label{eq:zsum}
z_m^{(r)}=\sum_{B'\in\binom{B}{m-1}}z_m^{(0)}\big|_{B'} .
\end{equation}
The bounding chain $w_m$ restricts in the same way.  For an $m$-subset $T\subset B$
let $w_T=\sum_{\alpha\in A}(-1)^{p(\alpha)}\Gamma^T_{A\setminus\alpha}$, where
$\Gamma^T_D$ is the sum over the bijections $D\to T$; thus
$w_m=\sum_{T\in\binom Bm}w_T$.

\subsection{The universal boundary formula}\label{sec:universal}
The matching chains and theta cycles constructed in this section are instances of one construction.  Let $K=K_{m+1,m+r-1}$
with $m\ge1$ and $r\ge1$.  For an $m$-matching $M$ of $K$ write $\alpha_M\in A$ for
its free $A$-vertex, $\beta_M\in\binom B{r-1}$ for its set of free $B$-vertices, and
$\gamma_M=\prod_{e\in M}\gamma_e$, the product taken in increasing order of the
$A$-endpoints.  For functions $h\colon\binom B{r-1}\to R$ and $g\colon\binom B{r}\to R$
put
\begin{equation}\label{eq:universal}
c_h=\sum_{M}(-1)^{p(\alpha_M)}\,h(\beta_M)\,\gamma_M\ \in S_m(K)\otimes R,
\qquad
z_g=\sum_{T\in\binom B{r}}g(T)\,z_m^{(0)}\big|_{B\setminus T}\ \in S_{m-1}(K)\otimes R .
\end{equation}
Writing $\mathbf1$ for the constant weight function, we have
$c_{\mathbf1}=w_m$ and $z_{\mathbf1}=z_m=z_m^{(r)}$ by \eqref{eq:zsum}; $c_{\delta_{B\setminus T}}=w_T$ for
$T\in\binom Bm$; and $z_{\delta_T}=z_m^{(0)}|_{B\setminus T}$ is a restricted theta
class.  Explicitly, $z_g$ is obtained from \eqref{eq:zm} by weighting the term with
crossings $\varphi\colon A\setminus\{a,a'\}\hookrightarrow B$ by
$g(B\setminus\operatorname{im}\varphi)$.  Every $z_g$ is a cycle by
Proposition~\ref{prop:theta}(1) (for $m=1$ it has degree zero). Since the coefficient of the monomial
$\prod_{a\in A\setminus\alpha}h_{a,a\sigma(a)}$ in $c_h$ is $\pm h(\beta)$, and the
coefficient of $[a]\prod_{d}h_{d,d\varphi(d)}$ in $z_g$ is $\pm g(B\setminus\operatorname{im}\varphi)$,
we have $c_h=0$ if and only if $h=0$, and $z_g=0$ if and only if $g=0$.

\begin{proposition}[Universal boundary formula]\label{prop:universal}
For every $h\colon\binom B{r-1}\to R$,
\[
\partial c_h=z_{U_{r-1}h},\qquad (U_{r-1}h)(T)=\sum_{b\in T}h(T\setminus b).
\]
Consequently $c_h$ is a cycle if and only if $U_{r-1}h=0$, and in that case $c_h$ is
the top cycle corresponding to $h$ in Theorem~\ref{thm:top}: its matching function
in the sense of Proposition~\ref{prop:Zcycles}, read off as the coefficient of the
monomial of half-edges at the tails, is $(-1)^{p(\alpha)}h(\beta)$, the function of
Step~4 in the proof of Theorem~\ref{thm:top}.  In particular every top cycle of
$K$ over $R$ is of the form $c_h$.
\end{proposition}

\begin{proof}
In a matching with free $A$-vertex $\alpha$, the crossing at $a$
occupies position $p(a)-[\alpha<a]$. Differentiating that factor gives
$[b]-[a]$ with sign $(-1)^{p(a)-[\alpha<a]-1}$.  A face with a
particle resting at $b\in B$ arises from the two matchings $M'\cup\{ab\}$ and
$M'\cup\{a'b\}$; they have the same free $B$-set, hence the same weight $h(\beta)$,
and their signs differ by $[a'<a]-[a<a']=\pm1$, so they cancel.  A face with a particle resting at
$a\in A$ and crossings $M'$ arises from the matchings $M'\cup\{ab\}$, where $b$ runs
over the $r$ free vertices $B'=B\setminus V(M')$; the matching $M'\cup\{ab\}$ has
free $B$-set $B'\setminus b$, so the face occurs with total coefficient
$\sum_{b\in B'}h(B'\setminus b)=(U_{r-1}h)(B')$ times $(-1)^{p(\alpha)+p(a)-[\alpha<a]}$, where $\alpha$
is the other free $A$-vertex. For $a<\alpha$ this is the coefficient
of $[a]$ in the pair $\{a,\alpha\}$ of \eqref{eq:zm}; for $a>\alpha$
it has the opposite sign. Thus it is the sign of the corresponding term of $z_m^{(0)}|_{V(M')\cap B}=z_{\delta_{B'}}$.
Summing over $B'$ gives $\partial c_h=z_{U_{r-1}h}$.  The remaining assertions follow
from $z_g=0\Leftrightarrow g=0$ and from Theorem~\ref{thm:top}, whose proof
identifies the top cycles with the functions $(-1)^{p(\alpha)}h(\beta)$,
$h\in\ker U_{r-1}$.
\end{proof}

We first specialize to $h\equiv1$ and $h=\delta_{B\setminus T}$.
We then obtain the theta presentation of torsion and, for $r=m+1$,
a basis of free top homology from $\ker U_{r-1}$.

\begin{proposition}\label{prop:cubes}
For all $m\ge1$ and $r\ge0$, $\partial w_m=r\,z_m$ in $S(K_{m+1,m+r-1})$.  (For
$r=0$ both sides are zero, since there are no $m$-matchings; for $r=1$ the formula
reads $\partial w_m=z_m$.)
\end{proposition}

\begin{proof}
For $r\ge1$, take $h\equiv1$ in Proposition~\ref{prop:universal};
then $U_{r-1}h=r\cdot1$. For $r=0$, there are no $m$-matchings
and both sides are zero.
\end{proof}
\begin{theorem}\label{thm:explicitgen}
Let $q\ge2$, $m\ge q$, and set $r=q$. Then $[z_m^{(q)}]\in H_{m-1}(\B_mK_{m+1,m+q-1};\Z)$ has order
exactly $q$.  For prime $q=p$, it generates the entire $p$-primary part $\Z/p$.  For a
prime power $q=p^a$, it generates a cyclic direct summand of order $q$ in the $p$-primary
part.
\end{theorem}

\begin{proof}
Proposition~\ref{prop:cubes} gives $\partial w_m=qz_m$.  Since the chain groups are free,
$q\,\partial z_m=0$ implies $\partial z_m=0$.  The reduction $\bar w_m$ is a cycle with
coefficients in $\Z/q$.  Its coefficient at $h_{a_2,a_2b_1}\cdots h_{a_{m+1},a_{m+1}b_m}$
is $-1$, so this chain has order exactly $q$.  There are no chains above degree $m$, and
hence its homology class also has order $q$.

Since $m\ge r$, the integral top group is zero.  In the coefficient sequence
$0\to\Z\xrightarrow{\,q\,}\Z\to\Z/q\to0$, the Bockstein
\[
\beta_q\colon H_m(\B_mK;\Z/q)\longrightarrow H_{m-1}(\B_mK;\Z)
\]
is therefore injective and sends $[\bar w_m]$ to $[z_m]$.  This proves the exact order.
When $q=p$ is prime, the entire $p$-primary part is $\Z/p$ by
Theorem~\ref{thm:exacttors}.  When $q=p^a$, that theorem gives exactly one summand
$\Z/p^a$ and all other $p$-primary summands of exponent at most $p^{a-1}$.  An element of
order $p^a$ has unit coordinate in the first summand; its cyclic subgroup is consequently
a direct summand.
\end{proof}

\begin{remark}
In the identification of Theorem~\ref{thm:top}, $\bar w_m$ corresponds to the constant
function $h\equiv1$.  In the proof of Proposition~\ref{prop:cubes}, the cancelling faces
are relation (R1), and the faces of multiplicity $r$ are the relation
$U_{r-1}1=r\cdot1$.  For $m=3$, $r=3$ the cycle is
\[
z_3=\sum_{1\le i<k\le4}(-1)^{i+k}\bigl([a_i]-[a_k]\bigr)\sum_{b\ne b'}\gamma_{a_jb}\gamma_{a_lb'},\qquad\{j<l\}=\{1,2,3,4\}\setminus\{i,k\}.
\]
It consists of $240$ squares; the chain $w_3$ consists of $240$ cubes.
\end{remark}

\begin{lemma}\label{lem:wT}
For every $T\in\binom Bm$, $\ \partial w_T=\sum_{b\in T}z_m^{(0)}|_{T\setminus b}$.
\end{lemma}

\begin{proof}
Take $h=\delta_{B\setminus T}$ in Proposition~\ref{prop:universal}.
Its image under $U_{r-1}$ is the sum of the indicators indexed by
$(B\setminus T)\cup\{b\}$, $b\in T$, giving the assertion.
\end{proof}
\begin{theorem}[Theta classes in complete bipartite graphs]\label{thm:thetatorsion}
Let $m\ge2$, $r\ge m$, $N=m+r-1$, $H=H_{m-1}(\B_mK_{m+1,m+r-1};\Z)$, and put
$x_{B'}=[z_m^{(0)}|_{B'}]\in H$ for $B'\in\binom B{m-1}$.
\begin{enumerate}[leftmargin=*]
\item For every $T\in\binom Bm$, $\sum_{B'\subset T}x_{B'}=0$.  That is, the vector
      $(x_{B'})_{B'}$ lies in $\ker\bigl(U_{m-1}(N)\otimes H\bigr)$.
\item Every $x_{B'}$ is a torsion element, of order dividing $\operatorname{lcm}(1,\dots,m)$.
\item Let $m=p$ be prime and $r=p$, so that $K=K_{p+1,2p-1}$ and the $p$-primary
      part of $H$ is $\Z/p$, generated by $[z_p]$ (Theorems~\ref{thm:exacttors}
      and~\ref{thm:explicitgen}).  Then the $p$-primary component of $x_{B'}$ is
      $[z_p]$ for every $B'$.  If $p$ is odd, the $2$-primary component of $x_{B'}$
      is also nonzero, so the order of $x_{B'}$ is divisible by $2p$.
\item In particular, in $H_1(\B_2K_{3,3})$ the star class of every vertex of $B$ is
      the element of order $2$, and in
      $H_2(\B_3K_{4,5})\cong\Z^{155}\oplus(\Z/2)^4\oplus\Z/3$ the theta class
      $x_{B'}$ has order exactly $6$ for every $B'$, and all the $x_{B'}$ are
      congruent to $[z_3]$ modulo $2$-torsion.
\end{enumerate}
\end{theorem}

\begin{proof}
(1) is Lemma~\ref{lem:wT}.

(2) Since $N\ge2m-1$, Theorem~\ref{thm:wilson} applies to $U_{m-1}(N)=W_{m-1,m}(N)^{\mathsf T}$:
there are unimodular $P,Q$ with $U_{m-1}(N)=PDQ$, where $D$ is rectangular diagonal
with nonzero entries $d_i=m-i$ of multiplicity $\nu_i=\binom Ni-\binom N{i-1}$,
$0\le i\le m-1$, and $\sum_i\nu_i=\binom N{m-1}$ is the number of columns.  Hence
$\ker(U_{m-1}(N)\otimes H)=Q^{-1}\ker(D\otimes H)$ and
$\ker(D\otimes H)=\bigoplus_iH[d_i]^{\nu_i}$.  Every coordinate of a vector in the
kernel is an integral combination of elements killed by some $d_i\le m$, hence is
killed by $\operatorname{lcm}(1,\dots,m)$.

(3) Write $H=F\oplus\bigoplus_\ell H_{(\ell)}$ with $F$ free and $H_{(\ell)}$ the
$\ell$-primary parts, and decompose the kernel accordingly.  Since $U_{m-1}(N)$ is
injective over $\Z$ (Lemma~\ref{lem:kerU}(1), as $N>2(m-1)$), the $F$-component of
$(x_{B'})$ vanishes; this is (2) again.  By Theorem~\ref{thm:exacttors},
$H_{(p)}\cong\Z/p$, and by Lemma~\ref{lem:kerU}(3), $\ker(U_{p-1}(2p-1)\otimes\Z/p)$
is spanned by the constant vector.  So the $p$-primary components of all $x_{B'}$ are
equal to one element $c\in H_{(p)}$.  Summing \eqref{eq:zsum} in homology,
$[z_p]=\sum_{B'}x_{B'}$, whose $p$-primary component is $\binom{2p-1}{p-1}c=c$,
because $\binom{2p-1}{p-1}\equiv\binom10\binom{p-1}{p-1}=1\pmod p$ by Lucas' theorem.
Hence $c=[z_p]$.  For the $2$-primary component, Theorem~\ref{thm:thetanonzero} with
$k=r$ shows that $x_{B'}$ is nonzero in $H\otimes\F_2$, i.e.\ $x_{B'}\notin2H$;
since $x_{B'}$ is torsion, its $2$-primary component is nonzero.

(4) For $K_{3,3}$, $B'=\{b\}$ and $x_{B'}$ is the star class $\sigma_b$; by (3) its
$2$-primary component is $[z_2]$, the generator of $\Tors H_1(\B_2K_{3,3})\cong\Z/2$,
and there is no other torsion, so $\sigma_b=[z_2]$.  For $K_{4,5}$, (2) and (3) give
$6x_{B'}=0$ and $2p=6\mid\ord(x_{B'})$.  The torsion subgroup is
$(\Z/2)^4\oplus\Z/3$ (Corollary~\ref{cor:exact}), so $x_{B'}-[z_3]$ is $2$-torsion.
\end{proof}

Theorem~\ref{thm:torsgen} below sharpens (2): for $r\ge m$ the order of $x_{B'}$ is
exactly $\operatorname{lcm}(1,\dots,m)$, and the classes $x_{B'}$ generate the torsion
subgroup.

\begin{remark}[From theta classes to torsion]\label{rem:inspiration}
Theorem~\ref{thm:thetatorsion} is how the torsion cycles of this paper were found.
The theta class $A_2$ has infinite order in the subdivision $K_{4,2}$ of
$\Theta_4$ by Proposition~\ref{prop:theta}. It retains infinite order in
$K_{4,3}$ and $K_{4,4}$ by Theorem~\ref{thm:theta-injective} and
Remark~\ref{rem:lowr}.
By Theorem~\ref{thm:thetatorsion},
it becomes a class of order $6$ in $K_{4,5}$, whose $3$-primary part is exactly the $3$-torsion $[z_3]$ of
Theorem~\ref{thm:explicitgen}; the star class of a degree-$3$ vertex, viewed in
$K_{3,3}$, becomes the $2$-torsion of Ko--Park.  The torsion cycles $z_m$ and the
bounding chains $w_m$ are the sums \eqref{eq:zsum} and $w_m=\sum_Tw_T$ of these
restricted classes, and Lemma~\ref{lem:wT} is the reason why the coefficient $r$
appears in $\partial w_m=rz_m$: each restricted class $z_m^{(0)}|_{B'}$ is
bounded by $\sum_{T\supset B'}w_T$ up to the other restricted classes, and $B'$ lies
in $r$ sets $T$. Theorem~\ref{thm:theta-injective} below shows that these
incidence relations are all the relations, including before the torsion range.
\end{remark}

\subsection{The theta subgroup and its presentation}\label{sec:thetapres}
The map $g\mapsto[z_g]$ is linear, and by Proposition~\ref{prop:universal} it kills
the image of $U_{r-1}$.  It therefore induces a homomorphism
\begin{equation}\label{eq:Phi}
\Phi\colon\coker\Bigl(U_{r-1}(N)\colon\Z^{\binom{[N]}{r-1}}\to\Z^{\binom{[N]}{r}}\Bigr)\longrightarrow H_{m-1}(\B_mK_{m+1,m+r-1};\Z),
\qquad N=m+r-1,
\end{equation}
with $\Phi(\delta_T)=x_{B\setminus T}$ in the notation of
Theorem~\ref{thm:thetatorsion}. We first show that this is a presentation of
the subgroup generated by theta classes, without a restriction to the torsion range.

\begin{theorem}[The theta subgroup]\label{thm:theta-injective}
Let $m\ge2$, $r\ge1$, and $K=K_{m+1,m+r-1}$.
The homomorphism $\Phi$ in \eqref{eq:Phi} is injective.
Its image is the subgroup generated by the restricted theta classes, and
contains the entire torsion subgroup of $H_{m-1}(\B_mK_{m+1,m+r-1};\Z)$.
In particular, all integral relations among these classes are generated by
\[
\sum_{B'\subset C,\ |B'|=m-1}x_{B'}=0,
\qquad C\in\binom Bm.
\]
\end{theorem}
\begin{proof}
Suppose $[z_g]=0$, and choose a chain $c\in S_m(K)_m$ with $\partial c=z_g$.
Every monomial of $c$ consists of $m$ half-edges at distinct vertices:
degree and weight are both $m$, so there are no edge variables or resting
vertices. Separate the boundary into terms containing an edge variable and
terms containing a resting vertex. The former part is zero, since $z_g$
has no edge variables.

Encode $c$ by its coefficients on ordered tuples of half-edges, as in
Section~\ref{sec:matching}, but without imposing the slot-sum equations.
The edge-variable part of the differential is still \eqref{eq:dtop}.
Its vanishing gives exactly conditions (a) and (b) of Lemma~\ref{lem:topcycle}:
a coefficient is zero if a half-edge points to another occupied vertex,
and reversing an edge whose endpoints are unoccupied by the other slots changes its sign.
The proof of condition (c) uses only these two conditions, so the coefficients
vanish unless the edges form a matching. Thus
\[
c=\sum_M a_M\gamma_M,
\]
where $M$ runs over the $m$-matchings, oriented from $A$ to $B$ and ordered
by their $A$-endpoints. Indeed, expansion of $\gamma_M$ gives precisely
the signs under reversal just obtained.
Write $F$ for these coefficients on ordered oriented matchings, with
the usual signs under permutation and reversal.

There are also no terms of $z_g$ with a resting vertex in $B$.
For an $(m-1)$-matching $M'$ and a free vertex $b\in B$, the coefficient
with resting vertex $b$ therefore says
$F(ab,M')+F(a'b,M')=0$, where $a,a'$ are the two free vertices of $A$.
These are relations (R1) in the proof of Theorem~\ref{thm:top}.
Steps~1 and~2 of that proof use only (R1), and show that
$a_M=(-1)^{p(\alpha_M)}h(\beta_M)$ for a function $h$ on the free
$(r-1)$-subsets of $B$. Consequently $c=c_h$. Proposition~\ref{prop:universal}
now gives $z_g=\partial c_h=z_{U_{r-1}h}$, and injectivity of $g\mapsto z_g$
as a map of chains gives $g=U_{r-1}h$. This proves injectivity of $\Phi$.

Its image is generated by $\Phi(\delta_T)=x_{B\setminus T}$, $T\in\binom Br$,
and the column of $U_{r-1}$ indexed by $S\in\binom B{r-1}$, namely
$U_{r-1}\delta_S=\sum_{b\notin S}\delta_{S\cup b}$, gives exactly the displayed
relation for $C=B\setminus S$. Finally,
Theorem~\ref{thm:exacttors} identifies the torsion groups of the source
and target as abstract finite groups. The injection on torsion is
therefore surjective, proving the remaining assertion.
\end{proof}

To determine the orders of individual theta classes, we use a pairing
on the cokernel.

\begin{lemma}\label{lem:pairing}
Let $r\ge1$, $N=m+r-1$, $T_0\in\binom{[N]}{r}$, and let $1\le n\le m$.  If
$k\,\delta_{T_0}\in\operatorname{im}U_{r-1}(N)$, then $n\mid k$.  Hence the class of
$\delta_{T_0}$ in $\coker U_{r-1}(N)$, if it has finite order, has order divisible by
$\operatorname{lcm}(1,\dots,m)$.
\end{lemma}

\begin{proof}
Put $B'_0=[N]\setminus T_0$, an $(m-1)$-set, and choose $Y\subset B'_0$ with
$|Y|=m-n$.  Let $y$ be the indicator function of the $(m-1)$-subsets containing
$Y$, and define the linear functional
$\lambda(g)=\sum_{T}g(T)\,y([N]\setminus T)$ on $\Z^{\binom{[N]}{r}}$.  For an
$(r-1)$-set $S$, the $m$-set $S^c=[N]\setminus S$ satisfies
\[
\lambda(U_{r-1}\delta_S)=\sum_{T\supset S}y(T^c)=\sum_{b\in S^c}y(S^c\setminus b)
=\#\{b\in S^c\setminus Y\}\cdot[\,Y\subset S^c\,]=n\,[\,Y\subset S^c\,]\equiv0\pmod n .
\]
So $\lambda$ vanishes modulo $n$ on $\operatorname{im}U_{r-1}$, while
$\lambda(\delta_{T_0})=y(B'_0)=1$.  Hence $k\delta_{T_0}\in\operatorname{im}U_{r-1}$
forces $k\equiv0\pmod n$.
\end{proof}

\begin{theorem}[Theta classes generate the torsion]\label{thm:torsgen}
Let $m\ge2$, $r\ge m$, $N=m+r-1$ and $K=K_{m+1,m+r-1}$.
\begin{enumerate}[leftmargin=*]
\item $\coker U_{r-1}(N)$ is finite, and $\Phi$ is an isomorphism
      \[
      \Phi\colon\coker U_{r-1}(N)\xrightarrow{\ \cong\ }\Tors H_{m-1}(\B_mK;\Z).
      \]
\item The restricted theta classes $x_{B'}=[z_m^{(0)}|_{B'}]$, $B'\in\binom B{m-1}$,
      generate $\Tors H_{m-1}(\B_mK;\Z)$, and the relations
      $\sum_{B'\subset T}x_{B'}=0$, $T\in\binom Bm$, generate all relations among
      them.
\item Every $x_{B'}$ has order exactly $\operatorname{lcm}(1,\dots,m)$, the exponent
      of $\Tors H_{m-1}(\B_mK;\Z)$.
\end{enumerate}
\end{theorem}

\begin{proof}
Put $q=\operatorname{lcm}(1,\dots,m)$.  Since $r\ge m$ we have $s=m$ in
Theorem~\ref{thm:exacttors}, so $q$ is the exponent of $\Tors H_{m-1}(\B_mK;\Z)$
(Corollary~\ref{cor:primary}).  In the proof of Theorem~\ref{thm:exacttors},
$U_{r-1}(N)$ is equivalent to a diagonal matrix with nonzero entries $d_i=m-i$,
$0\le i\le m-1$, of total multiplicity $\binom N{m-1}=\binom N{r}$, the number of
rows.  Hence $\coker U_{r-1}(N)\cong\bigoplus_i(\Z/d_i)^{\nu_i}$ is finite, killed by
$q$, and of the same order as $\Tors H_{m-1}(\B_mK;\Z)$ by \eqref{eq:torsformula}.

(1) By Theorem~\ref{thm:theta-injective}, $\Phi$ is injective and its
image contains the entire torsion subgroup. Since its source is finite,
its image is exactly that torsion subgroup.

(2) The classes $\Phi(\delta_T)=x_{B\setminus T}$ generate because the $\delta_T$
generate $\coker U_{r-1}(N)$, and by (1) the relations among them are generated by the
columns $U_{r-1}\delta_S=\sum_{b\notin S}\delta_{S\cup b}$ of $U_{r-1}$, which in the
complementary indexing $B'=B\setminus T$ read $\sum_{B'\subset T'}x_{B'}=0$ with
$T'=B\setminus S\in\binom Bm$; these are the relations of Lemma~\ref{lem:wT}.

(3) By (1), the order of $x_{B'}$ equals the order of $\delta_{B\setminus B'}$ in
$\coker U_{r-1}(N)$, which divides $q$ and is divisible by $q$ by
Lemma~\ref{lem:pairing}.
\end{proof}

Theorem~\ref{thm:torsgen} sharpens Theorem~\ref{thm:thetatorsion}(2) and recovers
its part~(4): in $H_2(\B_3K_{4,5})$ the ten theta classes $x_{B'}$ have order exactly
$6$ and generate $\Z/3\oplus(\Z/2)^4$, with the ten relations of
Lemma~\ref{lem:wT}, while in $H_1(\B_2K_{3,3})$ the three star classes are equal to
the element of order $2$.  In general the $p$-primary part of $x_{B'}$ is nonzero
for every prime $p\le m$, in accordance with Corollary~\ref{cor:primary}.

\begin{remark}[The range $1\le r\le m-1$]\label{rem:lowr}
In this range the theta subgroup has free rank
\[
\binom N{r}-\binom N{r-1}>0,\qquad N=m+r-1,
\]
by Theorems~\ref{thm:theta-injective} and~\ref{thm:exacttors}.
Each $\delta_T$ has infinite order in $\coker U_{r-1}(N)$: if it lay in
$\operatorname{im}U_{r-1}\otimes\Q$, then by the symmetry of $[N]$ so would
every $\delta_{T'}$, contradicting the rank. Hence every restricted
theta class $x_{B'}$ has infinite order. Together with
Proposition~\ref{prop:theta}, this shows that the theta class on
$K_{m+1,m-1}$ survives with infinite order throughout $0\le r\le m-1$;
at $r=m$ it becomes torsion of order $\operatorname{lcm}(1,\ldots,m)$.
\end{remark}

\subsection{Free top homology: loop classes}\label{sec:loopclasses}
Theorem~\ref{thm:exacttors} identifies
\[
H_m(\B_mK_{m+1,2m};\Z)\cong\ker U_m(2m).
\]
This group is free of Catalan rank $C_m=\frac1{m+1}\binom{2m}m$, and by
Proposition~\ref{prop:universal} its elements are the cycles $c_h$ with
$h\in\ker U_m(2m)$.  For $m=1$ the graph $K_{2,2}$ is a circle,
$\ker U_1(2)$ is spanned by $h=\delta_{b_1}-\delta_{b_2}$, and
$c_h=\pm(\gamma_{a_1b_1}-\gamma_{a_2b_1}+\gamma_{a_2b_2}-\gamma_{a_1b_2})$
is its fundamental cycle, the \emph{loop class}.  We now write down a
basis for every $m$.

A \emph{perfect matching} $P$ of $B$ is a partition of $B$ into $m$ pairs
$\{i_k<j_k\}$.  It is \emph{non-nesting} if no pair lies inside another, that is,
$i_k<i_l$ implies $j_k<j_l$.  There are $C_m$ non-nesting perfect matchings of
$[2m]$.  Define $h_P\colon\binom Bm\to\Z$ by
\[
h_P(F)=\begin{cases}(-1)^{\#\{k\colon j_k\in F\}}&\text{if $F$ contains exactly one element of each pair of $P$,}\\ 0&\text{otherwise,}\end{cases}
\]
and put
\[
L_P=c_{h_P}=\sum_M(-1)^{p(\alpha_M)}h_P(\beta_M)\gamma_M .
\]
Thus $L_P$ is the signed sum of the $m$-cubes in which the $m$ particles cross
the edges of an $m$-matching leaving one vertex of each pair of $P$ free, the sign
recording the free $A$-vertex and, for each pair, which of its two vertices is free.
For $m=1$ and $P=\{\{b_1,b_2\}\}$ this is the loop class.

Geometrically, let $UD_m(K)$ be Abrams' discretized configuration space
\cite{Abrams}, the cube complex whose cells are the products of $m$ pairwise
disjoint closed cells of $K$.  Sending a traversed edge $e=ab$ to $\gamma_e$ and a
resting vertex $v$ to $[v]$ defines an injective chain map
$\iota\colon C_*(UD_m(K))\to S(K)$ into the weight-$m$ part.  Let $X_P\subset UD_m(K)$
be the union of the closed $m$-cubes of the matchings $M$ with $h_P(\beta_M)\ne0$,
that is, the configurations in which no pair of $P$ carries two particles.

\begin{theorem}[Loop classes]\label{thm:loop}
Let $m\ge1$ and $K=K_{m+1,2m}$.
\begin{enumerate}[leftmargin=*]
\item For every perfect matching $P$ of $B$, $h_P\in\ker U_m(2m)$; hence $L_P$ is a
      cycle and $[L_P]\ne0$.
\item The classes $[L_P]$, $P$ non-nesting, form a basis of
      $H_m(\B_mK;\Z)\cong\Z^{C_m}$.  In representation-theoretic terms,
      $\ker U_m(2m)$ is the integral Specht module $S^{(m,m)}$ of $\Sigma_{2m}$, the
      $h_P$ are its polytabloids, and the $h_P$ with $P$ non-nesting are its
      standard basis.
\item $X_P$ is a connected orientable $m$-dimensional pseudomanifold: every
      $(m-1)$-cube of $X_P$ lies in exactly two $m$-cubes of $X_P$.  The chain
      $L_P$ is $\iota$ of its fundamental class.
\item For $m=2$, $UD_2(K_{3,4})$ is a deformation retract of the configuration
      space, and $X_P$ is a closed orientable surface of genus $2$ in which the three
      pairs of points over the configurations $\{a_i,a_j\}$ are identified.
      It has $19$ vertices, $48$ edges and $24$ squares,
      so $\chi(X_P)=-5$; the vertex links are $4$-cycles at the $12$ configurations
      $\{a,b\}$, $6$-cycles at the $4$ configurations $\{b,b'\}$ with $b,b'$ in
      different pairs, and two disjoint $4$-cycles at the $3$ configurations
      $\{a,a'\}$.
\end{enumerate}
\end{theorem}

\begin{proof}
(1) Let $T\in\binom B{m+1}$.  If $T$ misses a pair of $P$, then by counting it
contains two pairs entirely, and every $m$-subset of $T$ contains a full pair, so
$(U_mh_P)(T)=0$.  Otherwise $T$ meets every pair and contains exactly one pair
$\{i_k,j_k\}$ entirely; the only $m$-subsets $F\subset T$ with $h_P(F)\ne0$ are
$T\setminus i_k$ and $T\setminus j_k$, whose values differ by the factor $-1$ from
the $k$-th pair.  So $(U_mh_P)(T)=0$.  Since $h_P\ne0$, $L_P=c_{h_P}\ne0$, and there
are no boundaries in the top degree.

(2) A tabloid of shape $(m,m)$ is determined by its first row $F\in\binom Bm$, so
$M^{(m,m)}=\Z^{\binom Bm}$, and the polytabloid of a tableau $t$ with columns
$\{i_k,j_k\}$, $e_t=\sum_{\sigma\in C_t}\operatorname{sgn}(\sigma)\{\sigma t\}$, is
$h_P$ for the perfect matching $P$ formed by the columns of $t$; $t$ is standard
exactly when $P$ is non-nesting.  By James \cite[Theorem~8.4]{James}, the standard
polytabloids are linearly independent over every field.  Hence the lattice
$\mathcal{L}$ spanned by the $h_P$, $P$ non-nesting, has rank $C_m$ and its reduction
modulo every prime is injective, so $\mathcal{L}$ is a saturated sublattice of
$\Z^{\binom Bm}$.  By (1), $\mathcal{L}\subset\ker U_m(2m)$, which is saturated of the
same rank $C_m$ by Theorem~\ref{thm:exacttors}.  Hence $\mathcal{L}=\ker U_m(2m)$, and
Proposition~\ref{prop:universal} transports the basis to $H_m(\B_mK;\Z)$.

(3) An $(m-1)$-cube of $X_P$ is a face of an $m$-cube $\gamma_M$: either a particle
comes to rest at $b\in B$, with the remaining crossings $M'$ covering $m-1$
vertices of $A$ and one vertex from each of the other $m-1$ pairs, and then the
$m$-cubes of $X_P$ containing it are $M'\cup\{ab\}$ and $M'\cup\{a'b\}$ for the two
free $A$-vertices $a,a'$; or a particle rests at $a\in A$, the crossings $M'$ cover
$m-1$ pairs, one pair $\{i_k,j_k\}$ is untouched, and the $m$-cubes of $X_P$
containing the face are $M'\cup\{ai_k\}$ and $M'\cup\{aj_k\}$, in which the
resting particle crosses into the untouched pair.  In both cases exactly two.  The coefficients
$\pm1$ of $L_P$ are compatible orientations, because $L_P$ is a cycle.  $X_P$ is
connected: the second type of face lets one particle switch between the two
vertices of its pair, and the first type lets the free $A$-vertex trade places with
any occupied $A$-vertex, and these moves connect all pairs $(\alpha,\sigma)$.
Finally $\iota$ sends the fundamental class $\sum_M\pm\gamma_M$ of $X_P$ to $L_P$
by definition.

(4) Abrams' theorem \cite{Abrams} applies to $UD_2(K_{3,4})$, since every cycle of
$K_{3,4}$ has length at least $4>3$.  Let $P=\{\{b_1,b_2\},\{b_3,b_4\}\}$.  The
$0$-cells of $X_P$ are the $3$ configurations $\{a,a'\}$, the $12$ configurations
$\{a,b\}$ and the $4$ configurations $\{b,b'\}$ with $b\in\{b_1,b_2\}$,
$b'\in\{b_3,b_4\}$; the $1$-cells are the $12\cdot4=48$ pairs (edge $ab$, resting
vertex outside $ab$ and outside the pair of $b$); the $2$-cells are the $24$
matchings $\{ab,a'b'\}$ with $b,b'$ in different pairs.  The squares containing
$\{a,b\}$ are $\{ab'',a'b\}$ with $b''$ in the other pair and $a'\ne a$: four,
forming a $4$-cycle.  The squares containing $\{b,b'\}$ are the six ordered pairs
$(a,a')$ with $a\ne a'$. Their link edges join the direction in which
$b$ moves to $a$ to the direction in which $b'$ moves to $a'$; the link
is therefore $K_{3,3}$ minus a perfect matching, hence a $6$-cycle.
The squares containing
$\{a,a'\}$ are $\{ab,a'b'\}$ with $b$ and $b'$ in different pairs: eight, and two
squares share an edge at $\{a,a'\}$ only if the particle at $a$ crosses into the
same pair in both; so the link is the disjoint union of two $4$-cycles, one for
each assignment of the two pairs to $a$ and $a'$.  Separating the two sheets at
these three vertices gives a closed surface with $22$ vertices, $48$ edges and $24$
faces, connected by (3), orientable by (3), of Euler characteristic $-2$, hence
of genus $2$.
\end{proof}

\begin{remark}[The universal construction]\label{rem:universal}
Proposition~\ref{prop:universal} organizes the matching chains and theta cycles of this section.  The
matching chains $c_h$, $h\colon\binom B{r-1}\to R$, span the top chains that are
candidates for cycles, and the chains $z_g$, $g\colon\binom B{r}\to R$, are the
signed sums of restricted theta classes.  The boundary map between them is the
inclusion matrix $U_{r-1}$:
\[
\begin{array}{lll}
h\in\ker U_{r-1} &\leadsto& \text{top cycles; loop basis when }r=m+1,\\
h\equiv1,\ r=q &\leadsto& \partial w_m=qz_m,\\
h=\delta_{B\setminus T} &\leadsto&
 \partial w_T=\sum_{b\in T}z_m^{(0)}|_{T\setminus b},\\
g=\delta_T &\leadsto& \text{individual restricted theta cycles},\\
r=0 &\leadsto& \text{primitive generalized theta classes}.
\end{array}
\]
In the two extreme cases the crossings join the two parts: for $r=0$
the resting particle lies in the larger part $A$, whereas for $r=m+1$
the free $B$-vertices lie in the part of size $2m$.  The classical loop class ($m=1$, $r=2$) and the classical star class
($m=2$, $r=0$) sit at the two ends of this picture, and the torsion classes of
Theorem~B are what happens in between, when $U_{r-1}$ is neither injective nor
surjective modulo $r$.
\end{remark}

\section{Symmetry and disappearance of torsion}\label{sec:symmetry}

We return to the matching-function model of Section~\ref{sec:matching}
and allow both parts of the complete bipartite graph to vary.
The slot-permutation action relates top cycles to chessboard homology,
giving the rational representation decomposition of Theorem~D.
A modular argument then proves the odd-primary vanishing assertion
of that theorem. A $2$-local calculation determines the remaining
torsion for three particles and $a,b\ge5$. Combined with the exact torsion formula of
Section~\ref{sec:exacttors} and the explicit classes of
Section~\ref{sec:explicit}, these results describe contrasting
behaviours under graph inclusions.

\subsection{Slot modules and rational top homology}\label{sec:representations}
We now allow both part sizes to vary. Let $m\ge1$ and $a,b\ge m$,
orient the edges of $K_{a,b}$ from $[a]$ to $[b]$, and let $\Bbbk$ be
a field. The vertex-permutation group $\Sigma_a\times\Sigma_b$ acts
on homology. A second action, of $\Sigma_m$ permuting the slots of an
ordered matching, connects Proposition~\ref{prop:Zcycles} with chessboard
complexes. The two actions will be kept distinct.

Let $I_m(a;\Bbbk)$ be the space of functions on ordered injections
$\mathbf x=(x_1,\ldots,x_m):[m]\hookrightarrow[a]$. Define
$D_m(a;\Bbbk)\subset I_m(a;\Bbbk)$ by the equations
\begin{equation}\label{rep:eq:slot}
 \sum_{t\in[a]\setminus\{x_j:j\ne i\}}
 f(x_1,\ldots,x_{i-1},t,x_{i+1},\ldots,x_m)=0
 \qquad(1\le i\le m),
\end{equation}
for every choice of distinct values of the other $m-1$ coordinates.
Vertex permutations give an action of $\Sigma_a$; permuting the slots gives
a commuting action of $\Sigma_m$. These are the ordinary permutation actions
on functions, without an additional sign twist.

\begin{proposition}\label{rep:prop:arbitrary}
For every field $\Bbbk$ there is a natural $\Sigma_a\times\Sigma_b$-equivariant
isomorphism
\begin{equation}\label{rep:eq:allfields}
 H_m(\B_mK_{a,b};\Bbbk)\cong
 \operatorname{Hom}_{\Bbbk\Sigma_m}\bigl(\operatorname{sgn},D_m(a;\Bbbk)\otimes D_m(b;\Bbbk)\bigr).
\end{equation}
Here $\Sigma_m$ acts diagonally on the tensor product. Equivalently, the right
side is
\begin{equation}\label{rep:eq:homdual}
 \operatorname{Hom}_{\Bbbk\Sigma_m}\bigl(D_m(a;\Bbbk)^*\otimes\operatorname{sgn},D_m(b;\Bbbk)\bigr).
\end{equation}
\end{proposition}
\begin{proof}
An ordered $m$-matching is specified by two injections
$\mathbf x:[m]\hookrightarrow[a]$ and $\mathbf y:[m]\hookrightarrow[b]$,
with edges $(x_i,y_i)$ in slot order. Proposition~\ref{prop:Zcycles} identifies top homology with functions $F(\mathbf x,\mathbf y)$ that are
alternating under simultaneous slot permutations and satisfy the vertex
sum relations.

At a fixed first-part vertex, a vertex sum varies one coordinate of
$\mathbf y$, keeping all of $\mathbf x$ and the other coordinates of
$\mathbf y$ fixed. Thus these relations give
$I_m(a;\Bbbk)\otimes D_m(b;\Bbbk)$. The second-part relations similarly give
$D_m(a;\Bbbk)\otimes I_m(b;\Bbbk)$; reversing the new edge contributes a common
minus sign and does not change the zero-sum equation. The intersection
of these two subspaces is $D_m(a;\Bbbk)\otimes D_m(b;\Bbbk)$, as one sees by choosing
vector-space complements to the two $D_m$'s. Finally, alternating slot
permutations pick out precisely the sign-equivariant vectors, giving
\eqref{rep:eq:allfields}.

For finite-dimensional representations, tensor--Hom adjunction gives
\eqref{rep:eq:homdual}. Neither step uses averaging by $1/m!$, semisimplicity,
or exactness of the invariants functor. The proof therefore applies also
when $\operatorname{char}\Bbbk$ divides $m!$.
\end{proof}

Let $\Delta_{m,a}$ denote the chessboard complex whose faces are matchings
in $K_{m,a}$. Since $a\ge m$, it has dimension $m-1$.
Orient each top simplex by listing its rooks in increasing row order.
Its simplicial boundary is zero exactly when \eqref{rep:eq:slot} holds:
an $(m-2)$-simplex omits a unique row, and its possible completions vary
the column in that row. Permuting the rows changes this orientation by
the sign of the row permutation. Consequently
\begin{equation}\label{rep:eq:chessboard}
 \widetilde H_{m-1}(\Delta_{m,a};\Bbbk)
 \cong D_m(a;\Bbbk)\otimes\operatorname{sgn}_{\Sigma_m}
\end{equation}
as $\Sigma_a\times\Sigma_m$-modules, where the sign twist concerns only the
slot group. For $m=1$, reduced degree-zero chains give the same assertion.

Write $S^\rho$ for the rational Specht module of partition shape $\rho$,
and $\lambda'$ for the conjugate partition of $\lambda$. The external
tensor product for two distinct symmetric groups is denoted by $\boxtimes$.
Friedman and Hanlon \cite{FH} compute the combinatorial Laplacian of
$\Delta_{m,a}$ on each $\Sigma_a\times\Sigma_m$-isotypic component of the
top chain module, and the components with eigenvalue zero form the top
homology.  Reading off their result in the top degree gives
\begin{equation}\label{rep:eq:FH}
 \widetilde H_{m-1}(\Delta_{m,a};\Q)
 \cong
 \bigoplus_{\substack{\lambda\vdash m\\\lambda_1\le a-m}}
 S^{(a-m,\lambda)}\boxtimes S^{\lambda'}
\end{equation}
in the group order $\Sigma_a\times\Sigma_m$; in the notation of \cite{FH}
the summation index runs over the $(a-m,0)$-super-rectangular partitions
$\lambda$ of $m$, which are exactly those with $\lambda_1\le a-m$, and
$(a-m,\lambda)$ and $\lambda'$ are their minimally horizontally built
partitions of $a$ and $m$. The inequality makes $(a-m,\lambda)$ a partition
of $a$. If $a=m$, the sum is empty.
Twisting the second factor in \eqref{rep:eq:FH} by sign gives
\begin{equation}\label{rep:eq:Ddecomp}
 D_m(a;\Q)\cong
 \bigoplus_{\substack{\lambda\vdash m\\\lambda_1\le a-m}}
 S^{(a-m,\lambda)}\boxtimes S^\lambda.
\end{equation}

\begin{theorem}\label{rep:thm:rational}
For $m\ge1$ and $a,b\ge m$, there is a multiplicity-free decomposition
\begin{equation}\label{rep:eq:main}
 H_m(\B_mK_{a,b};\Q)\cong
 \bigoplus_{\substack{\lambda\vdash m\\
 \lambda_1\le a-m,\ \lambda'_1\le b-m}}
 S^{(a-m,\lambda)}\boxtimes S^{(b-m,\lambda')}.
\end{equation}
\end{theorem}
\begin{proof}
Insert \eqref{rep:eq:Ddecomp} for both factors in \eqref{rep:eq:allfields}.
The slot-group contribution from a pair $(\lambda,\mu)$ is
\[
 \operatorname{Hom}_{\Q\Sigma_m}(\operatorname{sgn},S^\lambda\otimes S^\mu).
\]
Rational Specht modules are absolutely irreducible and self-dual, and
$S^\lambda\otimes\operatorname{sgn}\cong S^{\lambda'}$. Schur's lemma therefore says
that the displayed space has dimension one if $\mu=\lambda'$ and is
zero otherwise. This gives \eqref{rep:eq:main}. Distinct $\lambda$ give
distinct first-factor shapes, so there are no repeated summands.
\end{proof}

\begin{corollary}\label{rep:cor:rank}
Let $f^\rho$ be the number of standard Young tableaux of shape $\rho$.
Then
\begin{equation}\label{rep:eq:rank}
 \rk H_m(\B_mK_{a,b};\Z)=
 \sum_{\substack{\lambda\vdash m\\
 \lambda_1\le a-m,\ \lambda'_1\le b-m}}
 f^{(a-m,\lambda)}f^{(b-m,\lambda')}.
\end{equation}
Each factor is evaluated by the hook-length formula
\[
 f^\rho=\frac{|\rho|!}{\prod_{x\in\rho}h(x)}.
\]
Moreover,
\begin{equation}\label{rep:eq:nonzero}
 H_m(\B_mK_{a,b};\Q)\ne0
 \quad\Longleftrightarrow\quad (a-m)(b-m)\ge m.
\end{equation}
\end{corollary}
\begin{proof}
The weight-$m$ graph configuration chain complex is free abelian and has
no chains above degree $m$. Its top homology is therefore free abelian.
Its rank equals its rational dimension, yielding \eqref{rep:eq:rank}.
The allowable diagrams are exactly the partitions of $m$ contained in
a rectangle with $a-m$ columns and $b-m$ rows. Such a diagram exists
exactly when the rectangle has at least $m$ boxes.
\end{proof}

\begin{remark}
The irreducible summands are indexed by partitions, not by individual
tableaux. After choosing standard Specht bases, their basis vectors are
indexed by pairs of standard Young tableaux of the indicated shapes.
Formula \eqref{rep:eq:rank} determines the integral top homology as an abstract
abelian group, but a rational decomposition alone does not determine its
integral equivariant lattice or give an integral direct-sum decomposition.
\end{remark}

\begin{remark}
For $a=m+1$, the only possible partition is $\lambda=(1^m)$.
Thus for $b\ge2m$ the rational top homology is
$\operatorname{sgn}_{\Sigma_{m+1}}\boxtimes S^{(b-m,m)}$.
At $b=2m$ this recovers the Specht module $S^{(m,m)}$ and its Catalan
rank, together with the sign action on the first part; see
Section~\ref{sec:loopclasses} for the integral polytabloid construction.
\end{remark}

\begin{example}\label{ex:rep-k55}
For $(m,a,b)=(3,5,5)$, the only allowed partition is $(2,1)$, so
\begin{equation}\label{rep:eq:25Q}
 H_3(\B_3K_{5,5};\Q)
 \cong S^{(2,2,1)}\boxtimes S^{(2,2,1)},
 \qquad \dim_\Q H_3=5\cdot5=25.
\end{equation}
Indeed $f^{(2,2,1)}=5!/(4\cdot2\cdot3\cdot1\cdot1)=5$.

\end{example}

\begin{corollary}[Polynomial top rank]\label{cor:rep-polynomial}
For fixed $m$, there is a polynomial $P_m(a,b)\in\Q[a,b]$ of degree $m$
in each variable such that
\[
\rk H_m(\B_mK_{a,b};\Z)=P_m(a,b)\qquad(a,b\ge2m).
\]
Its term of total degree $2m$ is $a^mb^m/m!$.
\end{corollary}
\begin{proof}
In this range every partition of $m$ occurs in \eqref{rep:eq:rank}.
For fixed $\lambda\vdash m$, the hook-length formula makes
$f^{(a-m,\lambda)}$ a polynomial in $a$ of degree $m$, with leading
coefficient $f^\lambda/m!$. Since $f^{\lambda'}=f^\lambda$ and
$\sum_{\lambda\vdash m}(f^\lambda)^2=m!$, summing the products gives
both assertions.
\end{proof}

\subsection{Odd-primary vanishing}\label{sec:modular-vanishing}
We now assume $m\ge2$ and $a,b\ge2m-1$. Write $G=\Sigma_m$ and
$\varepsilon=\operatorname{sgn}_G$. The bounds on the two parts allow us to use
shellability; the additional bound $p\ge m$ will make the chain modules
projective.

\medskip\noindent\textbf{The join and its chain modules.}\enspace
Put
\[
 X=\Delta_{m,a}*\Delta_{m,b},\qquad n=2m-1.
\]
The group $G$ permutes the row labels simultaneously in the two factors;
the two sets of columns are fixed. By Ziegler's theorem \cite{Ziegler},
both chessboard complexes are shellable, with reduced homology concentrated
in degree $m-1$, and that integral homology is free. The join formula gives,
over any field $\Bbbk$,
\[
 \widetilde H_i(X;\Bbbk)=0\quad(i\ne n),\qquad
 K_{\Bbbk}:=\widetilde H_n(X;\Bbbk)
 \cong\widetilde H_{m-1}(\Delta_{m,a};\Bbbk)
 \otimes\widetilde H_{m-1}(\Delta_{m,b};\Bbbk).
\]
The join identification is $G$-equivariant: $G$ preserves both factors,
so no interchange of the factors is involved. By \eqref{rep:eq:chessboard},
the two slot-sign twists cancel, giving
\begin{equation}\label{rep:eq:joinidentity}
 H_m(\B_mK_{a,b};\Bbbk)\cong\operatorname{Hom}_{\Bbbk G}(\varepsilon,K_{\Bbbk}).
\end{equation}

Every nonempty simplex of $X$ is a pair of rook placements, not both empty.
If their occupied row sets have union $U$, its stabilizer is the full
symmetric group on the unused rows, hence isomorphic to $\Sigma_{m-|U|}$.
Indeed, a permutation preserving the simplex must fix the row of each
rook because it fixes that rook's column and preserves the two factors.
This stabilizer fixes the simplex pointwise and therefore acts trivially
on its orientation. Consequently each simplicial chain module $C_i(X;\Bbbk)$
is a direct sum of modules induced from the trivial representation of
such stabilizers. If $p\ge m$ and $\Bbbk=\F_p$, all their orders are prime to
$p$, since $m-|U|\le m-1$. Thus every $C_i(X;\Bbbk)$ is projective.

Writing $P_i=C_i(X;\Bbbk)$, the augmented chain complex is the exact sequence
\begin{equation}\label{rep:eq:joinresolution}
 0\longrightarrow K_{\Bbbk}\longrightarrow P_n\longrightarrow P_{n-1}
 \longrightarrow\cdots\longrightarrow P_0\longrightarrow \Bbbk
 \longrightarrow0.
\end{equation}
For $\Bbbk=\Q$, or $\Bbbk=\F_p$ with $p\ge m$, it is the beginning of a projective resolution of the trivial module:
one can extend it on the left by a projective resolution of $K_{\Bbbk}$.

\medskip\noindent\textbf{The Euler characteristic correction.}\enspace
\begin{lemma}\label{rep:lem:correction}
If $p\ge m$ is odd, then
\begin{align}
 &\dim_{\F_p}H_m(\B_mK_{a,b};\F_p)
 -\dim_\Q H_m(\B_mK_{a,b};\Q)\notag\\
 &\hspace{2em}=\sum_{i=0}^{2m-2}(-1)^i
 \dim_{\F_p}H_i(G;\F_p^{\varepsilon}).\label{rep:eq:correction}
\end{align}
\end{lemma}
\begin{proof}
For a field $\Bbbk$ let $A_i(\Bbbk)=\operatorname{Hom}_{\Bbbk G}(\varepsilon,C_i(X;\Bbbk))$, with the
induced chain differential, and omit the augmentation from this complex.
Its top kernel is $\operatorname{Hom}_{\Bbbk G}(\varepsilon,K_{\Bbbk})$ by left exactness, so its
top homology has the dimension in \eqref{rep:eq:joinidentity}.

For a projective $\Bbbk G$-module $M$, the norm map
\[
 (\varepsilon\otimes M)_G\longrightarrow(\varepsilon\otimes M)^G,
 \qquad [v]\longmapsto\sum_{g\in G}gv,
\]
is an isomorphism, natural in $M$. This is checked for the regular module
and hence for its direct summands; it does not require division by $|G|$.
Since $(\varepsilon\otimes M)^G\cong\operatorname{Hom}_{\Bbbk G}(\varepsilon,M)$,
these maps identify the complex $A_\bullet(\Bbbk)$ with the sign-twisted
coinvariant complex of $P_\bullet$ when $\Bbbk=\F_p$.
By \eqref{rep:eq:joinresolution}, its homology in degrees $0\le i<n$ is
$H_i(G;\Bbbk^{\varepsilon})$, while in the top degree the truncated complex has
kernel given by \eqref{rep:eq:joinidentity}.

Set $c_i=\dim_{\Bbbk} A_i(\Bbbk)$. These numbers are the same over $\Q$ and
over every field of odd characteristic. In fact, Frobenius reciprocity
shows that each simplex orbit contributes one if the sign character is
trivial on its stabilizer, and zero otherwise. Here the stabilizer is
$\Sigma_t$, and the contribution is one exactly when $t\le1$.
Euler characteristic, using that $n=2m-1$ is odd, now gives
\[
 \sum_{i=0}^n(-1)^ic_i
 =\sum_{i=0}^{n-1}(-1)^i\dim_{\Bbbk} H_i(G;\Bbbk^{\varepsilon})
 -\dim_{\Bbbk} H_m(\B_mK_{a,b};\Bbbk)
\]
for $\Bbbk=\F_p$.
Over $\Q$, exactness of $\operatorname{Hom}_{\Q G}(\varepsilon,-)$ and
$\operatorname{Hom}_{\Q G}(\varepsilon,\Q)=0$ show that $A_\bullet(\Q)$ has homology
only at its top. Thus its Euler characteristic is
$-\dim_\Q H_m(\B_mK_{a,b};\Q)$. Subtraction proves the formula.
\end{proof}

\begin{lemma}\label{rep:lem:signhomology}
For an odd prime $p$ and $i\ge0$,
\[
 \dim_{\F_p}H_i(\Sigma_p;\F_p^{\operatorname{sgn}})=
 \begin{cases}
 1,&i\equiv p-2\text{ or }p-1\pmod{2(p-1)},\\
 0,&\text{otherwise}.
 \end{cases}
\]
\end{lemma}
\begin{proof}
We recall the normalizer computation underlying this standard formula;
we use the stable-elements method for group cohomology
\cite[Chapter~XII]{CE}; see also \cite[Chapter~VI]{AM} and Swan's work
on groups with cyclic Sylow $p$-subgroups \cite{Swan}.
Let $C_p\le\Sigma_p$ be generated by a $p$-cycle. Its normalizer is
$C_p\rtimes\F_p^\times$. Distinct subgroups of order $p$ have trivial
intersection, so the stable-elements description identifies positive-degree
cohomology with coefficients in sign with the normalizer-invariant part
of $H^*(C_p;\F_p)\otimes\operatorname{sgn}$.

Write $H^*(C_p;\F_p)=\Lambda(u)\otimes\F_p[v]$, with $|u|=1$ and
$|v|=2$. A multiplier $c\in\F_p^\times$ acts on each of $u,v$ by $c$
(or by $c^{-1}$ under the opposite action convention), and its sign
as a permutation of $\F_p$ is $c^{(p-1)/2}$. Thus the invariant monomials
have degrees $p-2$ and $p-1$ modulo $2(p-1)$, one in each such degree.
Degree-zero cohomology is zero because sign is nontrivial.
Finally, over a field the dual of group homology with coefficients in
a finite-dimensional module is cohomology with coefficients in its dual;
sign is self-dual. This gives the stated homology dimensions.
\end{proof}

\begin{theorem}\label{rep:thm:generalvanishing}
Let $m\ge2$, $a,b\ge2m-1$, and let $p\ge m$ be an odd prime. Then
\[
 \dim_{\F_p}H_m(\B_mK_{a,b};\F_p)
 =\dim_\Q H_m(\B_mK_{a,b};\Q),
 \qquad H_{m-1}(\B_mK_{a,b};\Z)[p]=0.
\]
In particular the latter integral group has no $p$-primary torsion.
\end{theorem}
\begin{proof}
If $p>m$, positive-degree group homology of $\Sigma_m$ over $\F_p$
vanishes by semisimplicity. Degree zero with sign coefficients also
vanishes because $m\ge2$ and $p$ is odd. Thus the correction in
\eqref{rep:eq:correction} is zero.

If $p=m$, Lemma~\ref{rep:lem:signhomology} shows that the only nonzero
terms in the range $0\le i\le2p-2$ occur in degrees $p-2$ and $p-1$.
Both have dimension one and their signs in the Euler sum are opposite.
The correction is again zero.

The integral top homology is free. Hence the universal coefficient
sequence
\[
 0\longrightarrow H_m(\B_mK_{a,b};\Z)\otimes\F_p
 \longrightarrow H_m(\B_mK_{a,b};\F_p)
 \longrightarrow H_{m-1}(\B_mK_{a,b};\Z)[p]
 \longrightarrow0
\]
and equality of dimensions prove the torsion assertion.
\end{proof}

\subsection{The \texorpdfstring{$2$}{2}-primary torsion for three particles}\label{sec:two-primary-three}
For three particles the slot-module identity determines the full torsion
for all $a,b\ge5$. We work over $\Z/2^\ell$ for every $\ell\ge1$ to
distinguish order-$2$ torsion from higher powers of $2$.

Put $R=\Z_{(2)}$, $G=\Sigma_3$, and let $\varepsilon$ be its sign module.
Let $P=R[G/C_3]$ be the regular module of the quotient $C_2$, inflated to
$G$, and let $V=\{(x_1,x_2,x_3)\in R^3:x_1+x_2+x_3=0\}$ with the
permutation action. For $n\ge5$ put
\[
 u_n=\binom n3-\binom n2,\qquad
 v_n=\frac{n(n-2)(n-4)}3,
\]
and define $D_n=D_3(n;R)$ by the slot-sum equations
\eqref{rep:eq:slot} over $R$.

\begin{lemma}\label{rep:lem:two-local-slot}
For $n\ge5$, there is an isomorphism of $RG$-modules
\[
 D_n\cong\varepsilon^{\oplus(n-1)}\oplus P^{\oplus u_n}
          \oplus V^{\oplus v_n}.
\]
This assertion concerns the slot group only. In particular, the summands
$P$ are retained as integral modules, not split into trivial and sign
summands as they would be over $\Q$.
\end{lemma}
\begin{proof}
For $c=(123)$ the element $e=(1+c+c^2)/3$ is a central idempotent of
$RG$, and $eRG\cong A:=R[C_2]$. Put
$T=\widetilde H_2(\Delta_{3,n};R)$, so that
$D_n\cong T\otimes\varepsilon$ by \eqref{rep:eq:chessboard}.
The chessboard complex has $3n$ vertices, $3n(n-1)$ edges, and
$n(n-1)(n-2)$ triangles. The slot group acts freely on the edges and
triangles; each of the $n$ vertex orbits is the three-point permutation
module. Applying $e$ to the augmented chain sequence therefore gives
\begin{equation}\label{rep:eq:two-local-chess}
 0\longrightarrow eT\longrightarrow A^{\binom n3}
 \longrightarrow A^{\binom n2}\longrightarrow R^n
 \longrightarrow R\longrightarrow0.
\end{equation}
Exactness follows from the concentration of reduced chessboard homology
in degree two. The action on $R^n$ is trivial and its augmentation is
summation, with kernel $R^{n-1}$.

Here $A$ is local, with maximal ideal $(2,t-1)$ for the generator $t$ of
$C_2$. An epimorphism $A^N\to R^d$ can be put, by a change of basis,
in the form of $d$ coordinate augmentations followed by zero coordinates;
its kernel is $\varepsilon^{\oplus d}\oplus A^{N-d}$.
The same assertion after twisting by sign gives kernel
$R^d\oplus A^{N-d}$ for an epimorphism $A^N\to\varepsilon^{\oplus d}$.
Apply the first assertion to the map onto $R^{n-1}$ in
\eqref{rep:eq:two-local-chess}. Split off the resulting free $A$-summand
from the preceding epimorphism; its complementary source is free since
$A$ is local. Applying the sign-twisted assertion to the remaining map
then gives
\[
 eT\cong R^{n-1}\oplus A^{\binom n3-\binom n2}.
\]
Since $A\otimes\varepsilon\cong A$, this proves
$eD_n\cong\varepsilon^{\oplus(n-1)}\oplus P^{\oplus u_n}$.

The action on $V$ identifies $(1-e)RG$ with $M_2(R)$. Indeed, in the
basis $(1,-1,0),(0,1,-1)$, the matrices of $c$ and $t=(12)$ are
\[
 c=\begin{pmatrix}0&-1\\1&-1\end{pmatrix},\qquad
 t=\begin{pmatrix}-1&1\\0&1\end{pmatrix}.
\]
The four matrices $I,c,t,ct$ have coordinate determinant $\pm3$, a unit
of $R$, so the action map is surjective. Both algebras are free of rank
four, hence the map is an isomorphism. Matrix units consequently identify
every $R$-free module in this block with a sum of copies of $V$.
The rank of $T$ is $n^3-6n^2+8n-1$, by its Euler characteristic.
Subtracting the rank $(n-1)+2u_n$ of $eD_n$ leaves $2v_n$, proving
$(1-e)D_n\cong V^{\oplus v_n}$.
\end{proof}

\begin{proposition}\label{rep:prop:two-primary-three}
For $a,b\ge5$ and $S=\Z/2^\ell$, $\ell\ge1$,
\[
 H_3(\B_3K_{a,b};S)
 \cong S^{\beta_3(a,b)}\oplus S[2]^{(a-1)(b-1)},
 \qquad S[2]=\{x\in S:2x=0\},
\]
where $\beta_3(a,b)=\rk H_3(\B_3K_{a,b};\Z)$ is given by
\eqref{rep:eq:rank}. Explicitly,
\[
 \beta_3(a,b)=(a-1)u_b+(b-1)u_a+2u_au_b+v_av_b.
\]
Consequently the full torsion subgroup is
\[
 \Tors H_2(\B_3K_{a,b};\Z)\cong(\Z/2)^{(a-1)(b-1)}.
\]
\end{proposition}
\begin{proof}
For $n\ge5$, the augmented simplicial chain sequence of $\Delta_{3,n}$,
with its top kernel included, is split exact as a sequence of abelian
groups: the reduced integral homology is free and concentrated in degree
two. Thus formation of the top kernel commutes with base change, and
its inclusion in the top chain module is split over the coefficient ring.
Using \eqref{rep:eq:chessboard}, the tensor-intersection argument of
Proposition~\ref{rep:prop:arbitrary} therefore applies over $S$, giving
\[
 H_3(\B_3K_{a,b};S)
 \cong\operatorname{Hom}_{SG}
 (\varepsilon_S,D_{a,S}\otimes_S D_{b,S}),
 \qquad D_{n,S}=D_n\otimes_R S.
\]

Apply Lemma~\ref{rep:lem:two-local-slot}. For the summands in the
$C_3$-trivial block, we have
\[
 \varepsilon_S\otimes\varepsilon_S\cong S,\qquad
 \varepsilon_S\otimes P_S\cong P_S,\qquad
 P_S\otimes P_S\cong P_S^{\oplus2}.
\]
The sign-equivariant vectors in $S$ form $S[2]$, whereas those in $P_S$
form a free copy of $S$, generated by the difference of its two basis
vectors. Tensor products of $V_S$ with either $\varepsilon_S$ or $P_S$
have no sign-equivariant vectors, since their $C_3$-invariant part is zero.
Finally, the perfect wedge pairing and $\det V_S=\varepsilon_S$ give
\[
 \operatorname{Hom}_{SG}(\varepsilon_S,V_S\otimes_S V_S)
 \cong\operatorname{End}_{SG}(V_S)\cong S;
\]
the last equality follows because the action spans $M_2(S)$.
Thus the only nonfree contribution is $S[2]^{(a-1)(b-1)}$, and the
free contributions have the asserted total rank $\beta_3(a,b)$.
The same calculation over $\Q$ identifies this rank with
\eqref{rep:eq:rank}.

The integral top homology is free of rank $\beta_3(a,b)$. The split
universal coefficient sequence and cancellation of the common finite
summand $S^{\beta_3(a,b)}$ give
\[
 \Tor\bigl(H_2(\B_3K_{a,b};\Z),S\bigr)
 \cong S[2]^{(a-1)(b-1)}\qquad(\ell\ge1).
\]
Taking $\ell=1$ counts the $2$-primary cyclic factors, and $\ell=2$
excludes factors of order $4$ or higher. Every odd prime satisfies
$p\ge3$, so Theorem~\ref{rep:thm:generalvanishing} excludes all odd
 torsion, completing the proof.
\end{proof}

\begin{remark}[External products]\label{rem:external-products}
Following \cite[Definition~5.10]{ADK} and \cite[Section~3.1]{AK},
for disjoint subgraphs $\Gamma_1,\Gamma_2\subset\Gamma$ we call the map
\[
 H_i(\B_{n_1}\Gamma_1;\Z)\otimes H_j(\B_{n_2}\Gamma_2;\Z)
 \longrightarrow H_{i+j}(\B_{n_1+n_2}\Gamma;\Z)
\]
obtained by the homology cross product followed by the union map
$(\xi_1,\xi_2)\mapsto\xi_1\cup\xi_2$ the \emph{external product}.
We denote its value by $\alpha\times\beta$; in particular,
$[C]\times\sigma$ denotes the product of a loop class and a star class
supported in disjoint subgraphs. The construction depends on the chosen
disjoint supports.
\end{remark}

The multiplicity $(a-1)(b-1)$ in Proposition~\ref{rep:prop:two-primary-three}
is the first Betti number of $K_{a,b}$. The following geometric description
explains this equality.

\begin{corollary}[Loop--star generators]\label{cor:loop-star-torsion}
Let $\Gamma=K_{a,b}$ with $a,b\ge5$, and enumerate its parts as
$A=\{a_0,\ldots,a_{a-1}\}$ and $B=\{b_0,\ldots,b_{b-1}\}$.
For $1\le i<a$ and $1\le j<b$, let
\[
 C_{ij}=a_0b_0a_ib_ja_0
\]
be the indicated $4$-cycle, and let $\Lambda_{ij}\cong K_{a-2,b-2}$ be
the induced subgraph on the vertices outside $C_{ij}$. Choose a star
class $\sigma_{ij}\in H_1(\B_2\Lambda_{ij};\Z)$. Then
\[
 t_{ij}=[C_{ij}]\times\sigma_{ij}
 \in H_2(\B_3\Gamma;\Z)
\]
form an $\F_2$-basis of $\Tors H_2(\B_3\Gamma;\Z)$.
In particular, every torsion class is a sum of loop--star products.
There is a canonical isomorphism
\[
 \Theta\colon\Tors H_2(\B_3\Gamma;\Z)
 \xrightarrow{\ \cong\ }H_1(\Gamma;\F_2)
\]
sending $t_{ij}$ to $[C_{ij}]$ modulo $2$.
\end{corollary}
\begin{proof}
The cycles $C_{ij}$ form an integral basis of $H_1(\Gamma;\Z)$: they
are the fundamental cycles for the spanning tree consisting of all
edges incident to $a_0$ or $b_0$. Each $\Lambda_{ij}$ is $3$-connected
and non-planar, so its star classes have order $2$
\cite[Lemmas~C.4 and~C.9]{ADK}. Since $C_{ij}$ and $\Lambda_{ij}$ are
disjoint, the union map above is defined and $2t_{ij}=0$.
It remains to prove that these products are independent.

For a loop in $\B_3\Gamma$, summing the trajectories of its three
particles gives an integral $1$-cycle in $\Gamma$, even when the
particles are permuted. This construction is invariant under homotopy
and additive under concatenation, and defines the total trace
\[
 \operatorname{tr}\colon H_1(\B_3\Gamma;\Z)\longrightarrow H_1(\Gamma;\Z).
\]
Let $\eta\in H^1(\B_3\Gamma;\F_2)$ be the character recording the
parity of the permutation of the three particles, obtained from the
ordered configuration covering. For $\alpha\in H^1(\Gamma;\F_2)$,
write $\operatorname{tr}^*\alpha$ for the degree-one cohomology class
obtained by composition with total trace.

Represent $t_{ij}$ by a torus whose first circle moves one particle
around $C_{ij}$ and whose second circle exchanges two particles in
the chosen star. Along the first circle, total trace is $[C_{ij}]$
and permutation parity is zero. Along the second, total trace is zero,
because the trajectories lie in a tree, and permutation parity is one.
Consequently, writing $\rho_2$ for reduction of homology modulo $2$,
\begin{equation}\label{eq:loop-star-pairing}
 \bigl\langle\operatorname{tr}^*\alpha\smile\eta,\rho_2(t_{ij})\bigr\rangle
 =\langle\alpha,[C_{ij}]\rangle.
\end{equation}
Taking the cohomology basis dual to the $C_{ij}$ proves that the
$(a-1)(b-1)$ products are nonzero and independent. They therefore form
a basis of the torsion subgroup by Proposition~\ref{rep:prop:two-primary-three}.

For a torsion class $t$, define $\Theta(t)$ by the identities
\[
 \langle\alpha,\Theta(t)\rangle
 =\bigl\langle\operatorname{tr}^*\alpha\smile\eta,\rho_2(t)\bigr\rangle
 \quad\text{for all }\alpha\in H^1(\Gamma;\F_2).
\]
The perfect pairing over $\F_2$ makes this definition unique.
Equation~\eqref{eq:loop-star-pairing} shows that $\Theta$ sends the
displayed basis to the cycle basis, hence is an isomorphism. Its
definition uses no choices of a spanning tree or star representatives.
\end{proof}

\subsection{Behaviour under graph inclusions}\label{sec:inclusion-behaviour}
The vanishing theorem contrasts with the torsion formula of
Section~\ref{sec:exacttors}. Together they show that graph inclusions
can destroy odd torsion when the first part is enlarged. In contrast,
keeping that part fixed and enlarging only the second part preserves
the order of the explicit class on $K_{m+1,m+q-1}$ for every $m\ge q\ge2$.
\begin{corollary}\label{rep:cor:allprimecounterexample}
For each odd prime $p$, the inclusion
\[
 K_{p+1,\,2p-1}\hookrightarrow K_{2p-1,\,2p-1}
\]
is a counterexample to upward persistence of the existence of $p$-torsion
in degree $p-1$ with $p$ particles. More precisely, the source homology
has a $\Z/p$ direct summand, whereas
\[
 H_{p-1}(\B_pK_{2p-1,\,2p-1};\Z)[p]=0.
\]
The induced map annihilates the entire $p$-primary torsion of the source.
\end{corollary}
\begin{proof}
Apply Theorem~\ref{thm:exacttors} with
$m=p$ and $r=p$. Its $j=p$ summand is $\Z/p$, with multiplicity
$\binom{2p-1}{0}-\binom{2p-1}{-1}=1$.
Apply Theorem~\ref{rep:thm:generalvanishing} to the target with $m=p$ and
$a=b=2p-1$. A homomorphism sends a $p$-power-torsion element to a
$p$-power-torsion element, so the final assertion follows as well.
\end{proof}

The target conclusion concerns the specified prime $p$; it does not
exclude torsion at smaller primes. The source assertion uses Theorem~\ref{thm:exacttors}; the rational representation formula alone does not detect torsion.

For $p=3$, Proposition~\ref{rep:prop:two-primary-three} gives the more
precise target description
$\Tors H_2(\B_3K_{5,5};\Z)\cong(\Z/2)^{16}$.

\begin{proposition}[Persistence under enlargement of the second part]\label{prop:second-part-persistence}
Let $q\ge2$, $m\ge q$, and $r\ge q$. Under the inclusion
\[
\iota_r\colon K_{m+1,m+q-1}\hookrightarrow K_{m+1,m+r-1}
\]
obtained by adding vertices to the second part, the image of
$[z_m^{(q)}]$ has order exactly $q$.
If $m=q=p$ is prime, the induced map on $p$-primary torsion is an
isomorphism; both groups are $\Z/p$.
\end{proposition}
\begin{proof}
Let $B_0\subseteq B$ be the second part of the source, so
$|B_0|=m+q-1$ and $|B|=m+r-1$. Use the same ordering of $A$ in both graphs.
The inclusion of chain complexes sends each restricted theta cycle
$z_m^{(0)}|_{B'}$ of the source to the restricted theta cycle of the
target with the same $B'$. By \eqref{eq:zsum}, the image of the source class is
\[
\Phi([g]),\qquad
g=\sum_{B'\in\binom{B_0}{m-1}}\delta_{B\setminus B'}.
\]
It is killed by $q$, since the source class has order $q$ by
Theorem~\ref{thm:explicitgen}. We prove that its order is divisible
by every maximal prime power $d=p^a$ dividing $q$.

Choose $Y\subset B_0$ with $|Y|=m-d$, and define
\[
\lambda_Y(f)=\sum_{\substack{T\in\binom B{r}\\Y\subseteq B\setminus T}}f(T).
\]
For $S\in\binom B{r-1}$, the calculation in Lemma~\ref{lem:pairing} gives
\[
\lambda_Y(U_{r-1}\delta_S)=d\,[\,Y\subseteq B\setminus S\,].
\]
Thus $\lambda_Y$ descends to a homomorphism
$\coker U_{r-1}\to\Z/d$. On the indicated element it takes the value
\[
\lambda_Y(g)=\binom{|B_0|-|Y|}{m-1-|Y|}
=\binom{q+d-1}{d-1}\equiv1\pmod p.
\]
For the congruence, $d=p^a$ divides $q$, so the lowest $a$ base-$p$
digits of $q+d-1$ and $d-1$ are all $p-1$, while all higher digits
of $d-1$ are zero; Lucas' theorem applies. This value is a unit
modulo $d$, and therefore $[g]$ has order divisible by $d$.
Theorem~\ref{thm:theta-injective} preserves this order under $\Phi$.
Doing this for each prime divisor of $q$ proves that the image has
order exactly $q$.

When $m=q=p$, Theorem~\ref{thm:exacttors} gives exactly one
$p$-primary cyclic factor, of order $p$, in both source and target.
The nonzero induced map is consequently an isomorphism.
\end{proof}

In the prime case $m=q=p$, the surviving class $\iota_{r*}[z_p^{(p)}]$ should be distinguished
from the class $[z_p^{(r)}]$ defined using all vertices of the target.
The former is the sum of the restricted theta classes supported on
$(p-1)$-subsets of $B_0$, whereas the latter sums over all such subsets
of $B$. For example, on $K_{4,6}$ the class $[z_3^{(4)}]$ has zero
$3$-primary component, since $\partial w_3=4z_3^{(4)}$, while the
image of $[z_3^{(3)}]$ from $K_{4,5}$ has order $3$.

\section{Pendant edges, loops and multiple edges}\label{sec:pendant}

We study how attaching pendant edges, loops and parallel edges changes
configuration-space homology. We begin with an integral splitting for
pendant attachment, which gives torsion-preservation results and serves
as the starting point for the loop calculation. We then consider
parallel edges and their effect on torsion.

\subsection{Pendant edges}\label{sec:pendant-edges}

Let $\Gamma$ be a finite graph, let $v$ be a nonisolated vertex with $d=d(v)\ge1$, and
let $\Gamma'=\Gamma\cup\ell$, where $\ell$ joins $v$ to a new univalent vertex.  Let
$\Gamma_v$ be the vertex explosion: $v$ is replaced by $d$ univalent vertices.  Then
$E(\Gamma_v)=E(\Gamma)$ and $\Sw(\Gamma_v)=\Z[E]\otimes\bigotimes_{x\ne v}\Sw(x)$.  Let
$\iota\colon\Sw(\Gamma_v)\to\Sw(\Gamma)$ be the inclusion of the summand with $v$ in state
$\varnothing$.

\begin{lemma}[Integral pendant-edge splitting]\label{lem:pendant}
There is a chain homotopy equivalence over the ring $\Z[E(\Gamma)]$
\[
\Sw(\Gamma')\;\simeq\;\Sw(\Gamma)\ \oplus\ \bigoplus_{j\ge1}\Sw(\Gamma_v)[1,j+1]^{\oplus(d-1)}.
\]
Consequently, for every $i,n\ge0$,
\begin{equation}\label{eq:pendantsplit}
H_i(\B_n\Gamma')\cong H_i(\B_n\Gamma)\oplus\bigoplus_{k=0}^{n-2}H_{i-1}(\B_k\Gamma_v)^{\oplus(d-1)}.
\end{equation}
The inclusion $\Gamma\hookrightarrow\Gamma'$ realizes the first direct summand.  These
statements remain valid with coefficients in any commutative ring.
\end{lemma}

\begin{proof}
Put $T=\Sw(\Gamma_v)$, $S=\Sw(\Gamma)$, and retain the special half-edge $h_v$ at $v$.
For every other old half-edge $h$ at $v$, write $\xi_h=h-h_v$ and $a_h=e(h)-e(h_v)$, so
that $\partial\xi_h=a_h$.  Take $g=h_\ell-h_v$ and $u=\ell-e(h_v)$.  Then $\partial g=u$
and
\[
\Sw(\Gamma')=S[u]\oplus g\,T[u],\qquad\partial(gy)=u\,\iota(y)-g\,\partial y.
\]
Give $u^aX$ auxiliary degree $a$ and $u^agY$ auxiliary degree $a+1$.  The differential
preserves this degree; degree zero is precisely $S$.  For each $j\ge1$ the remaining
piece is
\[
C_j=u^jT\ \oplus\ \bigoplus_{h\ne h_v}u^j\xi_hT\ \oplus\ u^{j-1}gT .
\]
Make the triangular change of generators
\[
z_{h,j}=u^j\xi_h-a_h\,u^{j-1}g .
\]
This is an invertible graded $\Z[E]$-linear change of basis.  The generator $z_{h,j}$ has
bidegree $(1,j+1)$ and satisfies $\partial z_{h,j}=0$, so $z_{h,j}T\cong T[1,j+1]$ as
chain complexes.  The remaining subcomplex $u^jT\oplus u^{j-1}gT$ is a shifted cone of
the identity on $T$, hence contractible.  Explicitly, a contraction sends $u^jy$ to
$u^{j-1}gy$ and sends $u^{j-1}gy$ to zero.  Thus
\[
C_j\cong\cone(\mathrm{id}_T)[0,j]\oplus T[1,j+1]^{\oplus(d-1)} .
\]
Discarding the contractible pieces proves the chain homotopy equivalence.  In total
weight $n$ the second summand has weight $k=n-j-1$ in $T$, which gives
\eqref{eq:pendantsplit}.  The construction also identifies the original complex with the
degree-zero summand, and works over every coefficient ring without a flatness assumption.
\end{proof}

\begin{corollary}[All primes and all degrees]\label{cor:pendantall}
For every prime $p$,
\[
\Tors_pH_i(\B_n\Gamma')\cong\Tors_pH_i(\B_n\Gamma)\oplus\bigoplus_{k=0}^{n-2}\Tors_pH_{i-1}(\B_k\Gamma_v)^{\oplus(d-1)}.
\]
Thus pendant attachment never kills torsion.  For $d\ge2$ it preserves the entire
$p$-primary subgroup in degree $i$, weight $n$, if and only if $H_{i-1}(\B_k\Gamma_v)$ has
no $p$-torsion for $0\le k\le n-2$.  For $d=1$ it preserves the whole homology in every
degree and weight.
\end{corollary}

\begin{proof}
Take the torsion subgroups in \eqref{eq:pendantsplit}.  The criterion follows since all
these torsion groups are finite.
\end{proof}

\begin{theorem}\label{thm:pendant}
Attaching a pendant edge at a nonisolated vertex preserves odd torsion in $H_2$: for
every odd prime $p$ and every $n$,
\[
\Tors_pH_2(\B_n\Gamma')\cong\Tors_pH_2(\B_n\Gamma).
\]
\end{theorem}

\begin{proof}
Every extra torsion summand in Corollary~\ref{cor:pendantall} is $p$-torsion in an
$H_1(\B_k\Gamma_v)$, which is zero for odd $p$ by Theorem~\ref{thm:KP}.  Repeating the
argument treats any finite sequence of such attachments.
\end{proof}

\begin{theorem}[One degree below the particle number]\label{cor:topminusone}
Let $m\ge2$.  Attaching a pendant edge at any vertex of a finite graph preserves the
full torsion subgroup of $H_{m-1}(\B_m\Gamma;\Z)$, including the $2$-primary part.  At a
nonisolated vertex the stronger formula is
\[
H_{m-1}(\B_m\Gamma')\cong H_{m-1}(\B_m\Gamma)\oplus H_{m-2}(\B_{m-2}\Gamma_v)^{\oplus(d-1)},
\]
and the added summand is free abelian.  Moreover, if $J$ is a disjoint union of arcs and
points, then
\[
\Tors H_{m-1}(\B_m(\Gamma\sqcup J))\cong\Tors H_{m-1}(\B_m\Gamma).
\]
\end{theorem}

\begin{proof}
In \eqref{eq:pendantsplit}, with $i=m-1$ and $n=m$, the groups at $k<m-2$ vanish, and
the group at $k=m-2$ is top homology, hence free by Observation~\ref{obs:top}.

Every component of every configuration space of $J$ is contractible.  The K\"unneth
decomposition therefore expresses $H_{m-1}(\B_m(\Gamma\sqcup J))$ as a sum of copies of
$H_{m-1}(\B_a\Gamma)$, with $a\le m$.  The term $a=m$ occurs once, the terms $a=m-1$ are
free top homology, and the terms $a<m-1$ vanish.  This proves the second assertion.  If
the attaching vertex is isolated, attaching the first edge replaces an isolated point
component by an arc component; the same argument shows that this also preserves the
indicated torsion subgroup.
\end{proof}

\begin{example}[Why the degree and prime restrictions matter]\label{ex:pendantfail}
Let $\Delta$ be a graph, attach a path $x-v-w$ at $x\in\Delta$, and call the resulting
graph $\Gamma$.  Attach a new pendant edge at the bivalent vertex $v$, obtaining
$\Gamma'$.  Then $\Gamma_v\cong(\Delta\cup\text{pendant at }x)\sqcup I$, where $I$ is an
arc.  With $\Delta=K_{3,3}$, Corollary~\ref{cor:pendantall} and Theorem~\ref{cor:topminusone}
give
\[
\Tors_2H_2(\B_4\Gamma')\cong\Tors_2H_2(\B_4\Gamma)\oplus\Z/2 .
\]
Only the extra term $H_1(\B_2\Gamma_v)$ has torsion, and its torsion is $\Z/2$ by
Corollary~\ref{cor:r2} and Theorem~\ref{cor:topminusone}. With $\Delta=K_{4,5}$ instead, the same splittings
give
\[
\Tors_3H_3(\B_5\Gamma')\cong\Tors_3H_3(\B_5\Gamma)\oplus\Z/3 .
\]
Here the only extra term with $3$-torsion is $H_2(\B_3\Gamma_v)$, whose $3$-primary part
is $\Z/3$ by Corollary~\ref{cor:exact}.  Thus neither all-prime invariance in $H_2$ for
arbitrary particle number, nor odd-torsion invariance in arbitrary degree, follows from
Theorem~\ref{thm:pendant}; both unrestricted assertions are false.
\end{example}

\subsection{Loops}
Let $\Gamma''=\Gamma\cup e_0$, where $e_0$ is a loop at $v$, i.e.\ a new edge
with both ends at $v$.  Up to homeomorphism, $\Gamma''$ is obtained by attaching a
simple cycle meeting the old graph only at $v$.  This is not the same as adding a path
between two distinct old vertices.  Keep $\Gamma'=\Gamma\cup\ell$ and $\Gamma_v$ as above.

\begin{proposition}\label{prop:loop}
There is an isomorphism of chain complexes
\[
\Sw(\Gamma'')\cong\Sw(\Gamma')\oplus c\cdot\Sw(\Gamma_v)[u],
\]
where $c$ has bidegree $(1,1)$ and $\partial(cy)=-c\,\partial y$.  Consequently, for all
$i$ and $n$,
\[
H_i(\B_n\Gamma'')\cong H_i(\B_n\Gamma')\oplus\bigoplus_{k=0}^{n-1}H_{i-1}(\B_k\Gamma_v).
\]
\end{proposition}

\begin{proof}
Let $h_1,h_2$ be the two half-edges of $e_0$ at $v$, and let $g_j=h_j-h_v$.  Then
$\partial g_1=\partial g_2=e_0-e(h_v)$, so $c=g_1-g_2$ is a cycle.  In the basis
$\{g_1,c\}$ of the new generators at $v$, write
\[
\Sw(\Gamma'')=A\oplus g_1B\oplus cB,\qquad A=\Sw(\Gamma)[u],\quad B=\Sw(\Gamma_v)[u],\quad u=e_0-e(h_v).
\]
The subcomplex $A\oplus g_1B$ is identified with $\Sw(\Gamma')$ by $e_0\mapsto\ell$ and
$g_1\mapsto h_\ell-h_v$; compare the proof of Lemma~\ref{lem:pendant}.  Since $\partial c=0$,
$\partial(cy)=-c\,\partial y$, so $cB$ is also a subcomplex.  Hence the sum is a direct sum of
complexes.

Finally, $H_i(cB)_n=\bigoplus_{j\ge0}u^jH_{i-1}(\B_{n-1-j}\Gamma_v)$.  Geometrically, one
particle runs once around the loop, $j$ particles rest on it, and the remaining
$k=n-1-j$ particles move in $\Gamma_v$.
\end{proof}

\begin{theorem}[Loop attachment at every prime]\label{cor:loop}
For every $i,n\ge0$,
\[
H_i(\B_n\Gamma'')\cong H_i(\B_n\Gamma)\oplus H_{i-1}(\B_{n-1}\Gamma_v)\oplus\bigoplus_{k=0}^{n-2}H_{i-1}(\B_k\Gamma_v)^{\oplus d}.
\]
The same direct-sum formula holds after taking $\Tors_p$, for every prime $p$.  In
particular, attaching a simple cycle at one nonisolated vertex preserves odd torsion in
$H_2$ for every particle number, but can create $2$-torsion there.  It never kills
existing torsion in any degree.  In degree $m-1$ with $m$ particles, the torsion formula
simplifies to
\[
\Tors_pH_{m-1}(\B_m\Gamma'')\cong\Tors_pH_{m-1}(\B_m\Gamma)\oplus\Tors_pH_{m-2}(\B_{m-1}\Gamma_v).
\]
\end{theorem}

\begin{proof}
Combine Proposition~\ref{prop:loop} with \eqref{eq:pendantsplit}.  For $i=2$, the extra
$H_1$ groups have only $2$-torsion by Theorem~\ref{thm:KP}.  For $i=m-1$ and $n=m$, the
terms with $k\le m-2$ are free or zero by Observation~\ref{obs:top}.
\end{proof}

Since the torsion of $H_{i-1}$ of the exploded graph appears in $H_i$, loops raise the
degree of torsion.  Call a \emph{lollipop at $x$} an edge from $x$ to a new vertex $v$,
together with a loop at $v$.

\begin{corollary}\label{cor:alldegrees}
Let $L_0=K_{4,5}$, and let $L_{j+1}$ be $L_j$ with a lollipop attached at a vertex of
$L_j$.  Then $H_{2+j}(\B_{3+j}L_j;\Z)$ contains an element of order $3$ for every
$j\ge0$.  Hence $3$-torsion occurs in every degree $\ge2$.
\end{corollary}

\begin{proof}
Let $\Gamma=L_j\cup(\text{edge }xv)$, so that $L_{j+1}=\Gamma\cup(\text{loop at }v)$ and
$\Gamma_v\cong L_j\cup(\text{pendant at }x)$.  By Proposition~\ref{prop:loop} with $k=n-1$,
\[
H_{3+j}(\B_{4+j}L_{j+1})\supseteq H_{2+j}(\B_{3+j}(L_j\cup\text{pendant})).
\]
By Lemma~\ref{lem:pendant}, $H_i(\B_nL_j)$ is a direct summand of
$H_i(\B_n(L_j\cup\text{pendant}))$, so an element of order $3$ survives.  The case $j=0$
is Theorem~\ref{thm:explicitgen}.
\end{proof}

\subsection{Multiple edges}
Let $e$ be an edge of $\Gamma$ with distinct endpoints $u,w$, and let $\Gamma^+=\Gamma\cup e'$,
where $e'$ is a new edge parallel to $e$.  Let $\Gamma_u$, $\Gamma_w$ and $\Gamma_{uw}$ be the
vertex explosions of $\Gamma$ at $u$, at $w$, and at both.  In $\Gamma_{uw}$ the edge $e$
becomes an isolated arc.  Let
\[
S=\Sw(\Gamma),\qquad S_u=\iota\,\Sw(\Gamma_u),\qquad S_w=\iota\,\Sw(\Gamma_w),
\]
so $S_u$ and $S_w$ are the subcomplexes in which $u$, respectively $w$, is in state
$\varnothing$, and $S_u\cap S_w=\iota\,\Sw(\Gamma_{uw})$.

\begin{theorem}\label{thm:parallel}
There is a decomposition of chain complexes $\Sw(\Gamma^+)=\bigoplus_{j\ge0}C_j$ with
\[
C_0=S,\qquad C_1=\cone\bigl(S_u\oplus S_w\to S\bigr)[0,1],\qquad
C_j\simeq\bigl(S/(S_u+S_w)\bigr)[0,j]\ \ (j\ge2),
\]
where $[0,j]$ shifts the weight by $j$ and $\simeq$ denotes a quasi-isomorphism.
Consequently, for all $i$ and $n$,
\[
H_i(\B_n\Gamma^+)\cong H_i(\B_n\Gamma)\ \oplus\ H_i(C_1)_n\ \oplus\ \bigoplus_{k=0}^{n-4}H_{i-2}(\B_k\Gamma_{uw})^{\oplus(d(u)-1)(d(w)-1)},
\]
and $C_1$ sits in a long exact sequence
\[
\begin{aligned}
\cdots&\longrightarrow H_i(\B_{n-1}\Gamma_u)\oplus H_i(\B_{n-1}\Gamma_w)
 \longrightarrow H_i(\B_{n-1}\Gamma)\\
&\longrightarrow H_i(C_1)_n
 \longrightarrow H_{i-1}(\B_{n-1}\Gamma_u)\oplus H_{i-1}(\B_{n-1}\Gamma_w)
 \longrightarrow\cdots .
\end{aligned}
\]
\end{theorem}

\begin{proof}
\emph{Choice of generators.}  Take the special half-edges $h_u$ and $h_w$ to be the two
half-edges of $e$.  The new half-edges $h'_u,h'_w$ of $e'$ give generators
$g_u=h'_u-h_u$ and $g_w=h'_w-h_w$ with
\[
\partial g_u=\partial g_w=e'-e=:t .
\]
Then $\Z[E][e']=\Z[E][t]$, and every cell of $\Sw(\Gamma^+)$ has a unique form
\[
t^aX,\qquad t^ag_uY,\qquad t^ag_wZ,\qquad t^ag_ug_wW,
\]
with $X\in S$, $Y\in S_u$, $Z\in S_w$ and $W\in S_u\cap S_w$.

\emph{The grading.}  Define the $t$-degree of a cell as $a$ plus the number of $g$'s.
The differential preserves it: $\partial g_u=\partial g_w=t$, and all other differentials
do not involve $t$.  Hence $\Sw(\Gamma^+)=\bigoplus_jC_j$, where $C_j$ is spanned by cells
of $t$-degree $j$.

\emph{Identifying the pieces.}
\begin{itemize}[leftmargin=*]
\item $C_0=S$.
\item $C_1=tS\oplus g_uS_u\oplus g_wS_w$, with
      $\partial(g_uY)=tY-g_u\partial Y$ and
      $\partial(g_wZ)=tZ-g_w\partial Z$, identifies with
      $\operatorname{cone}(f)[0,1]$, where
      $f\colon S_u\oplus S_w\to S$, $f(Y,Z)=Y+Z$, is the sum of the two
      inclusions.
\item For $j\ge2$, $C_j=t^jS\oplus t^{j-1}g_uS_u\oplus t^{j-1}g_wS_w\oplus t^{j-2}g_ug_w(S_u\cap S_w)$
      is the total complex of the commutative square of inclusions
      $S_u\cap S_w\to S_u,S_w\to S$.  Since these are inclusions of subcomplexes with
      $S_u\cap S_w$ as intersection, the total complex is quasi-isomorphic to
      $S/(S_u+S_w)$.
\end{itemize}

\emph{The quotient.}  The quotient $S/(S_u+S_w)$ is spanned by the cells in which both $u$
and $w$ carry generators.  The differential of those generators lands in $S_u+S_w$, so
\[
S/(S_u+S_w)\cong\bigoplus_{h\ne h_u,\ h'\ne h_w}\Sw(\Gamma_{uw})[2,2].
\]

\emph{Weights.}  In weight $n$, the piece $C_j$ involves $\Sw(\Gamma_{uw})$ in weight
$n-j-2$, so only $j\le n-2$, i.e.\ $k=n-j-2\ge0$, contributes.  The piece $C_1$ involves
$S$, $S_u$, $S_w$ in weight $n-1$.  This gives the stated formula and the long exact
sequence.
\end{proof}

\begin{theorem}[Parallel edges in top weight]\label{prop:parallel-top}
Let $\Gamma^+=\Gamma\cup e'$ be obtained by adding an edge parallel to
$e=uw$, with $u\ne w$.  For every $m\ge1$ and every commutative ring $R$,
\[
H_m(\B_m\Gamma^+;R)
\cong
H_m(\B_m\Gamma;R)
\oplus
H_{m-1}(\B_{m-1}\Gamma_{uw};R).
\]
Consequently,
\[
\Tors H_{m-1}(\B_m\Gamma^+;\Z)
\cong
\Tors H_{m-1}(\B_m\Gamma;\Z)
\oplus
\Tors H_{m-2}(\B_{m-1}\Gamma_{uw};\Z).
\]
The same statement holds primary-componentwise at every prime.
\end{theorem}

\begin{proof}
Use the notation of Theorem~\ref{thm:parallel}.  In degree $m$ and weight $m$,
the terms $C_{j\ge2}$ do not contribute: their homology groups are of the form
\[
H_{m-2}(\B_k\Gamma_{uw}),\qquad k\le m-4,
\]
and vanish because the homological degree exceeds the particle number.

The long exact sequence for $C_1$ therefore gives
\[
H_m(C_1)_m\cong
\ker\!\left(
H_{m-1}(\B_{m-1}\Gamma_u;R)\oplus
H_{m-1}(\B_{m-1}\Gamma_w;R)
\longrightarrow
H_{m-1}(\B_{m-1}\Gamma;R)
\right),
\]
since $H_m(\B_{m-1}-;R)=0$.

Put $T=S_u\cap S_w$ and $U=S_u+S_w$.  In weight $m-1$, the exact sequence
\[
0\longrightarrow T\longrightarrow S_u\oplus S_w\longrightarrow U\longrightarrow0
\]
shows that
\[
H_{m-1}(T)_{m-1}\longrightarrow
H_{m-1}(S_u\oplus S_w)_{m-1}
\]
is injective, because $H_m(U)_{m-1}=0$.  Moreover
\[
S/U\cong T[2,2]^{\oplus(d(u)-1)(d(w)-1)}
\]
and hence
\[
H_m(S/U)_{m-1}
\cong
H_{m-2}(T)_{m-3}^{\oplus(d(u)-1)(d(w)-1)}
=0.
\]
Thus $H_{m-1}(U)_{m-1}\to H_{m-1}(S)_{m-1}$ is injective, and the displayed kernel
is precisely
\[
H_{m-1}(T)_{m-1}\cong H_{m-1}(\B_{m-1}\Gamma_{uw};R).
\]
Together with $C_0=S$, this proves the first formula.

For the torsion statement, apply the first formula with
$R=\Z/\ell^a$ for every prime $\ell$ and every $a\ge1$.  The integral top homology
groups
\[
H_m(\B_m-;\Z),\qquad H_{m-1}(\B_{m-1}-;\Z)
\]
are free.  Comparing the split universal coefficient sequences and cancelling the
common free tensor summands gives
\[
\Tor\!\left(H_{m-1}(\B_m\Gamma^+;\Z),R\right)
\cong
\Tor\!\left(H_{m-1}(\B_m\Gamma;\Z),R\right)
\oplus
\Tor\!\left(H_{m-2}(\B_{m-1}\Gamma_{uw};\Z),R\right).
\]
Comparison for all $\ell$ and $a$ determines the elementary divisors and gives the
claimed torsion splitting.
\end{proof}

The index $j$ is an auxiliary algebraic grading: the variable $t=e'-e$ is a difference
of edge variables, not a count of particles on $e'$.  All following arguments use the
mapping-cone description of $C_1$.

\begin{corollary}\label{cor:parallel}
\begin{enumerate}[leftmargin=*]
\item $H_*(\B_n\Gamma)$ is a direct summand of $H_*(\B_n\Gamma^+)$.  Hence adding a parallel
      edge never kills torsion, and deleting one of several parallel edges never
      creates torsion.  This holds for every prime and every degree.
\item For every odd prime $p$,
\[
\Tors_pH_2(\B_n\Gamma^+)\cong\Tors_pH_2(\B_n\Gamma)\oplus\Tors_p\coker\Bigl(H_2(\B_{n-1}\Gamma_u)\oplus H_2(\B_{n-1}\Gamma_w)\to H_2(\B_{n-1}\Gamma)\Bigr).
\]
\item Suppose $d(u),d(w)\ge2$.  Then, in degrees $i\ge3$, $\Tors H_{i-2}(\B_k\Gamma_{uw})$ is a
      direct summand of $\Tors H_i(\B_n\Gamma^+)$ for each $k\le n-4$, with multiplicity
      $(d(u)-1)(d(w)-1)$.  In particular, if $H_2(\B_k\Gamma_{uw})$ has odd torsion for
      some $k\le n-4$, then adding $e'$ produces new odd torsion summands in
      $H_4(\B_n\Gamma^+)$.
\end{enumerate}
\end{corollary}

\begin{proof}
Part (1) is immediate from Theorem~\ref{thm:parallel}.  For (2): the summands coming
from $C_{j\ge2}$ contribute $H_0$'s to $H_2$, which are free.  In the long exact sequence
for $C_1$, the term following $H_2(C_1)$ is a subgroup of $H_1\oplus H_1$, which has no
odd torsion by Theorem~\ref{thm:KP}.  So after localizing at $p$ the sequence splits.
Part (3) is read off from Theorem~\ref{thm:parallel}.
\end{proof}

\begin{proposition}[The obstruction in top particle weight]\label{prop:Dn}
With the notation of Theorem~\ref{thm:parallel}, put
\[
T=S_u\cap S_w,\qquad U=S_u+S_w,\qquad
\mu=(d(u)-1)(d(w)-1).
\]
For $n\ge0$ define
\[
D_n=
\coker\!\left(
H_n(\B_n\Gamma_u)\oplus H_n(\B_n\Gamma_w)
\longrightarrow H_n(\B_n\Gamma)
\right).
\]
Then there is a short exact sequence
\[
0\longrightarrow E_n\longrightarrow D_n\longrightarrow L_n\longrightarrow0,
\]
where
\[
E_n=
\ker\!\left(
H_{n-1}(\B_n\Gamma_{uw})
\longrightarrow
H_{n-1}(\B_n\Gamma_u)\oplus H_{n-1}(\B_n\Gamma_w)
\right)
\]
and
\[
L_n\subseteq
H_{n-2}(\B_{n-2}\Gamma_{uw})^{\oplus\mu}.
\]
In particular $L_n$ is free abelian and
\[
\Tors D_n\cong\Tors E_n.
\]
\end{proposition}

\begin{proof}
The exact sequence
\[
0\longrightarrow T\xrightarrow{(1,-1)}
S_u\oplus S_w\longrightarrow U\longrightarrow0
\]
gives, in weight $n$,
\[
H_n(U)_n\big/
\operatorname{im}\!\left(
H_n(S_u)_n\oplus H_n(S_w)_n
\right)
\cong E_n.
\]
Indeed, this quotient is the image of the connecting map to
$H_{n-1}(T)_n$, hence the displayed kernel.

Next write
\[
Q=S/U\cong T[2,2]^{\oplus\mu}.
\]
Then
\[
H_{n+1}(Q)_n
\cong H_{n-1}(\B_{n-2}\Gamma_{uw})^{\oplus\mu}=0,
\]
because the homological degree exceeds the particle number.  Thus
$H_n(U)_n\to H_n(S)_n$ is injective.  Quotienting by the common image of
$H_n(S_u)_n\oplus H_n(S_w)_n$ gives
\[
0\longrightarrow E_n\longrightarrow D_n\longrightarrow L_n\longrightarrow0,
\]
where $L_n$ is the image of $D_n$ in
\[
H_n(Q)_n
\cong H_{n-2}(\B_{n-2}\Gamma_{uw})^{\oplus\mu}.
\]
The latter is top homology and hence free.  Therefore $L_n$ is free, the
short exact sequence splits as abelian groups, and
$\Tors D_n\cong\Tors E_n$.
\end{proof}

\begin{corollary}[Parallel edges with at most three particles]\label{cor:parallel3}
For every odd prime $p$ and every $0\le n\le3$,
\[
\Tors_pH_2(\B_n\Gamma^+)\cong\Tors_pH_2(\B_n\Gamma).
\]
Thus adding or deleting parallel edges does not change odd torsion in $H_2(\B_3\Gamma)$.
\end{corollary}

\begin{proof}
For $n\le2$, both groups are zero by Observation~\ref{obs:top}.  For $n=3$,
Theorem~\ref{prop:parallel-top} identifies the extra torsion with that of
$H_1(\B_2\Gamma_{uw})$, which has no odd torsion by Theorem~\ref{thm:KP}.
Iteration proves the assertion
for any number of parallel edges.
\end{proof}

For four or more particles, preservation of odd torsion in $H_2$ remains open
(Question~\ref{q:parallel}). By Corollary~\ref{cor:parallel}(2), the obstruction is
exactly the odd torsion of the cokernel appearing there.

\begin{example}[Loops and parallel edges can add torsion]\label{ex:loopfail}
Attach an edge $xv$ at a vertex of a graph $\Delta$, where $v$ is new, and call this
graph $\Gamma$.  Add a loop at the univalent vertex $v$ to obtain $\Gamma''$.  Since
$\Gamma_v\cong\Gamma$, Theorem~\ref{cor:loop} and the top-minus-one pendant invariance
give
\begin{align*}
\Delta=K_{3,3}:&\quad \Tors_2H_2(\B_3\Gamma'')\cong\Tors_2H_2(\B_3\Gamma)\oplus\Z/2,\\
\Delta=K_{4,5}:&\quad \Tors_3H_3(\B_4\Gamma'')\cong\Tors_3H_3(\B_4\Gamma)\oplus\Z/3 .
\end{align*}
In each case only the highest-weight added term can have the indicated torsion.  The
same examples arise from parallel-edge attachment: start with $\Delta\cup(x-v-w)$ and
double the edge $vw$.  The graph before attachment is homeomorphic to $\Gamma$, and the
graph afterwards to $\Gamma''$.  Thus $2$-primary invariance in
Corollary~\ref{cor:parallel3}, and invariance of odd torsion in all homological degrees,
would both be false.  All three attachment operations nevertheless preserve existing
torsion as a direct summand, by Lemma~\ref{lem:pendant}, Theorem~\ref{cor:loop}, and
Corollary~\ref{cor:parallel}(1).
\end{example}

\section{Minimality}\label{sec:minimal}

This section proves Theorem~C.  The question is whether any graph smaller than
$K_{m+1,m+q-1}$ carries an element of order $q$ in $H_{m-1}(\B_m)$, and the answer
is given for three successively larger notions of ``smaller'', each defined at the
beginning of its subsection: the subgraph relation $\subseteq$
(Section~\ref{sec:subgraphs}), the order $\preceq$ of homeomorphism types of
subgraphs of subdivisions (Section~\ref{sec:topminors}), and the full minor order
$\le$, in which edges may also be contracted and the resulting loops and parallel
edges are retained (Section~\ref{sec:minororder}); thus
$\Gamma\subseteq K\Rightarrow\Gamma\preceq K\Rightarrow\Gamma\le K$.  For subgraphs, extension by zero embeds top homology, and the mod-$q$ top
class of $K_{m+1,m+q-1}$ is nonzero on every matching, so it cannot come from a
proper subgraph.  For $\preceq$, a subgraph of a subdivision is a subgraph of
$K_{m+1,m+q-1}$ with pendant edges, arcs and points attached, and
Theorem~\ref{cor:topminusone} removes these without changing the torsion in
bidegree $(m-1,m)$.  For the minor order, the contractions of a complete bipartite
graph form a small family whose simple cores are joins $K_t+K_{s,s'}$; their top
homology is computed directly, and the loops and parallel edges created by
contraction are controlled by the splittings of Theorems~\ref{cor:loop}
and~\ref{prop:parallel-top}.  Section~\ref{sec:complete} treats $2$-torsion, where
the complete graphs $K_{2m+1}$ provide a second family of minimal graphs beside
$K_{m+1,m+1}$.

\subsection{Prime-power torsion and proper subgraphs}\label{sec:subgraphs}
For graphs $C$ and $\Gamma$ we write $C\subseteq\Gamma$ if $C$ is a subgraph of
$\Gamma$, that is, $C$ is obtained from $\Gamma$ by deleting vertices and edges, and
$C\subsetneq\Gamma$ if moreover $C\ne\Gamma$.  In this subsection we show that no
proper subgraph of $K_{m+1,m+q-1}$ has an element of order $q$ in $H_{m-1}(\B_m)$.
The tool is extension by zero, which embeds the top homology of a subgraph into that
of the ambient graph.

\begin{lemma}\label{lem:extend}
Let $\Gamma$ be a simple graph and $C\subseteq\Gamma$ a subgraph.  Extension by zero
defines an injective map
\[
H_m(\B_mC;R)\hookrightarrow H_m(\B_m\Gamma;R)
\]
for every $m\ge1$ and every commutative ring $R$.
\end{lemma}

\begin{proof}
Use Proposition~\ref{prop:Zcycles}.  Let $F$ be a matching function on $C$ satisfying
\eqref{eq:Z} for $C$, and extend it by zero to the ordered oriented $m$-matchings of
$\Gamma$.  The extension is alternating and odd under reversal.  Consider \eqref{eq:Z}
for $\Gamma$ at $(M',v)$.  If $M'\not\subset C$ or $v\notin V(C)$, every term vanishes.
Otherwise the terms with $vx\notin E(C)$ vanish, and the remaining sum runs over
$x\in N_C(v)\setminus V(M')$, which is \eqref{eq:Z} for $C$.  So the extension is a top
cycle of $\Gamma$, and the map is injective because it is injective on functions.
\end{proof}

\begin{theorem}\label{thm:subgraph}
Let $q=p^a$ with $p$ prime and $a\ge1$, let $m\ge q$, and put $N=m+q-1$ and
$K=K_{m+1,N}$.
\begin{enumerate}[leftmargin=*]
\item The $p$-primary part of $H_{m-1}(\B_mK;\Z)$ has exactly one cyclic summand of order
      $q$, and all its other cyclic summands have order less than $q$.
\item For every proper subgraph $C$ of $K$, the exponent of $H_m(\B_mC;\Z/q)$ is less
      than $q$.  Also $H_m(\B_mC;\Z)=0$.  Consequently $H_{m-1}(\B_mC;\Z)$ contains no
      element of order $q$.
\end{enumerate}
\end{theorem}

\begin{proof}
(1) By Theorem~\ref{thm:exacttors} with $r=q$, the $p$-primary part is
$\bigoplus_{i=0}^{q-1}\bigl(\Z/p^{v_p(q-i)}\bigr)^{\binom Ni-\binom N{i-1}}$.  For $i=0$ the
exponent is $a$ and the multiplicity is $1$; for $1\le i\le q-1$ we have
$v_p(q-i)=v_p(i)\le a-1$.

(2) Let $G=H_m(\B_mK;\Z/q)$.  By Theorem~\ref{thm:top}, Lemma~\ref{lem:kerU}(1) and the
proof of Theorem~\ref{thm:exacttors},
\[
G\cong\ker(U_{q-1}\otimes\Z/q)\cong\Tor(\coker U_{q-1},\Z/q)\cong\Z/q\oplus H,
\]
where, since $\Tor(\Z/d,\Z/q)\cong\Z/\gcd(d,q)$, the group $H$ is a direct sum of cyclic
groups of order $p^{v_p(i)}$ with $1\le i\le q-1$; hence $p^{a-1}H=0$.

The constant function $\mathbf 1$ lies in $G$ (Lemma~\ref{lem:kerU}(2)) and has order
$q$, since it is a vector with entries $1$ in $(\Z/q)^{\binom{[N]}{q-1}}$.  Write
$\mathbf 1=(u,h)$ in $\Z/q\oplus H$; then $p^{a-1}\mathbf 1=(p^{a-1}u,0)\ne0$, so $u$ is a
unit modulo $p$.  Now let $x=(u',h')\in G$ have order $q$.  Again $u'$ is a unit, and
$y:=x-u'u^{-1}\mathbf 1\in H$ satisfies $p^{a-1}y=0$.  Thus
\begin{equation}\label{eq:orderq}
x=c\,\mathbf 1+y,\qquad c\in(\Z/q)^\times,\quad p^{a-1}y=0 .
\end{equation}
In terms of matching functions, $\mathbf 1$ is $F_0(M)=(-1)^{p(\alpha(M))}$, and $y(M)$
lies in $p\Z/q$ for every $m$-matching $M$, because $p^{a-1}y(M)=0$ in $\Z/q$.  Hence
$x(M)\equiv\pm c\not\equiv0\pmod p$: \emph{every element of order $q$ in $G$ is nonzero
on every $m$-matching of $K$.}

Let $C\subsetneq K$ and let $e$ be an edge of $K$ not in $C$.  The edge $e$ lies in
some $m$-matching $M$ of $K$: match its $A$-endpoint to its $B$-endpoint and $m-1$ of
the remaining $m$ vertices of $A$ into the remaining $m+q-2\ge m-1$ vertices of $B$.
By Lemma~\ref{lem:extend}, $H_m(\B_mC;\Z/q)$ embeds in $G$, and every element in its
image vanishes on $M$.  So the image contains no element of order $q$, and
$H_m(\B_mC;\Z/q)$ has exponent less than $q$.

Finally, $H_m(\B_mC;\Z)$ embeds in $H_m(\B_mK;\Z)=0$ by Lemma~\ref{lem:extend}, so the
universal coefficient theorem gives $H_m(\B_mC;\Z/q)\cong\Tor(H_{m-1}(\B_mC;\Z),\Z/q)$.
Since $\Tor(\Z/p^j,\Z/q)\cong\Z/p^{\min(j,a)}$, the exponent of this group is less than
$q$ exactly when $H_{m-1}(\B_mC;\Z)$ has no $p$-primary cyclic summand of order $\ge q$,
i.e.\ no element of order $q$.
\end{proof}

\begin{corollary}\label{cor:subgraphmin}
For every prime power $q$ and every $m\ge q$, the graph $K_{m+1,m+q-1}$ is minimal among
its subgraphs with respect to elements of order $q$ in $H_{m-1}(\B_m)$: it has exactly
one $p$-primary cyclic summand of order $q$, and no proper subgraph has an element of
order $q$.
\end{corollary}

\begin{remark}\label{rem:primepower}
\begin{enumerate}[leftmargin=*]
\item For $a=1$ the group $H$ in the proof is zero, so $H_m(\B_mK;\F_p)\cong\F_p$.  This
      reproves Lemma~\ref{lem:kerU}(3) from Wilson's diagonal form instead of his rank
      formula.
\item The restriction to prime powers is necessary.  For $q=6$, the proper subgraph
      $K_{m+1,m+4}$ of $K_{m+1,m+5}$ already has cyclic summands $\Z/2$ and $\Z/3$ in
      $H_{m-1}(\B_m)$ (Theorem~\ref{thm:exacttors} with $r=5$), hence elements of order
      $6$ and a direct summand $\Z/6\cong\Z/2\oplus\Z/3$; see
      Corollaries~\ref{cor:orderk} and~\ref{cor:orderk-minimal}.
\end{enumerate}
\end{remark}

\subsection{Topological minors}\label{sec:topminors}
For graphs $\Gamma$ and $K$, write $\Gamma\preceq K$ if $\Gamma$ is homeomorphic to a
subgraph of a subdivision of $K$, and $\Gamma\prec K$ if moreover $\Gamma$ is not
homeomorphic to $K$.  Since graph braid group homology depends only on the
homeomorphism type, this is the natural order for our purposes; it contains the
subgraph relation $\subseteq$ of Section~\ref{sec:subgraphs}, and
Remark~\ref{rem:topminor} compares it with the usual topological minor order.  The
passage from $\subseteq$ to $\preceq$ rests on the shape of a subgraph of a
subdivision (Lemma~\ref{lem:shape}) and on the pendant-edge invariance of
Theorem~\ref{cor:topminusone}.  We begin with a K\"unneth observation used for
disconnected graphs.

\begin{lemma}\label{lem:kunneth}
Let $p$ be an odd prime and $\Gamma=\Gamma_1\sqcup\Gamma_2$.  Then $H_2(\B_n\Gamma)$ has
$p$-torsion if and only if some $H_2(\B_a\Gamma_i)$ with $a\le n$ has $p$-torsion and
$\UConf_{n-a}(\Gamma_{3-i})\ne\varnothing$.
\end{lemma}

\begin{proof}
Apply the K\"unneth formula to
$\UConf_n\Gamma=\bigsqcup_{a+b=n}\UConf_a\Gamma_1\times\UConf_b\Gamma_2$.  The terms
$\Tor(H_0,H_1)$ vanish because $H_0$ is free, and $H_1\otimes H_1$ has no odd torsion by
Theorem~\ref{thm:KP}.
\end{proof}

\begin{lemma}\label{lem:shape}
Let $\Gamma\preceq K$.  Then $\Gamma$ is homeomorphic to a disjoint union
$(\Lambda\cup P)\sqcup J$, where:
\begin{itemize}[leftmargin=*]
\item $\Lambda$ is a subgraph of $K$, obtained by deleting vertices and edges;
\item $P$ is a set of pendant edges attached at vertices of $\Lambda$;
\item $J$ is a disjoint union of arcs and points.
\end{itemize}
If $\Gamma\prec K$, then $\Lambda\ne K$.
\end{lemma}

\begin{proof}
Let $\Gamma$ be a subgraph of a subdivision $K'$ of $K$.  Each edge $f=xy$ of $K$
becomes a path $\pi_f$ in $K'$, and $\Gamma\cap\pi_f$ is one of the following:
\begin{itemize}[leftmargin=*]
\item all of $\pi_f$ (put $f$ into $\Lambda$);
\item a union of subpaths, where a subpath containing $x$ or $y$ is a pendant edge at
      that vertex, and every other subpath is an arc or a point in $J$.
\end{itemize}
The vertices of $K$ lying in $\Gamma$ are the vertices of $\Lambda$.  If $\Lambda=K$, then every
path $\pi_f$ is fully contained in $\Gamma$, so $P=J=\varnothing$ and $\Gamma\cong K$.
\end{proof}

\begin{lemma}[Closed connected subspaces]\label{lem:closed}
Let $K$ be a graph and $\Gamma$ a connected graph.  Then $\Gamma\preceq K$ if and only if
$\Gamma$ is homeomorphic to a closed connected subspace of $K$.
\end{lemma}

\begin{proof}
A subgraph of a subdivision of $K$ is a closed subspace of the
underlying space of $K$, so the forward implication is immediate.
Conversely, let $A\subseteq K$ be a closed subspace homeomorphic to
$\Gamma$. Since $\Gamma$ is a connected finite graph, $A$ is compact
and path-connected.

After subdividing every loop edge of $K$, we may assume that each
closed edge is an embedded interval. If $A$ is a point, the conclusion
follows by making that point a vertex if necessary. If $A$ is contained
in a closed edge, compactness and connectedness imply that it is a
closed interval or a point, and subdivision at its endpoints gives the
conclusion.

Suppose now that $A$ is not contained in any closed edge. Let $e$ be a
closed edge and let $x\in A\cap e^\circ$. Choose $y\in A\setminus e$
and a path in $A$ from $x$ to $y$. At the first time this path exits
$e^\circ$, it reaches an endpoint of $e$. The part of the path before
and including that exit is a connected subset of $e$ containing $x$
and that endpoint, so it contains the whole interval between them.
Thus every point of $A\cap e^\circ$ is joined within $A\cap e$ to an
endpoint of $e$.

It follows that $A\cap e$ is empty or is a union of at most two closed
intervals, each containing an endpoint of $e$; intervals consisting of
a single endpoint are allowed. There are only finitely many edges.
Subdivide $K$ at the endpoints of these intervals that lie in edge
interiors. Then $A$ is a connected subgraph of the resulting
subdivision. Since $\Gamma$ is homeomorphic to $A$, this proves
$\Gamma\preceq K$.
\end{proof}

\begin{theorem}[Prime-power minimality for topological minors]\label{thm:topminimal}
Let $q=p^a$ with $p$ prime and $a\ge1$, let $m\ge q$, and put $K=K_{m+1,m+q-1}$.  If
$\Gamma\prec K$, then $H_{m-1}(\B_m\Gamma;\Z)$ has no element of order $q$.  Thus $K$ is
minimal in the topological minor order for elements of order $q$ in this degree and
particle number.
\end{theorem}

\begin{proof}
By Lemma~\ref{lem:shape}, write $\Gamma\cong(\Lambda\cup P)\sqcup J$, where $\Lambda\subsetneq K$,
$P$ consists of pendant edges and $J$ consists of arcs and points.
Theorem~\ref{cor:topminusone}, including its isolated-vertex case, removes all edges
of $P$ without changing the torsion in the indicated bidegree.  The same corollary
removes $J$.  Hence
\begin{equation}\label{eq:reduce}
\Tors H_{m-1}(\B_m\Gamma;\Z)\cong\Tors H_{m-1}(\B_m\Lambda;\Z).
\end{equation}
The latter group has no element of order $q$ by Theorem~\ref{thm:subgraph}.  The full
graph $K$ has such an element by Theorem~\ref{thm:explicitgen}.
\end{proof}

\begin{theorem}\label{thm:minimal}
$H_2(\B_3K_{4,5})$ contains $\Z/3$, and $H_2(\B_3\Gamma)$ has no odd torsion for every
$\Gamma\prec K_{4,5}$.
\end{theorem}

\begin{proof}
The first assertion is Corollary~\ref{cor:exact}.  For a proper topological minor use
\eqref{eq:reduce} with $m=q=3$ to reduce to a proper subgraph $\Lambda\subsetneq K_{4,5}$.
Theorem~\ref{thm:subgraph} excludes $3$-torsion.  For an odd prime $p\ge5$,
Lemma~\ref{lem:extend} embeds $H_3(\B_3\Lambda;\F_p)$ into
$H_3(\B_3K_{4,5};\F_p)\cong T_{\F_p}(5)=0$.  Also $H_3(\B_3\Lambda;\Z)=0$ by integral extension
by zero.  The universal coefficient theorem therefore excludes $p$-torsion in
$H_2(\B_3\Lambda;\Z)$ for all such $p$.
\end{proof}

\begin{remark}[Three levels of minimality for $3$-torsion]\label{rem:threelevels}
For $q=3$ the class $[z_3]\in H_2(\B_3K_{4,5})$ of order $3$ is minimal in three
independent senses.
\begin{enumerate}[leftmargin=*]
\item \emph{Homological degree.}  For every finite graph $\Gamma$ and every $n$, the
      groups $H_0(\B_n\Gamma)$ and $H_1(\B_n\Gamma)$ have no $3$-torsion: $H_0$ is free,
      and $H_1$ has only $2$-torsion by Theorem~\ref{thm:KP}.  Thus $3$-torsion first
      occurs in degree $2$, for any graph whatsoever.
\item \emph{Particle number.}  In degree $2$, torsion of any kind requires at least
      three particles (Observation~\ref{obs:top}).
\item \emph{The graph.}  With three particles and in degree $2$, no $\Gamma\prec K_{4,5}$
      has $3$-torsion (Theorem~\ref{thm:minimal}), while $K_{4,5}$ does.
\end{enumerate}
For a general prime power $q$, levels (2) and (3) hold with
\[
(2,3,K_{4,5})\quad\text{replaced by}\quad(q-1,q,K_{q+1,2q-1}),
\]
by Observation~\ref{obs:top} and Theorem~\ref{thm:topminimal};
level (1), namely that an element of order $q$ never occurs in $H_i(\B_n\Gamma)$ with
$i<q-1$, is known only for $q=2,3$ (Question~\ref{q:kuratowski}(2)). 
For $m>q$, the examples in degree $m-1$ do not attain the smallest
homological degree in which this paper constructs an element of order
$q$: the graph $K_{q+1,2q-1}$ already provides such an example in degree
$q-1$. This is a comparison between examples on different graphs; it
does not assert that the corresponding torsion class survives under
inclusion into a larger graph.
\end{remark}

\subsection{The minor order}\label{sec:minororder}
For a graph $K$ and a set $F$ of edges of $K$, let $K/F$ be the multigraph obtained by
contracting every edge of $F$; loops and parallel edges may arise.  Let
$\overline{K/F}$ denote its \emph{simple reduction}, obtained by deleting loops and
replacing each nonempty parallel class by one edge.  Write
\[
\Gamma\le K\quad\Longleftrightarrow\quad
\Gamma\preceq K/F\ \text{ for some }F,
\]
and $\Gamma<K$ if moreover $\Gamma$ is not homeomorphic to $K$.  This is the third
and largest of the three relations of this section.  Since deletions and
contractions commute, and contracting part of a subdivided edge is a re-subdivision,
$\le$ is the order generated by edge and vertex deletion, edge contraction and
subdivision, up to homeomorphism.  It contains $\preceq$ (take $F=\varnothing$), hence
also $\subseteq$, and it
contains the usual simple-graph minor order, in which loops and parallel classes are
simplified after contraction, since $\overline{K/F}\preceq K/F$.  Theorem~\ref{thm:minimal}
does not apply to $\le$ directly: contraction can create loops and parallel edges, and
there is no homological comparison between $K$ and $K/F$.  In this subsection we prove
prime-power minimality of $K_{m+1,m+q-1}$ for $\le$ in the bidegree of
Theorem~\ref{thm:topminimal}, with the odd-torsion minimality of $K_{4,5}$ in $H_2(\B_3)$
as a corollary.  The ingredients are a description of the contractions of complete
bipartite graphs, a formula for the top homology of their simple reductions, and the
top-weight splittings for loops and parallel edges from Section~\ref{sec:pendant}.

The first ingredient is that contractions of complete bipartite graphs form a very
small family.  For graphs $\Gamma_1,\Gamma_2$ let $\Gamma_1+\Gamma_2$ be their join,
obtained from $\Gamma_1\sqcup\Gamma_2$ by adding all edges between $V(\Gamma_1)$ and
$V(\Gamma_2)$; thus $K_t+K_{s,s'}$ has a clique of size $t$ joined to every vertex of
$K_{s,s'}$.

\begin{lemma}[Contractions of $K_{a,b}$]\label{lem:contract}
Let $K_{a,b}$ have bipartition $A\sqcup B$, let $F$ be a set of edges, and let
$M=K_{a,b}/F$.  Let
\[
C_1,\ldots,C_t
\]
be the nontrivial branch sets, and write
\[
\alpha_i=|C_i\cap A|,\qquad \beta_i=|C_i\cap B|.
\]
Thus $\alpha_i,\beta_i\ge1$.  Put
\[
s=a-\sum_{i=1}^t\alpha_i,\qquad
s'=b-\sum_{i=1}^t\beta_i.
\]
After deleting loops and retaining one edge in each parallel class, the resulting
simple graph is
\[
L\cong K_t+K_{s,s'}.
\]
In particular
\[
t+s=a-\sum_{i=1}^t(\alpha_i-1),\qquad
t+s'=b-\sum_{i=1}^t(\beta_i-1).
\]

More precisely, if $i\ne j$, the number of edges of $M$ joining the vertices
corresponding to $C_i$ and $C_j$ is
\[
\alpha_i\beta_j+\beta_i\alpha_j.
\]
The multiplicity between $C_i$ and a singleton vertex originating in $A$ is
$\beta_i$, and that between $C_i$ and a singleton originating in $B$ is $\alpha_i$.
If $F_i=F\cap E(K_{a,b}[C_i])$, then the number of loops at the vertex corresponding
to $C_i$ is
\[
\alpha_i\beta_i-|F_i|.
\]
In particular, the existence of a loop at $C_i$ forces
$\alpha_i,\beta_i\ge2$.

Conversely, every triple
\[
t\ge1,\qquad 0\le s\le a-t,\qquad 0\le s'\le b-t
\]
occurs as the simple reduction of a contraction of $K_{a,b}$.
\end{lemma}

\begin{proof}
The vertices of $M$ are the vertex sets of the connected components of $(V,F)$.
Every nontrivial component meets both $A$ and $B$, while every other branch set is a
singleton.  Distinct branch sets are joined by all original edges running from the
$A$-part of one to the $B$-part of the other.  This gives the stated multiplicities
and shows that, after simplification, every mixed branch set is universal, the
remaining $A$-singletons and $B$-singletons span $K_{s,s'}$, and hence
$L\cong K_t+K_{s,s'}$.

Inside $C_i$ there are $\alpha_i\beta_i$ original edges, of which precisely the
edges of $F_i$ are contracted; all the others become loops.  Since $F_i$ connects
$C_i$, one has $|F_i|\ge\alpha_i+\beta_i-1$, so a loop can occur only when
$(\alpha_i-1)(\beta_i-1)>0$.

For the converse, choose disjoint subsets
\[
A=A_1\sqcup\cdots\sqcup A_t\sqcup A_0,\qquad
B=B_1\sqcup\cdots\sqcup B_t\sqcup B_0
\]
with every $A_i,B_i$ nonempty, $|A_0|=s$ and $|B_0|=s'$.  In each complete bipartite
graph $K_{A_i,B_i}$ choose a spanning tree and let $F$ be the union of these $t$
trees.  Its nontrivial branch sets are exactly $A_i\sqcup B_i$, so the simple
reduction is $K_t+K_{s,s'}$.
\end{proof}

\begin{proposition}[Top homology of the simple contraction cores]
\label{prop:join-top}
Let
\[
J=K_t+K_{s,b},\qquad d=t+s,
\]
and let $n\ge1$.
\begin{enumerate}[leftmargin=*]
\item If $d\le n$, then
\[
H_n(\B_nJ;R)=0
\]
for every commutative ring $R$.
\item If $n\ge2$ and $d=n+1$, let $D$ be the union of the $K_t$-vertices and the
$s$-part of $K_{s,b}$.  Then restriction to the spanning complete bipartite graph
$K_{D,B}\cong K_{n+1,b}$ induces
\[
H_n(\B_nJ;R)\cong H_n(\B_nK_{n+1,b};R)
\]
for every commutative ring $R$.
The right-hand side is zero for $b\le n$; if $b=n+\rho-1$ with $\rho\ge2$,
it is $\ker(U_{\rho-1}(n+\rho-1)\otimes R)$ by Theorem~\ref{thm:top}.
\end{enumerate}
\end{proposition}

\begin{proof}
Every edge of $J$ has at least one endpoint in $D$.

If $d<n$, there is no $n$-matching.  Suppose $d=n$.  Every $n$-matching uses every
vertex of $D$ and contains no edge with both endpoints in $D$.  Remove one edge
$db$ from such a matching.  In the relation \eqref{eq:Z} at the free vertex $b$,
all vertices of $D$ except $d$ are occupied, so the relation has the single term
$F(db,M')$.  Hence every top cycle vanishes.

Now assume $d=n+1$ and $n\ge2$.  If an $n$-matching $M$ contains an edge with both
endpoints in $D$, it contains exactly one such edge and the remaining $n-1$ edges
join $D$ to $B$.  Removing one of those $D$--$B$ edges and applying
\eqref{eq:Z} at its $B$-endpoint again gives a single-term relation, so every top
cycle vanishes on $M$.  Thus every top cycle is supported on the matchings of the
spanning $K_{D,B}$.  Conversely, extension by zero from $K_{D,B}$ to $J$ gives a
top cycle by Lemma~\ref{lem:extend}.  Restriction and extension by zero are inverse.
If $b<n$, the bipartite graph has no $n$-matching.  If $b=n$, removing an edge
from an $n$-matching leaves only one free $B$-vertex, so the relation
\eqref{eq:Z} at its $D$-endpoint has a single term and forces the matching value
to vanish.  For $b=n+\rho-1$ with $\rho\ge2$, apply Theorem~\ref{thm:top}.
\end{proof}

The next ingredient passes from a simple graph to all of its subgraphs at once.

\begin{lemma}[Prime-power inheritance criterion]\label{lem:inherit}
Let $\Gamma$ be a simple graph, $m\ge1$ and $q=p^a$ a prime power.  Suppose
\[
H_m(\B_m\Gamma;\Z)=0
\qquad\text{and}\qquad
\exp H_m(\B_m\Gamma;\Z/q)<q.
\]
Then $H_{m-1}(\B_mC;\Z)$ has no element of order $q$ for every subgraph
$C\subseteq\Gamma$.  For $q=p$ prime, the second hypothesis is simply
$H_m(\B_m\Gamma;\F_p)=0$.
\end{lemma}

\begin{proof}
Lemma~\ref{lem:extend} embeds both $H_m(\B_mC;\Z)$ and
$H_m(\B_mC;\Z/q)$ into the corresponding groups for $\Gamma$.  Hence the integral top
group of $C$ is zero and the exponent of its $\Z/q$ top homology is less than $q$.
The universal coefficient theorem gives
\[
H_m(\B_mC;\Z/q)\cong\Tor(H_{m-1}(\B_mC;\Z),\Z/q).
\]
An element of order $q=p^a$ in $H_{m-1}(\B_mC;\Z)$ would contribute a cyclic summand of
order $q$ to the right-hand side, a contradiction.
\end{proof}

By Proposition~\ref{prop:Zcycles}, both hypotheses of Lemma~\ref{lem:inherit} are
statements about one integer matrix, the matrix $M_\Gamma$ of the relations \eqref{eq:Z}
on the $m$-matchings of $\Gamma$: they say that $M_\Gamma$ has full column rank and
that no nonzero Smith invariant is divisible by $q=p^a$.  For $a=1$ this says simply
that $p$ divides none of the nonzero invariants.

We say that a relation \eqref{eq:Z} is a \emph{single-term relation} if the sum has
exactly one term; it then forces that value of $F$ to vanish.  This simple device
disposes of most contractions by hand.

Proposition~\ref{prop:join-top} covers the one-edge contraction of $K_{4,5}$:
\[
H_3(\B_3(K_1+K_{3,4});R)\cong H_3(\B_3K_{4,4};R)\cong R[2]
\]
for every commutative ring $R$, by Corollary~\ref{cor:r2}.  Its top homology thus
vanishes over $\Z$ and over every field of odd characteristic, while the mod-$2$ top
class is the one inherited from $K_{4,4}$.

\begin{lemma}[A torsion-free threshold for contraction multigraphs]
\label{lem:minor-threshold}
Let $M=K_{a,b}/F$, and let
\[
K_t+K_{s,s'}
\]
be its simple reduction as in Lemma~\ref{lem:contract}, where $s$ refers to the
singleton vertices originating in the $a$-part.  Fix $n\ge1$.  If
\[
t+s\le n,
\]
then for every $\Gamma\preceq M$,
\[
H_{n-1}(\B_n\Gamma;\Z)
\]
is torsion-free.
\end{lemma}

\begin{proof}
The case $n=1$ is immediate because $H_0$ is free abelian.  Assume $n\ge2$.
By the proof of Lemma~\ref{lem:shape}, applied to the multigraph $M$, one may write
\[
\Gamma\cong(\Lambda\cup P)\sqcup J,
\]
where $\Lambda\subseteq M$, $P$ is a collection of pendant edges and $J$ is a disjoint
union of arcs and points.  Theorem~\ref{cor:topminusone} removes $P$ and $J$
without changing the torsion in $H_{n-1}(\B_n)$, so it is enough to treat $\Lambda$.

We use induction first on $n$ and, for fixed $n$, on the number of loops plus the
number of edges in excess of one in the parallel classes of $\Lambda$.
This family is closed under vertex deletion: deleting vertices of $M$ amounts to
deleting their entire branch sets from $K_{a,b}$ and restricting $F$ to the
remaining vertices.  The result is a contraction of the complete bipartite graph
on the remaining original vertices (allowing an empty part).  Thus the outer
induction applies to the graphs $M-v$ and $M-\{u,w\}$ below.  When it reaches
particle number one, torsion vanishes directly in $H_0$.

If $\Lambda$ is simple, it is a subgraph of $K_t+K_{s,s'}$.  By
Proposition~\ref{prop:join-top}(1),
\[
H_n(\B_n(K_t+K_{s,s'});R)=0
\]
for every commutative ring $R$.  Extension by zero therefore gives
$H_n(\B_n\Lambda;R)=0$ for every $R$.  Taking $R=\Z/\ell$ and using the universal
coefficient theorem shows that $H_{n-1}(\B_n\Lambda;\Z)$ is torsion-free.

Suppose next that $\Lambda=\Lambda_0\cup e'$ and $e'$ is parallel to an edge $e=uw$ of $\Lambda_0$.
Theorem~\ref{prop:parallel-top} gives
\[
\Tors H_{n-1}(\B_n\Lambda)
\cong
\Tors H_{n-1}(\B_n\Lambda_0)
\oplus
\Tors H_{n-2}(\B_{n-1}(\Lambda_0)_{uw}).
\]
The first summand vanishes by the inner induction.  Exploding both $u$ and $w$
turns every edge incident to $u$ or $w$ into a pendant edge or an isolated arc.
After removing those pieces, the remaining core is a subgraph of $M-\{u,w\}$.
A parallel class in a contraction of a complete bipartite graph cannot join an
$A$-singleton to a $B$-singleton; hence at least one of $u,w$ is a mixed branch set.
Deleting $u,w$ therefore lowers $t+s$ by at least one.  The simple reduction of
$M-\{u,w\}$ consequently has parameter at most
\[
(t+s)-1\le n-1.
\]
The outer induction, together with Theorem~\ref{cor:topminusone}, shows that the
second summand also vanishes.

Now suppose that $\Lambda$ contains a loop $e_0$ at $v$.  A loop in $M$ can occur only
at a mixed branch set, by Lemma~\ref{lem:contract}.  If $v$ is nonisolated after
deleting $e_0$, Theorem~\ref{cor:loop} in the bidegree $(n-1,n)$ gives
\[
\Tors H_{n-1}(\B_n\Lambda)
\cong
\Tors H_{n-1}(\B_n(\Lambda-e_0))
\oplus
\Tors H_{n-2}(\B_{n-1}(\Lambda-e_0)_v).
\]
The first term vanishes by the inner induction.  After exploding $v$, the remaining
core is a subgraph of $M-v$, whose simple reduction has parameter
$(t+s)-1\le n-1$; the second term therefore vanishes by the outer induction.

Finally, if the loop is an isolated circle component, write $\Lambda=S^1\sqcup \Lambda'$.  Since
$\UConf_k(S^1)\simeq S^1$ for $k\ge1$, its homology is free and concentrated in
degrees $0$ and $1$.  The K\"unneth formula in total degree $n-1$ reduces possible
torsion either to the same group for $\Lambda'$ with $n$ particles, or to
$H_{n-2}(\B_{n-1}\Lambda')$.  The former vanishes by the inner induction and the latter by
the outer induction, because deleting the mixed branch set supporting the loop lowers
$t+s$ by one.  This completes the induction.
\end{proof}

\begin{theorem}[Prime-power minimality in the full minor order]
\label{thm:strongmin}
Let
\[
q=p^a,\qquad a\ge1,\qquad m\ge q,
\qquad K=K_{m+1,m+q-1}.
\]
If $\Gamma<K$ in the order generated by edge deletion, vertex deletion, edge
contraction and subdivision, up to homeomorphism, then
\[
H_{m-1}(\B_m\Gamma;\Z)
\]
has no element of order $q$.  Thus $K_{m+1,m+q-1}$ is minimal for order-$q$ torsion
in the full minor order, with loops and parallel edges retained after contraction.
\end{theorem}

\begin{proof}
The full graph $K$ has an element of order $q$ by Theorem~\ref{thm:explicitgen}.
Let $\Gamma\preceq K/F$.  If $F=\varnothing$, the assertion is
Theorem~\ref{thm:topminimal}, so assume $F\ne\varnothing$ and put
\[
M=K/F.
\]
Use Lemma~\ref{lem:contract} with the $(m+1)$-vertex part as the first bipartition
class, and write the simple reduction of $M$ as
\[
L=K_t+K_{s,s'}.
\]
Set
\[
d=t+s.
\]
Since
\[
d=m+1-\sum_{i=1}^t(\alpha_i-1),
\]
one has $d\le m+1$.

If $d\le m$, Lemma~\ref{lem:minor-threshold} gives the stronger conclusion that
$H_{m-1}(\B_m\Gamma;\Z)$ is torsion-free.  It remains to consider $d=m+1$.

The equality $d=m+1$ forces
\[
\alpha_1=\cdots=\alpha_t=1.
\]
Hence $M$ has no loops.  Moreover $t\ge1$ and
\[
s'=m+q-1-\sum_{i=1}^t\beta_i\le m+q-2.
\]
By Proposition~\ref{prop:join-top}(2),
\[
H_m(\B_mL;R)\cong H_m(\B_mK_{m+1,s'};R).
\]
If $s'\le m$, the right side is zero for every $R$.  If $s'>m$, put
$r=s'-m+1$.  Then
\[
2\le r\le q-1\le m.
\]
Theorem~\ref{thm:exacttors} therefore gives $H_m(\B_mL;\Z)=0$, and, with coefficients
in $\Z/q$, $H_m(\B_mL;\Z/q)$ is a direct sum of groups $(\Z/q)[j]\cong\Z/\gcd(j,q)$ with
$2\le j\le r<q$, hence has exponent less than $q$.  By Lemma~\ref{lem:inherit},
$H_{m-1}(\B_mC;\Z)$ has no element of order $q$ for every simple subgraph $C\subseteq L$.

Now let $\Lambda\subseteq M$ be an arbitrary subgraph, and let $C$ be obtained from $\Lambda$
by retaining one edge in each parallel class.  There are no loops.  We add the
remaining parallel edges back one at a time.  Because every mixed branch set has
$\alpha_i=1$, a parallel class of $M$ is either between two mixed branch sets or
between a mixed branch set and an $A$-singleton; a mixed branch set and a
$B$-singleton are joined by only one edge.  Hence the two endpoints $u,w$ of every
extra parallel edge both contribute to $d=t+s$.  Deleting them lowers $d$ by two.

Suppose $\Lambda_1=\Lambda_0\cup e'$ is one step in this reconstruction.  By
Theorem~\ref{prop:parallel-top},
\[
\Tors H_{m-1}(\B_m\Lambda_1)
\cong
\Tors H_{m-1}(\B_m\Lambda_0)
\oplus
\Tors H_{m-2}(\B_{m-1}(\Lambda_0)_{uw}).
\]
After exploding $u,w$ and removing the resulting pendant edges and isolated arcs,
the remaining core is a subgraph of $M-\{u,w\}$.  Its simple reduction has parameter
\[
d-2=m-1.
\]
Lemma~\ref{lem:minor-threshold}, now with $n=m-1$, shows that the second summand is
torsion-free.  Thus adding contraction-created parallel edges does not change the
torsion subgroup in the bidegree $(m-1,m)$.  It follows that $H_{m-1}(\B_m\Lambda)$ has no
element of order $q$.

Finally, the proof of Lemma~\ref{lem:shape} applies to the multigraph $M$, so
\[
\Gamma\cong(\Lambda\cup P)\sqcup J
\]
with $P$ a collection of pendant edges and $J$ a disjoint union of arcs and points.
Theorem~\ref{cor:topminusone} removes $P$ and $J$ without changing the torsion in
$H_{m-1}(\B_m)$.  The result follows.
\end{proof}

\begin{corollary}[Minimality for arbitrary finite orders]\label{cor:orderk-minimal}
Let $k\ge2$, let $q=q(k)$ be the largest prime power dividing $k$, and let $m\ge q$.
Then $K=K_{m+1,m+q-1}$ is minimal in the full minor order for elements of order $k$
in $H_{m-1}(\B_m)$: the group $H_{m-1}(\B_mK;\Z)$ has a direct summand $\Z/k$,
whereas $H_{m-1}(\B_m\Gamma;\Z)$ has no element of order $k$ for every $\Gamma<K$.
\end{corollary}

\begin{proof}
Corollary~\ref{cor:orderk} gives the direct summand $\Z/k$ for $K$.
If $x\in H_{m-1}(\B_m\Gamma;\Z)$ had order $k$ for a proper minor $\Gamma<K$,
then $(k/q)x$ would have order $q$, contradicting Theorem~\ref{thm:strongmin}.
\end{proof}

For a non-prime-power order $k$, the relevant graph parameter is $q(k)$ rather
than $k$. Thus this corollary does not remove the prime-power hypothesis from
Theorem~\ref{thm:strongmin} with its original parameterization; see
Remark~\ref{rem:primepower}(2).

\begin{corollary}\label{cor:fullminor-k45}
$K_{4,5}$ is minimal in the full minor order for odd torsion in $H_2(\B_3)$.
\end{corollary}

\begin{proof}
The full graph has $3$-torsion by Corollary~\ref{cor:exact}, and
Theorem~\ref{thm:strongmin} with $(q,m)=(3,3)$ excludes $3$-torsion from every
proper minor.  It remains to exclude torsion at odd primes $p\ge5$.

Write $\Gamma\preceq M=K_{4,5}/F$.  If $F=\varnothing$, this follows from
Theorem~\ref{thm:minimal}.  If $F\ne\varnothing$, use the notation $d=t+s$ of
the proof of Theorem~\ref{thm:strongmin}.  For $d\le3$, Lemma~\ref{lem:minor-threshold}
excludes all torsion.  For $d=4$, the simple reduction $L$ satisfies
$H_3(\B_3L;R)\cong H_3(\B_3K_{4,s'};R)$ with $s'\le4$.  This group is zero for
$s'\le3$ and is $R[2]$ for $s'=4$.  In particular it vanishes for $R=\Z$ and
$R=\F_p$ with $p$ odd.  Extension by zero and the universal coefficient theorem
therefore exclude odd torsion for every simple subgraph of $L$.  As proved in
Theorem~\ref{thm:strongmin}, restoring the parallel edges in the case $d=m+1$
does not change the torsion subgroup.  Finally Theorem~\ref{cor:topminusone}
removes the pendant edges and isolated arcs and points arising from subdivisions.
This excludes every odd prime and proves the assertion.
\end{proof}

\begin{remark}
Theorem~\ref{thm:strongmin} establishes prime-power minimality in the full
      multigraph minor order.  The key additional input beyond the simple-core
      argument is Theorem~\ref{prop:parallel-top}: in bidegree $(m-1,m)$ the
      torsion splits as a direct sum, so no nontrivial extension can create an
      element of order $q$.  Lemma~\ref{lem:minor-threshold} makes the additional
      summands vanish for the contraction-created parallel edges when $d=m+1$,
      and handles all loops and parallel edges when $d\le m$.
\end{remark}

\begin{remark}[Comparison with the usual topological minor order]\label{rem:topminor}
In the usual definition, $\Gamma$ is a topological minor of $K$ if some subdivision of
$\Gamma$ is isomorphic to a subgraph of $K$.  Our order $\preceq$ contains this order,
and it is strictly larger. 
Deleting an open subinterval whose closure is contained in the interior
of an edge $xy$ of $K$ leaves pendant stubs at both $x$ and $y$, and the
two stubs cannot in general be realized by two distinct edges of $K$.
For instance, $(K_{4,5}\setminus a_1b_1)$ with pendant edges at
$a_1$ and at $b_1$ has $11$ vertices, so no subdivision of it is a subgraph of
$K_{4,5}$, yet it is $\prec K_{4,5}$.  Minimality with respect to $\preceq$, as in
Theorems~\ref{thm:topminimal} and~\ref{thm:minimal}, is therefore the stronger
statement.  By Lemma~\ref{lem:closed}, for connected $\Gamma$ the relation
$\Gamma\preceq K$ means exactly that $\Gamma$ is homeomorphic to a closed connected
subspace of $K$; in general it means that $\Gamma$ is homeomorphic to a closed subspace
of $K$ with finitely many components.
\end{remark}

\begin{remark}\label{rem:twotorsion}
For $2$-torsion, $K_{4,5}$ is not minimal. Corollary~\ref{cor:r2} gives
\[
\Tors H_2(\B_3K_{4,4})\cong\Z/2,
\]
and $K_{4,4}$ is minimal
for $2$-torsion in $H_2(\B_3)$ by Theorem~\ref{thm:topminimal} with $(q,m)=(2,3)$.  It is
not the only minimal graph: see Section~\ref{sec:complete}.
Theorem~\ref{thm:topminimal} establishes prime-power minimality in the topological minor
order for every $m\ge q$, not just for $(q,m)=(2,2)$ or $(3,3)$.  The extra input is the
exact pendant splitting: in degree $m-1$ with $m$ particles its new terms are free for
dimensional reasons, irrespective of torsion in lower degrees at larger particle
numbers.
\end{remark}

\subsection{Two-torsion and the complete graphs}\label{sec:complete}
For $m=2$ the graphs $K_{3,3}$ and $K_5$ play symmetric roles: both have
$\Tors H_1(\B_2)\cong\Z/2$, and by Ko--Park and Kuratowski's theorem they are the minimal
graphs with $2$-torsion in $H_1(\B_2)$.  Geometrically, $\UConf_2(K_{3,3})$ and
$\UConf_2(K_5)$ deformation retract onto closed non-orientable surfaces,
of Euler characteristics $-3$ and $-5$ \cite{Abrams}, which is where $H_2(\B_2;R)\cong R[2]$ and the $\Z/2$ in $H_1$
come from.  Corollary~\ref{cor:r2} and
Theorem~\ref{thm:topminimal} generalize the first to $K_{m+1,m+1}$.  Here we generalize
the second to $K_{2m+1}$. Its minimality follows directly from the
description of its top cycles and extension by zero.

\begin{theorem}\label{thm:complete}
Let $m\ge2$.  For every commutative ring $R$,
\[
H_m(\B_mK_{2m+1};R)\cong R[2],
\]
spanned by the class of the constant matching function when $2R=0$.
Consequently,
\[
H_m(\B_mK_{2m+1};\Z)=0,\qquad
\Tors H_{m-1}(\B_mK_{2m+1};\Z)\cong\Z/2 .
\]
Moreover, if $\Gamma\prec K_{2m+1}$, then $H_{m-1}(\B_m\Gamma;\Z)$ is torsion-free.  In
particular $K_{2m+1}$ is minimal in the order $\preceq$ for $2$-torsion in
$H_{m-1}(\B_m)$, and $K_{m+1,m+1}$ and $K_{2m+1}$ are incomparable in this order.
\end{theorem}

\begin{proof}
Let $F$ be a matching function on $K_{2m+1}$ as in Proposition~\ref{prop:Zcycles}.  An
$(m-1)$-matching $M'$ leaves exactly three free vertices, which are pairwise adjacent.
Hence every relation \eqref{eq:Z} has exactly two terms, and the three relations at the
free vertices of $M'$ are those of Example~\ref{ex:C3}: for free vertices $x,y,f$,
\begin{equation}\label{eq:rotation}
F(xy,M')=F(yf,M')=F(fx,M').
\end{equation}

\emph{Step 1: $2F=0$.}  Let $M$ be an $m$-matching, and split $M=N\cup M''$ with $N$
consisting of two edges of $M$; this uses $m\ge2$.  Let $S$ be the set of five vertices
not covered by $M''$, and define $G(N')=F(N',M'')$ for every ordered oriented
$2$-matching $N'$ of the complete graph $K_5$ on $S$, with the edges of $N'$ listed
first.  For a $1$-matching $N''$ of $K_5[S]$ and a vertex $v\in S$ not covered by it,
the free neighbours of $v$ in $K_{2m+1}$ relative to $N''\cup M''$ are exactly the two
remaining vertices of $S$.  So the relations \eqref{eq:Z} of $K_{2m+1}$ at
$(N''\cup M'',v)$ are precisely the relations of $K_5$ at $(N'',v)$, and $G$ is a matching
function of $K_5$ with two particles.  By Theorem~\ref{thm:KP} in the form of
\cite[Theorem~3.16]{KP}, $H_1(\B_2K_5)\cong\Z^6\oplus\Z/2$, and Gal's formula gives
$\chi(\B_2K_5)=-5$, so $H_2(\B_2K_5;\Z)=0$.  The universal coefficient theorem yields
$H_2(\B_2K_5;R)\cong\Tor(\Z/2,R)=R[2]$, so $2G=0$ by
Proposition~\ref{prop:Zcycles}.  In particular $2F(M)=0$.

\emph{Step 2: connectivity.}  Call two $m$-matchings adjacent if one is obtained from
the other by replacing an edge $xy$ by $yf$ or $fx$, where $f$ is the vertex missed by
the first; by \eqref{eq:rotation} and Step~1, $F$ takes the same value on adjacent
matchings, whatever their orientation and order.  We claim that this graph on
$m$-matchings is connected.  Replacing $xy\in M$ by $fx$ moves the missed vertex from
$f$ to $y$, so every vertex can be made the missed one.  Given $M$ missing $f$ and two
edges $xy,zw\in M$, the three replacements $xy\to fx$, $zw\to yz$, $fx\to xw$ produce
the matching $M\setminus\{xy,zw\}\cup\{xw,yz\}$, again missing $f$.  Such exchanges
generate all perfect matchings of $K_{2m+1}\setminus f$, which proves the claim.

\emph{Step 3: the module.}  By Steps~1 and~2, $F$ is determined by the single value
$c=F(M_0)\in R[2]$ of a fixed matching $M_0$, and $F$ takes the value $c$ on every ordered
oriented $m$-matching.  Conversely, for $c\in R[2]$ the function with constant value $c$
is alternating and odd under reversal, because $c=-c$, and satisfies every relation
\eqref{eq:Z}, because each relation has two terms and $c+c=0$.  Hence
$H_m(\B_mK_{2m+1};R)\cong R[2]$.

\emph{Step 4: torsion.}  With $R=\Z$ the top homology vanishes, so the universal
coefficient theorem gives $H_m(\,\cdot\,;\Z/q)\cong\Tor(H_{m-1},\Z/q)$.  For $q=2$ this
is $\F_2$, for $q=4$ it is $\Z/2$, and for odd primes it is $0$; as in the proof of
Corollary~\ref{cor:r2}, $\Tors H_{m-1}(\B_mK_{2m+1})\cong\Z/2$.

\emph{Step 5: minimality.} Let $\Gamma\prec K_{2m+1}$. By Lemma~\ref{lem:shape},
$\Gamma\cong(\Lambda\cup P)\sqcup J$ with $\Lambda\subsetneq K_{2m+1}$.
Theorem~\ref{cor:topminusone} removes the pendant edges $P$ and the
disjoint arcs and points $J$ without changing the full torsion subgroup
in degree $m-1$ and weight $m$. It therefore suffices to treat $\Lambda$.
Choose an edge $e$ of $K_{2m+1}$ missing from $\Lambda$.
By Lemma~\ref{lem:extend}, every top matching function on $\Lambda$ extends by zero
to one on $K_{2m+1}$. By Steps~1--3 this extension is constant on all
$m$-matchings. There exists an $m$-matching containing $e$: pair any
$2m-2$ of the $2m-1$ remaining vertices. Its value under the extension
is zero, so the constant is zero. Thus $H_m(\B_m\Lambda;R)=0$ for every $R$.
The universal coefficient theorem, with $R=\F_\ell$ for each prime $\ell$,
now shows that $H_{m-1}(\B_m\Lambda;\Z)$ is torsion-free.
Finally, $K_{m+1,m+1}$ has $2m+2$ vertices of degree at least three,
whereas $K_{2m+1}$ has only $2m+1$; and $K_{2m+1}$ has vertices of degree
$2m>m+1$, the maximum degree in $K_{m+1,m+1}$. These obstructions persist
under subdivision and passage to a subgraph, so the two graphs are
incomparable in the order $\preceq$.
\end{proof}

\begin{remark}
\begin{enumerate}[leftmargin=*]
\item As in Section~\ref{sec:explicit}, the class of the constant function is
      represented by a chain $w$ which is a sum of the cubes of all $m$-matchings of
      $K_{2m+1}$ with coefficients $\pm1$.  Since $\bar w$ is a mod-$2$ cycle,
      $\partial w=2z$ for an integral $(m-1)$-chain $z$, which is a cycle, and
      $[z]=\beta[\bar w]$ generates $\Tors H_{m-1}(\B_mK_{2m+1})\cong\Z/2$ because the
      Bockstein is injective.  This is the analogue of Theorem~B for $q=2$.
\item For $m=1$ the statement fails: $K_3$ is a circle and $H_1(\B_1K_3;R)\cong R$.  The
      argument breaks down in Step~1, which needs two edges in a matching.
\item Thus, for three particles, $2$-torsion in $H_2$ has at least the two incomparable
      minimal graphs $K_{4,4}$ and $K_7$.  Unlike odd torsion, it can also be created by
      attaching a lollipop: by Theorem~\ref{cor:loop}, $\Tors_2H_2(\B_3)$ of
      $\Delta\cup(\text{lollipop at }x)$ contains $\Tors_2H_1(\B_2(\Delta\cup\text{pendant at }x))$,
      which is $\Z/2$ for $\Delta=K_{3,3}$ or $K_5$ (Example~\ref{ex:loopfail}).  These
      graphs are again incomparable with $K_{4,4}$ and $K_7$.  A complete list of the
      minimal graphs with $2$-torsion in $H_2(\B_3)$ is not known to us.
\end{enumerate}
\end{remark}

\section{Further questions}\label{sec:questions}

The preceding results leave open the lower Betti numbers, the remaining torsion,
and its behaviour under graph operations. We formulate questions arising from
these results.

\medskip\noindent\textbf{Betti numbers and polynomiality.}\enspace
The top homology and the torsion one degree below do not determine all Betti
numbers. These are known for $m\le3$ (Corollary~\ref{cor:small-cases}), but remain
open for $m\ge4$. A general formula would complement Ko--Park's calculation of
$H_1$ \cite{KP} and the results for trees \cite{FS} and graphs of circumference
at most one \cite{AnKim}; for ordered configurations on trees with loops,
see \cite{CL}.

Now fix $m$ and vary the size $N=m+r-1$ of the second part.
For a fixed set $A$ of $m+1$ vertices, injections of finite sets induce graph
embeddings, making
\[
S\longmapsto H_i(\B_m(K_{A,S});\Z)
\]
an $\mathrm{FI}$-module, where $\mathrm{FI}$ is the category of finite sets and
injections. Each basis element of the unreduced \'Swi\k{a}tkowski complex in
weight $m$ uses at most $m$ labels in $S$. The chain groups are therefore finitely
generated, and so is homology by the Noetherian theorem \cite[Theorem~A]{CEFN}.
The same argument over any field $\F$, together with \cite[Theorem~B]{CEFN}, gives
a polynomial $P_{m,i,\F}\in\Q[r]$ with
\[
\dim_{\F}H_i(\B_mK_{m+1,m+r-1};\F)
=P_{m,i,\F}(r)
\quad\text{for all sufficiently large }r.
\]

The integral torsion submodule $T_i$ is also finitely generated, so its exponent
is bounded independently of $N$. For its $p$-primary part $T_{i,p}$, the layers
$p^{a-1}T_{i,p}/p^aT_{i,p}$ are finitely generated $\mathrm{FI}$-modules over
$\F_p$. Successive differences of their dimensions count the summands of exact
order $p^a$. Thus both the free rank and each prime-power multiplicity are
eventually polynomial in $r$; the polynomials and their optimal ranges remain
to be determined.

Theorem~\ref{thm:exacttors} already gives explicit polynomials for the
top free rank and for the torsion one degree below when $r\ge m$:
\[
\rk_{\Z}H_m
=\binom{m+r-1}{m}-\binom{m+r-1}{m-1},\qquad
\Tors H_{m-1}
\cong\bigoplus_{j=2}^{m}(\Z/j)^{
\binom{m+r-1}{m-j}-\binom{m+r-1}{m-j-1}}.
\]
Here the homology groups are those of $\B_mK_{m+1,m+r-1}$.
The integral top homology first becomes nonzero at $r=m+1$, one step
after these polynomial formulas begin to apply.
Corollary~\ref{cor:small-cases} shows that the free rank one degree below
also agrees with a polynomial throughout $r\ge m$ for $m\le3$.

For comparison with field coefficients, write $b_i$ for the free rank
and $t_{i,p}$ for the total number of $p$-primary cyclic summands in
integral degree $i$, with $t_{-1,p}=0$. The universal coefficient theorem gives
\[
\dim_{\F}H_i=b_i+t_{i,p}+t_{i-1,p}
\qquad(\operatorname{char}\F=p).
\]
Thus polynomial ranges over fields also depend on torsion one degree below.
The red ranges in Table~\ref{tab:small-cases} work in every characteristic but
need not be optimal in each one. Polynomiality concerns dimensions and
multiplicities, not isomorphisms induced by adding vertices.

Theorem~\ref{rep:thm:rational} also allows both part sizes to vary.
Corollary~\ref{cor:rep-polynomial} gives an explicit polynomial for the
top rank on $a,b\ge2m$, with leading term $a^mb^m/m!$.
The analogous lower-degree information is not supplied by the rational
top representation.

\begin{question}[Betti numbers and polynomiality]\label{q:betti}
Let $m\ge4$.
\begin{enumerate}[leftmargin=*]
\item Determine the Betti numbers of $\B_mK_{m+1,m+r-1}$. The torsion
      in degree $m-1$ is known for every $r\ge0$ (Theorem~\ref{thm:exacttors}), and
      for $1\le r\le m-1$ the theta subgroup contributes free rank
      $\binom{m+r-1}{r}-\binom{m+r-1}{r-1}$ by Theorem~\ref{thm:theta-injective},
      but the remaining free classes are not determined by that presentation.
\item Is $\rk_\Z H_{m-1}(\B_mK_{m+1,m+r-1};\Z)$ given by a polynomial in $r$ for
      every $r\ge m$? More generally, for fixed $m$ and $i<m$, are the free rank and
      the prime-power torsion multiplicities of $H_i(\B_mK_{a,b};\Z)$ eventually
      polynomial in $a$ and $b$, and with what ranges of validity?
\end{enumerate}
\end{question}

\medskip\noindent\textbf{Torsion outside the vanishing range.}\enspace
Proposition~\ref{rep:prop:arbitrary} expresses top homology over every
field as a Hom space for the slot group $\Sigma_m$.
Theorem~\ref{rep:thm:generalvanishing} evaluates its dimension for odd
$p\ge m$ in the range $a,b\ge2m-1$. For $p<m$, simplex stabilizers can
have order divisible by $p$, so the chain modules used in that proof
need not be projective. The rational decomposition alone does not
resolve this obstruction.

The $2$-primary calculation of Proposition~\ref{rep:prop:two-primary-three}
suggests a corresponding problem for more particles. For $a,b\ge2m-1$,
put $R=\Z_{(2)}$ and
$L=\operatorname{sgn}\otimes_R D_m(a;R)\otimes_R D_m(b;R)$, with the
diagonal slot action. The split chessboard chain sequences give the
same base-change argument as in that proposition. Comparing the
invariants sequence of $0\to L\xrightarrow{2^\ell}L\to L/2^\ell L\to0$
with universal coefficients identifies the $2^\ell$-torsion in
$H_{m-1}(\B_mK_{a,b};\Z)$ with $H^1(\Sigma_m;L)[2^\ell]$ for every
$\ell\ge1$. Thus the $2$-primary subgroup is abstractly isomorphic to
$H^1(\Sigma_m;L)$. The usual restriction--transfer argument annihilates
this group by $m!$ (\cite[Chapter~XII]{CE}), giving the uniform exponent
bound $2^{v_2(m!)}$, where $v_2$ denotes the $2$-adic valuation.

\begin{question}[Torsion outside the vanishing range]\label{q:small-prime-vanishing}
Let $m\ge4$ and $a,b\ge2m-1$.
\begin{enumerate}[leftmargin=*]
\item For an odd prime $p<m$, does
      $\dim_{\F_p}H_m(\B_mK_{a,b};\F_p)-\rk H_m(\B_mK_{a,b};\Z)$ vanish? If not,
      describe its values in terms of the modular slot representations
      $D_m(a;\F_p)$ and $D_m(b;\F_p)$. The first case not covered by
      Theorem~\ref{rep:thm:generalvanishing} is $m=4$, $p=3$.
\item Determine the cyclic factors of the $2$-primary subgroup of
      $H_{m-1}(\B_mK_{a,b};\Z)$ from the integral slot modules. What is the
      optimal exponent bound for fixed $m$, and when do factors of order $4$ or
      higher occur? The first case is $m=4$, $a,b\ge7$.
\end{enumerate}
\end{question}

\medskip\noindent\textbf{Where torsion disappears.}\enspace
Corollary~\ref{rep:cor:allprimecounterexample} gives a source and a target
between which the entire $p$-primary torsion disappears. It does not
locate the first inclusion that kills the original class, or determine
whether intermediate graphs carry other $p$-torsion classes.

The prime-power case lies outside the present projectivity argument:
if $q=p^a$ with $a\ge2$, then $p<q$. Enlargement of the second part alone
preserves the order by Proposition~\ref{prop:second-part-persistence}.

\begin{question}[Disappearance under enlargement of the first part]\label{q:first-vanishing}
Let $p$ be an odd prime, $q=p^a$ with $a\ge1$, and let $\tau_q$ be the order-$q$
class of Theorem~\ref{thm:explicitgen} on $K_{q+1,2q-1}$. Does its image in
$H_{q-1}(\B_qK_{s,2q-1};\Z)$ vanish for some $s\ge q+1$? If so, determine the
least such $s$. At which values of $s$, if any, does the entire $p$-primary
torsion vanish? For $a=1$, Corollary~\ref{rep:cor:allprimecounterexample} gives
vanishing at $s=2p-1$; is this the first value,
and in particular does $\tau_5$ survive in $H_4(\B_5K_{7,9};\Z)$? For $a\ge2$,
even the case $s=2q-1$ is open.
\end{question}

\medskip\noindent\textbf{Tableau cycles and integral lattices.}\enspace
Theorem~\ref{rep:thm:rational} identifies the rational representation,
but does not give an integral decomposition into Specht lattices.
The loop classes of Section~\ref{sec:loopclasses} supply an explicit
polytabloid construction in one boundary case.

\begin{question}[Explicit representatives]\label{q:tableau-cycles}
Can the summands in \eqref{rep:eq:main} be realized by explicit matching
cycles indexed by pairs of standard Young tableaux? Determine the
integral lattice generated by such cycles, its index in the full top
homology, and the maps induced by inclusions of complete bipartite
graphs. Can this construction describe the failure of integral lifting
of modular top cycles, and hence the torsion one degree below?
\end{question}

\medskip\noindent\textbf{Necessity.}\enspace
Theorem~C is a minimality statement: no graph below $K_{m+1,m+q-1}$ in the minor
order has an element of order $q$ in $H_{m-1}(\B_m)$.  The Ko--Park theorem for
$p=2$ says more, namely that $2$-torsion in $H_1$ occurs \emph{only} in the presence
of a $K_5$ or $K_{3,3}$ minor \cite[Theorem~3.16]{KP}.  Whether odd torsion is
likewise detected by a forbidden minor is the central structural question left open
here. A positive answer would give a necessary minor obstruction:
a graph without a $K_{p+1,2p-1}$ minor could not have $p$-torsion in its
braid-group homology. It would not give a sufficient condition.
A negative answer would exhibit a source of $p$-torsion without this
minor, as $K_5$ is a source of $2$-torsion without a $K_{3,3}$ minor.
For $p=2$ both Kuratowski graphs contribute, and Section~\ref{sec:complete} shows
that the complete graphs $K_{2m+1}$ continue to do so in every degree.
For odd $p$ we know of no independent minimal non-bipartite source.

A separate question concerns the least possible homological degree.
For every prime $p$ and every $i\ge p-1$, Theorem~\ref{thm:exacttors} gives
$p$-torsion in degree $i$: take $m=i+1$ and $r=p$.  The first $p$-torsion of this
paper thus appears in degree $p-1$, with $p$ particles, on $K_{p+1,2p-1}$.  By the
transfer argument of the introduction, $p$ particles are necessary for $p$-torsion
if the folklore torsion-freeness conjecture for ordered configuration spaces holds
\cite[Conjecture~1.7]{HKW}; whether degree $p-1$ is necessary as well is a
different question, about which nothing is known beyond $H_0$ and $H_1$.

\begin{question}[Necessary conditions for torsion of order $q$]\label{q:kuratowski}
Let $q=p^a$ be a prime power.
\begin{enumerate}[leftmargin=*]
\item For $p$ odd, does $p$-torsion in $H_*(\B_n\Gamma)$ imply that $\Gamma$ has a
      $K_{p+1,2p-1}$ minor? The restriction to odd primes is essential, since $K_5$
      has $2$-torsion in $H_1(\B_2K_5)$ but no $K_{3,3}$ minor. Already for $p=3$
      and $H_2(\B_3)$ the question is open.
\item Can an element of order $q$ occur in $H_i(\B_n\Gamma)$ for some finite graph
      $\Gamma$, some $n$, and some $i<q-1$? For $q=2,3$ the answer is no ($H_0$ is
      free, and $H_1$ has only $2$-torsion by Theorem~\ref{thm:KP}); the first open
      cases are $q=4$ in degree $2$ and $q=5$ in degrees $2,3$.
\end{enumerate}
A negative answer to \textup{(2)} would make $K_{q+1,2q-1}$ minimal for order-$q$
torsion in homological degree, particle number, and the order $\preceq$
simultaneously (Remark~\ref{rem:threelevels}).
\end{question}

A positive answer to Question~\ref{q:kuratowski}(1) would give a necessary minor
obstruction, not upward persistence, which fails by
Corollary~\ref{rep:cor:allprimecounterexample}. The minimality theorem concerns
the specified degree and particle number; Question~\ref{q:kuratowski}(2) asks
about lower degrees.

\medskip\noindent\textbf{Parallel edges.}\enspace
Corollary~\ref{cor:parallel3} proves that parallel edges preserve odd torsion in
$H_2$ for at most three particles. For larger particle numbers,
Corollary~\ref{cor:parallel}(1) still ensures that existing torsion survives as a
direct summand, but does not rule out new odd torsion.

\begin{question}[Odd torsion and parallel edges]\label{q:parallel}
Let $\Gamma$ be a finite graph with an edge $e=uw$, $u\ne w$, and let
$\Gamma^+$ be obtained by adding an edge parallel to $e$. For every odd prime
$p$ and every $n\ge4$, does the inclusion induce an isomorphism
\[
\Tors_pH_2(\B_n\Gamma)\longrightarrow\Tors_pH_2(\B_n\Gamma^+)?
\]
Equivalently, with $\Gamma_u$ and $\Gamma_w$ denoting the vertex explosions,
is the group
\[
\coker\!\left(
H_2(\B_{n-1}\Gamma_u)\oplus H_2(\B_{n-1}\Gamma_w)
\longrightarrow H_2(\B_{n-1}\Gamma)
\right)
\]
free of odd torsion for every $n\ge4$?
\end{question}

The equivalence follows from Corollary~\ref{cor:parallel}(2). Already for $n=4$
the cokernel involves $H_2(\B_3)$, outside the top-degree range of
Proposition~\ref{prop:Dn}.

\end{document}